\documentclass[a4paper,12pt,twoside,tbtags]{article}
\usepackage{amstext,amsmath,amssymb,mathrsfs}
\usepackage{amsthm}
\usepackage{bm}
\usepackage[dvipdfmx]{graphicx}

\newcommand\Borel{\mathcal{B}}
\newcommand\C{\mathbb{C}}
\newcommand\D{\mathbb{D}} 
\newcommand\Dk{\mathbb{D}^\k}
\newcommand\N{\mathbb{N}}
\newcommand\R{\mathbb{R}}

\renewcommand\u{{\bm{u}}}
\newcommand\w{\bm{w}}
\newcommand\x{{\bm{x}}}
\newcommand\y{{\bm{y}}}
\newcommand\bN{{\bm{N}}}
\newcommand\X{{\bm{X}}}
\newcommand\Y{{\bm{Y}}}

\newcommand\n{\mathsf{n}}
\newcommand\m{\mathsf{m}}

\renewcommand\k{\mathsf{k}}

\newcommand\f{\mathfrak{f}}
\newcommand\g{\mathfrak{g}}
\newcommand\lab{\mathfrak{l}}
\newcommand\unl{\mathfrak{u}}

\newcommand\e{\varepsilon}
\renewcommand\tilde{\widetilde}
\renewcommand\hat{\widehat}
\renewcommand\#{\sharp}

\newcommand\Dc{\mathcal{D}_{\circ}}
\newcommand\Dcmu{\mathcal{D}_{\circ}^{\mu}}
\newcommand\Dck{\mathcal{D}_{\circ}^\k}
\newcommand{\Dcmuk}{\mathcal{D}_{\circ}^{\muk}}
\newcommand{\Dmu}{\mathcal{D}^{\mu}}
\newcommand{\Dmuk}{\mathcal{D}^{\mu^{[\k]}}}

\newcommand{\Dinfmuk}{\mathcal{D}_{\circ}^{\muk}}
\newcommand\Emu{{\mathcal{E}^{\mu}}}
\newcommand\Emuk{\mathcal{E}^{\mu^{[\k]}}}

\newcommand\Dqban[1]{\mathbb{D}^{\n, \m, #1}_{r,\ell, q, \eta}}

\renewcommand\P{{\mathbb P}}
\newcommand\E{{\mathbb E}}
\newcommand\muk{\mu^{[\k]}}

\newcommand\muxnm{(\check{\mu}_{\x_\k})_{r,\ell, \eta}^{\n,\m}}

\newcommand\Dq{\mathbb{D}^{\n, \m}_{r,\ell, q, \eta}}

\newcommand\nadeffsa[2]{\nabla_{#1}^{#2}}
\newcommand{\As[1]}{\textbf{(#1)}}
\newcommand{\as[1]}{\textrm{(#1)}}

\newcommand{\Capa}{{\rm Cap}}

\newcommand\1{{\bf 1}}

\newcommand\hiku[1]{\langle #1 \rangle}
\newcommand{\chesig}{\check{\sigma}}
\newcommand\Qr{U_{2^r}}
\newcommand{\M}{\mathfrak{M}}
\newcommand{\Msi}{\M_{\text{s.i.}}}
\newcommand\chian{\chi[\bd{a}_n]}
\newcommand\chia{\chi[\bd{a}]}
\newcommand\Ur{U_r}
\newcommand\Uri{U_{r(i)}}

\newcommand\Kpqn{\langle \mathfrak{K}[\bd{a}]_{p,q,n} \cap \mathfrak{I}_k\rangle }

\newcommand{\qdet}{{\rm qdet}}
\newcommand{\Ai}{{\rm Ai}}

\newcommand{\bd}[1]{\mbox{\boldmath$#1$}}

\newcommand\supp{{\rm supp}\,}

\newcounter{Const} 
\newcommand\DN{\newcommand}
\DN\Ct{\refstepcounter{Const}c_{\theConst}}
\DN\cref[1]{c_{\ref{#1}}}	

\theoremstyle{definition}
\newtheorem{theorem}{Theorem}[section]
\newtheorem{lemma}[theorem]{Lemma}
\newtheorem{proposition}[theorem]{Proposition}
\newtheorem{corollary}[theorem]{Corollary}
\newtheorem{definition}[theorem]{Definition}
\newtheorem{example}{Example}

\newtheorem{remark}{Remark}

\def\supp{{\rm supp}\,}

\newcommand\Kpqnponemone{\langle \mathfrak{K}[\bd{a}]_{p+1,q+1,n} \cap \mathfrak{I}_k^{[1]}\rangle }
\newcommand\Kpqnmone{\langle \mathfrak{K}[\bd{a}]_{p,q,n} \cap \mathfrak{I}_k^{[1]}\rangle }

\makeatletter
 
 \@addtoreset{equation}{section}
\makeatother

\title{\bf Stochastic differential equations \\ for infinite particle systems of jump type \\ with long range interactions}

\author{Syota Esaki\thanks{The Department of Applied Mathematics, 
Fukuoka University, Fukuoka 814-0180, Japan; e-mail: sesaki@fukuoka-u.ac.jp} and Hideki Tanemura\thanks{Department of Mathematics, Keio University, Yokohama 223-8522, Japan.
e-mail: tanemura@math.keio.ac.jp}}

\begin{document}

\maketitle
\begin{abstract}

Infinite-dimensional stochastic differential equations (ISDEs) describing systems with an infinite number of particles are considered.
Each particle undergoes a L\'evy process, and the interaction between particles is determined by the long-range interaction potential. The potential is of Ruelle's class or logarithmic. We discuss the existence and uniqueness of strong solutions of the ISDEs.
\end{abstract}



\section{Introduction}\label{S:1} 
In this paper, we study stochastic processes describing infinite-particle jump-type systems with long-range interactions.
The first author has previously constructed Hunt processes describing jump-type systems of unlabeled particles with interaction using Dirichlet form technique \cite{E.19}. The space of configurations of unlabeled particles is represented as the set of non-negative integer valued Radon measures:
$$
\M= \big\{ \xi = \sum_j \delta_{x^j}; \xi(K)<\infty\mbox{ for all compact sets $K \subset \R^d$} \big\},
$$
where $\delta_x$ denotes the delta measure at $x$. We endow $\M$ with the vague topology. Then $\M$ is a Polish space and $\mathfrak{N}\subset \M$ is relatively compact if and only if $\sup_{\xi\in\mathfrak{N}}\xi(K)<\infty$ for any compact set $K\subset\R^d$. 
For $x,y\in \R^d$ and $\xi \in \M$, we write $\xi^{xy}=\xi-\delta_x+\delta_y$ and $\xi\setminus x = \xi-\delta_x$ if $\xi(\{x\})\ge 1$.

Let $\mu$ be a probability measure on $\M$ with correlation functions $\rho^\k$ and reduced Palm measures $\mu_{\x_\k}$, as defined in (\ref{:2a}) and (\ref{:2b}) respectively, for $\x_\k\in (\R^d)^\k$, $\k\in \N$. We also assume that for $x,y\in\R^d$, the Radon-Nikodym derivative $d\mu_y/d\mu_x$ of $\mu_y$ with respect to $\mu_x$ exists, that is $\mu_x$ and $\mu_y$ are equivalent.
An infinite-particle jump-type system is characterized by its rate function $c(x,y,\xi)$, $x,y\in \R^d$, $\xi\in \M$, which controls the jump rate from $x$ to $y$ under the configuration $\xi$.
In this paper, we consider the case in which the rate function $c$ is given by
\begin{equation*}
c(x,y,\xi)=
\begin{cases}
\frac{1}{2}p(x,y)\mathsf{d}^\mu(x,y,\xi\setminus x),
\quad &\mbox{ if }x\in \xi,
\\
 0 \quad &\mbox{ if }x\notin \xi,
\end{cases}
\end{equation*}
with
\begin{align*}
\mathsf{d}^\mu(x,y,\eta) = 1+ \frac{\rho^1(y)}{\rho^1(x)}\frac{d\mu_y}{d\mu_x}(\eta)
\end{align*}
and a positive measurable symmetric function $p(x,y)=p(y,x)$ with conditions ({\bf p.1}) and ({\bf p.2}) specified in Section 2.
For example, the rate of the $\alpha$-stable process,
\begin{align*}\label{stable}
p(x, y) = |x-y|^{-d-\alpha}, \text{ $x,y \in\R^d$, $\alpha \in (0, 2)$}
\end{align*} 
satisfies these conditions.

We introduce the linear operator $L_0$ on the space of local smooth functions $\Dc$ defined by 
\begin{align*}
L_0 f(\xi) = \int_{\R^d} \xi(dx) \int_{\R^d} dy \; c(x,y,\xi)\left\{f(\xi^{x y})-f(\xi)\right\}, 
\end{align*}
and the associated bilinear form $\Emu$ on $\Dcmu$ in (\ref{:2e}) given by
\begin{align*}
\Emu(f, g) = \frac{1}{2} \int_{\M} \mu(d\xi) \int_{\R^d} \xi(dx) \int_{\R^d} dy \; p(x,y)\{ f(\xi^{x y})-f(\xi) \} \{ g(\xi^{x y})-g(\xi) \}. 
\end{align*}
In this situation, the rate function $c$ satisfies the following detailed balance condition: 
\begin{align*}
c(y, x,\xi^{x y}) = c(x,y,\xi) \frac{\rho^1(x)}{\rho^1(y)} \frac{d\mu_x}{d\mu_y}(\xi \setminus y), \quad x, y \in \R^d. 
\end{align*}
In \cite{E.19} sufficient conditions for the closability of the bilinear form $(\Emu, \Dcmu)$ and the quasi-regularity of its closure $(\Emu, \Dmu)$ were discussed.
One of the sufficient conditions for $\mu$ that ensures the closability of $(\Emu, \Dcmu)$ was given by employing the quasi-Gibbs property  (Definition \ref{D:2a}) introduced by Osada \cite{o.rm, o.rm2}.
The class of quasi-Gibbs measures includes Sine, Airy, Bessel and Ginibre random point fields (RPFs), which are significant RPFs in random matrix theory,  in addition to the Gibbs measures of Ruelle's class. 
Sufficient conditions for the quasi-regularity of the Dirichlet form $(\Emu, \Dmu)$ are given in \As[A.2] in Section 2. 
Under these conditions, the associated Hunt process $\Xi_t$ with the reversible measure $\mu$ exists.
Our results can be applied to interacting $\alpha$-stable systems. 
Using \As[p.2] and \As[A.2] for translation-invariant RPFs (for instance, the Dyson or Ginibre RPF) we can construct interacting $\alpha$-stable systems for any $\alpha \in (0, 2)$. Note that the $1$-correlation function $\rho^1$ of Airy RPFs satisfies $\rho^1(x) = O(|x|^{1/2})$ as $x \to -\infty$, and so the parameter $\alpha$ for Airy interacting systems is restricted to $\alpha \in (\frac{1}{2}, 2)$.

We label each particle of the process $\Xi_t$ to obtain the $(\R^d)^{\N}$-valued strong Markov process $\X_t=(X_t^j)_{j\in\N}$ and present its stochastic differential equation (SDE) representation. 
For Dirichlet processes, 
the main tool for obtaining the representation is Fukushima's decomposition. Because the processes $X_t^j$, $j\in\N$ are not Dirichlet processes for $\Xi_t$, we employ the argument in \cite{o.tp}. We introduce a sequence of Dirichlet forms $(\Emuk, \Dmuk)$, $\k\in\N$ which are the closure of $(\Emuk, \Dcmuk)$ defined in (\ref{:31e}). We show that the Dirichlet forms $(\Emuk, \Dmuk)$, $\k\in\N$ are quasi-regular (Theorem \ref{T:31a}) and the sequence of the associated processes $(\X_t^{[\k]}, \Xi_t^{[\k]})$, $k\in\N$ has the following consistency condition:
$$
\Xi_t^{[\k]}= \Xi_t^{[\k+1]}+ \delta_{\X_t^{[\k],\k}}
\quad \mbox{and} \quad
\X_t^{[k+1]}= (\X_t^{[\k],1}, \X_t^{[\k],2},\dots,\X_t^{[\k],\k},\X_t^{[\k+1],\k+1}),
$$
which implies that the process $\X_t=(X_t^j)_{j\in\N}$ exists:
\begin{align*}
&\Xi_t =\Xi_t^{[\k]}+\sum_{j=1}^\k\delta_{X_t^j}, 
\quad \X_t^{[\k]} = (X_t^j)_{j=1}^\k, 
\quad \k\in\N
\end{align*}
(Theorem \ref{T:31c}).
For each $j\in\N$, $X_t^j$ is a Dirichlet process for $\{(\X_t^{[\k]}, \Xi_t^{[\k]})\}_{\k\in\N}$, which is composed using the sequence $(\X_t^{[\k]}, \Xi_t^{[\k]})$, $\k\in\N$. Thus, we can apply Fukushima's decomposition and obtain the following infinite-dimensional SDE (ISDE) representation:
\begin{equation}\label{:1a}
X_t^j=X_0^j+\int_0^t \int_S  \int_0^\infty N^j(dsdudr) \;
u a^j(u,r,\X,s),
\quad j\in\N ,
\end{equation}
where
$\bN=(N^j)_{j\in\N}$ are independent Poisson RPFs on $[0,\infty)\times \R^d \times [0,\infty)$ whose intensity measure is the Lebesgue measure $dsdudr$, and $a^j$, $j\in \N$ are coefficients defined as
$$
a^j(u,r,\X,s)=\mathbf{1}(0\le r \le c(X_{s-}^j,X_{s-}^j+u, \Xi_{s-})),
\quad j\in\N ,
$$
(Theorem \ref{T:32a}).

The ISDE representation implies the existence of (weak) solutions of (\ref{:1a}). 
We call a solution $\X$ of (\ref{:1a}) a strong solution if $\X$ is a function of $\bN$ (Definition \ref{D:33c}).
We examine the existence and uniqueness of a strong solution of (\ref{:1a}) by employing the argument in \cite{o-t.tail}, which was used for the ISDE of interacting Brownian motions. This is sufficiently robust to apply to our ISDE. 
For a solution $\X$ and $m\in\mathbb{N}$ we set
\[
\X^{[\m]}=(X^1,X^2,\dots,X^\m),
\quad
\X^{\m*}=(X^{\m+1}, X^{\m+2},\dots).
\]
For a given $\X^{\m*}$ we consider the following $\m d$-dimensional SDEs:

\begin{equation}\label{:1b}
Y_t^{\m,j}(t)=Y_0^{\m, j}+\int_0^t \int_S  \int_0^\infty N^j(dsdudr) \;
u a^j(u,r,(\Y^\m,\X^{\m*}),s),
\; 1\le j \le \m.
\end{equation}
We call the sequence $(\Y^\m=(Y^{\m,1}, Y^{\m,2},\dots,
Y^{\m,\m}))_{\m\in\N}$ solutions of an infinite system of 
finite-dimensional SDEs associated with (\ref{:1a}).
Suppose that (\ref{:1b}) has a unique strong solution in the sense of Definition \ref{D:33g}. Then $(\Y^\m,\X^{\m*})$, $\m\in\N$ satisfies the following consistency condition:
\begin{align*}
(\Y^\m,\X^{\m*})=(\Y^{\m+1},\X^{\m+1*})=\X, \quad \m\in \N, \; \mbox{ a.s. }
\end{align*}
From the fact the solution $\X$ is measurable with respect to the $\sigma$-field generated by $\bN=(N^j)_{j\in\N}$, $\X_0=(X_0^j)_{j\in\N}$ and the tail $\sigma$-field
\begin{align*}
\mathcal{T}_{\rm path}((\R^d)^{\mathbb{N}})=\bigcap_{\m=1}^\infty
\sigma( \X^{\m*}).
\end{align*}
The existence of a strong solution is derived from the triviality of $\mathcal{T}_{\rm path}((\R^d)^{\mathbb{N}})$ concerning the conditional distribution of $\X$ given by $\bN$, and the uniqueness of solutions is derived from the family of tail events with probability $1$ being independent of the conditional distribution \cite[First tail theorem]{o-t.tail}.
Sufficient conditions for such solutions are given in \cite[Second tail theorem]{o-t.tail}. 

The reminder of this paper is organized as follows:
In Section 2, we prepare some terminology and previous results.
Section 3 states the main results and presents some relevant examples.
In Section 4, we give proofs of the theorems.
Finally, Section 5 derives the sufficient conditions for the assumptions in the theorems and checks the assumption for the examples in Section 3.
\vskip 10mm

\section{Preliminaries}\label{S:2} 

Let $S$ be a closed set in $\R^d$ such that $0 \in S$ and $\overline{S^{{\rm{int}}}}=S$, where $S^{{\rm{int}}}$ denotes the interior of $S$. 
Let $\M$ be the configuration space over $S$ defined by
$$
\M=\M(S)=\{ \xi = \sum_{j\in\Lambda} \delta_{x^j}; \xi(K)<\infty \mbox{ for all compact sets } K \subset S\},
$$ 
where $\Lambda$ is a countable index set and $\delta_a$ denotes the delta measure at $a$.  $\M$ is a Polish space with the vague topology. For a topological space $\mathcal{X}$ we denote the topological Borel field of $\mathcal{X}$ as $\mathcal{\Borel}(\mathcal{X})$. For $x,y\in \R^d$ and $\xi \in \M$, we write $\xi^{xy}=\xi-\delta_x+\delta_y$ and $\xi\setminus x = \xi-\delta_x$ if $\xi(\{x\})\ge 1$. Let $\mu$ be a probability measure on $\M$. 
We say that a nonnegative permutation-invariant function $\rho^{\k}$ on $S^{\k}$ is the $\k$-correlation function of $\mu$ if 
\begin{equation}\label{:2a} 
\int_{A_1^{\k_1} \times \cdots \times A_\ell^{\k_\ell}} \rho^{\k}(\x_\k)d\x_\k = \int_{\M} \prod_{i=1}^\ell \frac{\xi(A_i)!}{(\xi(A_i)-\k_i)!}\mu(d\xi)
\end{equation}
for any sequence of disjoint bounded measurable subsets $A_1, \ldots, A_\ell \subset S$ and a sequence of positive integers $\k_1, \ldots, \k_\ell$ satisfying $\k_1+\cdots+\k_\ell=\k$. 
When $\xi(A_i)< \k_i$, according to our interpretation, 
$\xi(A_i)!/(\xi(A_i)-\k_i)!=0$ by convention.
Let $\tilde{\mu}^{[\k]}$ be the measure on $(S^{\k}\times\M, \Borel(S^{\k})\otimes \Borel(\mathfrak{\M}))$ determined by
$$
\tilde{\mu}^{[\k]}\Big(\prod_{i=1}^\ell A_i^{\k_i} \times \mathfrak{N}\Big)= \int_{\mathfrak{N}} \prod_{i=1}^\ell \frac{\xi(A_i)!}{(\xi(A_i)-\k_i)!}\mu(d\xi)
$$
for any $\mathfrak{N}\in \Borel(\M)$ and $\{A_i\}_{i=1}^\ell$ as defined above.
The measure $\tilde{\mu}^{[\k]}$ is called the $\k$-Campbell measure of $\mu$. In the case where $\mu$ has a $\k$-point correlation function $\rho^{\k}$, there exists a regular conditional probability measure $\tilde{\mu}_{\x_\k}$
for $\x_\k = (x_\k^1, \ldots, x_\k^\k) \in S^\k$
satisfying 
$$
\int_{A_1^{\k_1} \times \cdots \times A_\ell^{\k_\ell}} 
\tilde{\mu}_{\x_\k}(\mathfrak{N})\rho^{\k}(\x_\k)d\x_\k = \tilde{\mu}^{[\k]}\big( \prod_{i=1}^\ell A_i^{\k_i} \times \mathfrak{N} \big).
$$
The measure $\tilde{\mu}_{\x_\k}$ is called the Palm measure of $\mu$ at $\x_\k \in S^{\k}$ (see e.g., \cite{kal}).
We use the probability measure $\mu_{\x_\k}\equiv \tilde{\mu}_{\x_\k}(\cdot - \sum_{j=1}^\k \delta_{x_\k^j})$, which is called the reduced Palm measure of $\mu$ at $\x_\k \in S^{\k}$.
Informally, $\mu_{\x_\k}$ is given by 
\begin{align} \label{:2b}
\mu_{\x_\k} = \mu \Big(  \cdot - \sum_{j=1}^\k \delta_{x_\k^j} \Big| \text{$\xi(\{ x_\k ^j \}) \geq 1$ for any $j=1, \ldots, \k$}  \Big). 
\end{align}
We set 
\begin{equation*} 
\muk(d\x_\k d\eta) = \mu_{\x_\k}(d\eta)\rho^\k(\x_\k)d\x_\k.
\end{equation*}
We call $\mu^{[\k]}$ the reduced $\k$-Campbell measure. In this paper, we assume the existence of the reduced $\k$-Campbell measure for each $\k\in\N$. In addition, we always use $\mu^{[\k]}$ instead of $\tilde{\mu}^{[\k]}$, and call $\mu^{[\k]}$ the $\k$-Campbell measure.

In \cite{E.19} Hunt processes describing infinite-particle jump-type systems were constructed using the Dirichlet form technique. We now review the results.
Let $\Ur (x)= \{ y \in S ; |y-x| \le r \}$. We set $U_r =U_r(0)$ and denote the $\m$-product of $\Ur$ by $\Ur^\m$. 
Let $\M_{r, \m} = \{ \xi \in \M; \xi(\Ur) = \m \}$ for $r,\m \in\N$. We set maps $\pi_r, \pi_r^c : \M\to \M$ such that
$$
\pi_r(\xi) = \xi(\cdot \cap \Ur) \quad \mbox{ and }\quad \pi_r^c(\xi) = \xi(\cdot \cap \Ur^c).
$$
A function $\x_{r, \m} : \M_{r, \m} \to \Ur^{\m}$ is called a $\Ur^{\m}$-coordinate (or a coordinate on $\M_{r, \m}$) of $\xi$ if 
\begin{equation*}
\pi_r(\xi) = \sum_{j=1}^{\m} \delta_{x_{r, \m}^j(\xi)}, \quad \x_{r, \m}(\xi)=(x_{r, \m}^1(\xi), \ldots, x_{r, \m}^{\m}(\xi)). 
\end{equation*}
For $f : \M \to \mathbb{R}$, a function $f_{r, \xi}^{\m}(x): \M \times \Ur^{\m} \to \mathbb{R}$ is called the $\Ur^{\m}$-representation of $f$ if $f_{r, \xi}^{\m}$ satisfies the following: 
\begin{enumerate}
\item[(1)] $f_{r, \xi}^{\m}(\cdot)$ is a permutation-invariant function on $\Ur^{\m}$ for each $\xi \in \M$. 
\item[(2)] $f_{r, \xi_{1}}^{\m}(\cdot) = f_{r, \xi_{2}}^{\m}(\cdot)$ \; if $\pi_r^c(\xi_{1}) = \pi_r^c(\xi_{2})$ and $\xi_{1}, \xi_{2} \in \M_{r, \m}$. 
\item[(3)] $f_{r, \xi}^{\m}(\x_{r, \m}(\xi)) = f(\xi)$ for $\xi \in \M_{r, \m}$, where $\x_{r, \m}(\xi)$ is a $\Ur^{\m}$-coordinate of $\xi$. 
\item[(4)] $f_{r, \xi}^{\m}(\cdot) = 0$ for $\xi \notin \M_{r, \m}$. 
\end{enumerate}
Note that for fixed $r, \m \in \N$, and $\xi \in \M$ $f_{r, \xi}^{\m}$ is uniquely determined and $f(\xi) = \sum_{{\m}=0}^\infty f_{r, \xi}^{\m}({\x}_{r, \m}(\xi))$. 
We say that a function $f : \M \to \R$ is {\it local} if $f$ is $\sigma[\pi_r]$-measurable for some $r \in \N$. 
When $f$ is $\sigma[\pi_r]$-measurable, its $\Ur^\m$-representation is a permutation-invariant function $f_r^\m$ on $\Ur^\m$ such that 
$$
f(\xi)=f(\pi_r(\xi))=f_r^\m(x^1,x^2,\dots,x^\m), \quad \xi\in \M_{r,\m} \mbox{ with } \pi_r(\xi)=\sum_{j=1}^\m \delta_{x^j}.
$$
We say that a local function $f$ is {\it smooth} if $f_r^\m$ is smooth for any $\m\in\N$. 
We denote the space of smooth local functions on $\M$ by $\Dc$. 
For $f, g \in \Dc$ we set $\D[f, g] : \M \to \R$ by
\begin{align}\label{:2c}
\D[f, g](\xi) = \frac{1}{2} \sum_{j\in\Lambda} \int_{S} \{ f(\xi^{x^j y})-f(\xi) \} \{ g(\xi^{x^j y})-g(\xi) \} p(x^j, y)dy, \quad \xi = \sum_{j\in\Lambda} \delta_{x^j}, 
\end{align}
where $p$ is a nonnegative measurable function on $S\times S \setminus \Delta$ with $\Delta = \{ (x, y) \in S \times S ; x = y \}$  satisfying $p(x, y)=p(y,x)$ and 
\vskip 3mm \noindent
\As[p.1] \ $p(x,y) = O(|x-y|^{-(d+\alpha)})$ as $|x-y| \to \infty$ for some $\alpha>0$. 
\vskip 1mm \noindent
\As[p.2] \ $p(x,y) = O(|x-y|^{-(d+\gamma)})$ as $|x-y| \to +0$ for some $0 < \gamma < 2$. 
\vskip 3mm\noindent
Under conditions \As[p.1] and \As[p.2] we have
\begin{equation}\notag
\int_S (1\wedge |y-x|^2) p(x,y)dy <\infty,
\quad x \in S. 
\end{equation}
For $\alpha \in (0, 2)$, a rotation-invariant, symmetric $d$-dimensional $\alpha$-stable process is related to the Dirichlet form $(\tilde{\mathfrak{E}}, \tilde{\mathfrak{D}})$ given by
\begin{align*}
&\tilde{\mathfrak{E}}(u, v) = \frac{1}{c(d, \alpha)}\int_{\R^d \times \R^d \setminus \Delta} \frac{(u(y)-u(x))(v(y)-v(x))}{|x-y|^{d+\alpha}} dxdy \quad \text{for $u, v \in \tilde{\mathfrak{D}}$}, \\
&\tilde{\mathfrak{D}} = \{ u \in L^2(\R^d, dx) , \tilde{\mathfrak{E}}(u, u) < \infty \}, 
\end{align*}
where $c(d, \alpha)$ is the positive constant given by 
$$ 
c(d, \alpha) = \frac{2^{-\alpha+1}\pi^{\frac{d+1}{d}}}{\Gamma(\frac{\alpha}{2}+1)\Gamma(\frac{\alpha+d}{2})\sin\frac{\pi\alpha}{2}}, 
$$
and $\Delta = \{ (x, y) \in \R^d \times \R^d; x=y \}$. 
Note that the rate of the $\alpha$-stable process,
$p(x, y) = |x-y|^{-d-\alpha}$, $x,y \in\R^d$, $\alpha \in (0, 2)$
satisfies conditions \As[p.1] and \As[p.2].

For some given $f$ and $g$ in $\Dc$, the right-hand side of (\ref{:2c}) depends only on $\xi$. Hence, $\D[f, g](\xi)$ is well defined. 
We introduce the bilinear form defined by
\begin{equation}\label{:2e}
\begin{split}
&\Emu(f, g) = \int_{\M} \D[f, g](\xi)\mu(d\xi), \\ 
&\Dcmu = \{ f \in \Dc \cap L^2(\M, \mu) ; \Emu(f, f) < \infty \}. 
\end{split}
\end{equation}
We make the following assumptions:
\vskip 3mm \noindent
\As[A.1] \quad  $(\Emu, \Dcmu)$ is closable on $L^2(\M, \mu)$. \\ 
\\
\As[A.2] \quad $\rho^1(x) = O\left( |x|^{\kappa} \right)$ as $|x| \to \infty$ for some $0 \leq \kappa < \alpha$, where $\alpha >0$ in \As[p.1]. 
\vskip 3mm

For a topological space $\mathcal{X}$, we denote the set of all $\mathcal{X}$-valued right continuous functions on $[0,\infty)$ with left limits by $W(\mathcal{X})$. Here, we induce $W(\mathcal{X})$ to have the Skorohod topology. 
The following lemma is a refinement of \cite[Theorem 2.1]{E.19}. The proof is essentially the same as that of Theorem 3.1 and is omitted. 

\begin{lemma}(\cite[Theorem 2.1]{E.19})\label{L:2a}
Assume that \As[A.1] and \As[A.2] hold. Let $(\Emu, \Dmu)$ be the closure of $(\Emu, \Dcmu)$ on $L^2(\M, \mu)$. Then, we have the following. 
\begin{enumerate}
\item $(\Emu, \Dmu, L^2(\M, \mu))$ is a quasi-regular Dirichlet space. 

\item The Hunt process $(\Xi_t, \{ \P_{\xi} \}_{\xi \in \M})$ associated with $(\Emu, \Dmu, L^2(\M, \mu))$ exists. 

\item There exists a measurable subset $\M_{\mu}$ satisfying 
\begin{align}\notag
\P_{\xi}(\Xi \in W(\M_{\mu}))=1
\quad \mbox{ for any $\xi\in \M_{\mu}$.}
\end{align}
\item The process $(\Xi_t, \{ \P_{\xi} \}_{\xi \in \M})$ is reversible with respect to $\mu$, that is,
$$
\P_{\mu}(\Xi_s\in \mathfrak{N}_1, \Xi_t\in \mathfrak{N}_2)
=\P_{\mu}(\Xi_s\in \mathfrak{N}_2, \Xi_t\in \mathfrak{N}_1),
\quad s<t, \quad \mathfrak{N}_1, \mathfrak{N}_2 \in \Borel(\M),
$$ 
where $\P_\mu= \int_{\M}\P_\xi \; \mu(d\xi)$. 
\end{enumerate}
\end{lemma}

\vskip 3mm

We set $\Omega = W(\M)$ and $\Borel_t(\Omega)=\sigma[\omega_s ; 0\le s\le t]$.
We introduce the $\sigma$-fields defined by
$$
\mathcal{F}_t = \bigcap_{\nu} \mathcal{F}_t^\nu, \quad \mathcal{F}=\bigcap_{\nu} \Borel(\Omega)^\nu,
$$
where the intersections are taken over all Borel probability measures $\nu$,
$\mathcal{F}_t^\nu$ is the completion of 
$\Borel_{t}^+ (\Omega)=\cap_{\varepsilon >0}\Borel_{t+\varepsilon}(\Omega)$
with respect to $\P_\nu=\int_{\M}\P_\xi \; \nu(d\xi)$, and $\Borel(\Omega)^\nu$ is that of $\Borel(\Omega)$.
Let $\Xi=\{\Xi_t\}$ be the canonical process, that is,
$\Xi_t(\omega)=\omega_t$ for $\omega\in \Omega$.
 Then, $\Xi$ is adapted to $\{\mathcal{F}_t\}_{t\ge 0}$.

We recall the definition of capacity \cite[Section 2.1]{fot.2}.
Let $\mathcal{O}$ denote the family of all open subsets of $\M$. Let $(\mathfrak{E}, \mathfrak{D})$ be a quasi-regular Dirichlet form on $L^2(\M,\mu)$, and $\mathfrak{E}_1(f,f)= \mathfrak{E}(f,f)+(f,f)_{L^2(\M,\mu)}$.
For $\mathfrak{O}\in \mathcal{O}$, we define
\begin{align}
&\mathcal{L}_{\mathfrak{O}}=\{u\in\mathfrak{D} : u\ge 1, \; \mbox{ $\mu$-a.e. on $\mathfrak{O}$}\}
\nonumber
\\ \notag
&{\rm Cap}^{\mu}(\mathfrak{O})=
\begin{cases}
\inf_{u\in\mathcal{L}_\mathfrak{O}}\mathfrak{E}_1
(u,u), \quad &\mathcal{L}_{\mathfrak{O}}\not=\emptyset
\\
\infty &\mathcal{L}_{\mathfrak{D}}=\emptyset
\end{cases}
\end{align}
and for any set $\mathfrak{N}\subset \M$, we set
\begin{equation} \label{:2g}
{\rm Cap}^{\mu}(\mathfrak{N})=\inf_{\mathfrak{O}\in\mathcal{O}, \mathfrak{N}\subset \mathfrak{O}}{\rm Cap}^{\mu}(\mathfrak{\mathfrak{O}}).
\end{equation}
We call this the capacity of $\mathfrak{N}$.
We use the abbreviated form ``q.e.'' for ``quasi everywhere'' (see \cite[p155]{fot.2} for a precise definition).

We now define the quasi-Gibbs measures. 
Let $\Phi : S\to\R\cup \{\infty\}$ be a self-potential, and
$\Psi : S\times S \to\R\cup \{\infty\}$ be an interaction potential with $\Psi(x,y)=\Psi (y,x)$.
We introduce the Hamiltonian $\mathcal{H}_r$ on $\Ur$ defined by
\begin{equation*}
\mathcal{H}_r(\zeta) = \sum_{x^j \in \Ur} \Phi(x^j) + \sum_{x^i, x^j \in \Ur, i<j} \Psi(x^i, x^j) \quad \text{for $\zeta = \sum_j \delta_{x^j}$. }
\end{equation*}

Let $\Lambda_r$ be the Poisson RPF whose intensity is the Lebesgue measure on $\Ur$, and set $\Lambda_r^\m=\Lambda_r (\cdot \cap \M_{r,\m})$.
\begin{definition} \label{D:2a}
An RPF $\mu$ is called a $(\Phi, \Psi)$-quasi-Gibbs measure if its regular conditional probabilities 
\begin{equation*}
\mu_{r, \xi}^\m(\cdot) = \mu(\left. \pi_r(\eta) \in \cdot \ \right| \pi_r^c(\eta) = \pi_r^c(\xi), \eta(\Ur) = \m)
\end{equation*}
satisfy
\begin{equation} \label{:2h}
\cref{qg}^{-1} e^{-\mathcal{H}_r(\zeta)} \Lambda_r^\m(d\zeta) \leq \mu_{r, \xi}^\m(d\zeta) \leq \cref{qg} e^{-\mathcal{H}_r(\zeta)} \Lambda_r^\m(d\zeta)
\quad \mbox{ for any $r, \m \in \N$ and $\mu$-a.s. $\xi$}. 
\end{equation}
Here $\Ct \label{qg}=\cref{qg}(r, \m, \xi)$ is a positive constant depending on $r, \m $, and $\xi$. For two measures $\mu$, $\nu$ on a $\sigma$-field $\mathcal{F}$, we write $\mu \leq \nu$ if $\mu(A) \leq \nu(A)$ for all $A \in \mathcal{F}$. 
\end{definition}
\vskip 3mm

We make the following assumption. 
\vskip 1mm \noindent
\As[QG] \quad $\mu$ is a $(\Phi, \Psi)$-quasi-Gibbs measure. 
Moreover, there exist upper semi-continuous functions $\Phi_0: S \to \R \cup \{ \infty \}$, $\Psi_0: S \to \R \cup \{ \infty \}$ and positive constants $\Ct \label{Phi}$ and $\Ct \label{Psi}$ such that 
\begin{equation*}
\cref{Phi}^{-1}\Phi_0(x) \leq \Phi(x) \leq \cref{Phi}\Phi_0(x), \quad \text{for all $x \in S$, } 
\end{equation*}
\begin{equation*}
\cref{Psi}^{-1}\Psi_0(x-y) \leq \Psi(x, y) \leq \cref{Psi}\Psi_0(x-y), \quad \text{for all $x, y \in S$}. 
\end{equation*}
\vskip 1mm \noindent

We also introduce the following condition.
\vskip 3mm 

\noindent \As[EQ] \quad  For any $\k\in\N$ and almost all $\x_\k$ and $\y_\k$, with respect to the Lebesgue measure,
the reduced Palm measures $\mu_{\x_\k}$ and $\mu_{\y_\k}$ are equivalent and $\rho^\k(\x_\k)>0$ . 

\vskip 3mm 

Condition \As[EQ] is used to show the closability of the Dirichlet form.

\begin{lemma}(\cite[Theorem 2.5]{E.19})\label{L:2b}
Assume that \As[EQ] and \As[QG] hold. Then $(\Emu, \Dcmu)$ is closable on $L^2(\M, \mu)$. 
\end{lemma}

\noindent
\begin{remark}
We can relax condition \As[EQ] if we restrict the configuration space $\M$ to an appropriate subspace $\mathfrak{X}$ with $\mu (\mathfrak{X})=1$.
For instance, 
let $\mu$ be the canonical Gibbs state with a hard-core potential of range $r>0$. Then, $\mu$ does not satisfy \As[EQ]. However, if we consider the configuration space of hard balls 
$$
\mathfrak{X}=\{ \xi= \sum_{j \in \N} \delta_{x^j}\in \M : |x^i -x^j|\ge r, \; i\not= j \},
$$
we can readily check the condition that for   $\x_\k, \y_\k \in S^\k$ with $\eta + \sum_{j=1}^\k \delta_{x^j}, \eta + \sum_{j=1}^\k \delta_{y^j}\in\mathfrak{X}$,
$$
\frac{ d\mu_{\y_\k}}{d\mu_{\x_\k}}(\eta) \in (0,\infty) 
\text{  }.
$$
This condition can be used instead of \As[EQ] to generalize our results (see, e.g., \cite{T89}).
\end{remark}

\section{Statement of results} \label{S:3}

\subsection{Tagged particle processes and their consistency}\label{SS:31}

We introduce the subset $\Msi$ of $\M$ defined by
\begin{equation*}
\Msi = \{ \text{$\xi \in \M$; $\xi( \{ x \} ) \leq 1$ for all $x \in S$ and $\xi(S)=\infty$} \}.
\end{equation*}
Then an element of $\Msi$ represents a configuration of infinitely many particles with no
multiple point. A configuration $\xi = \sum_{j=1}^{\infty} \delta_{x^j} \in \Msi$ is identical to a countable set $\{ x^j; j \in \N \}$. In the following, for simplicity of notation, we write $x\in \xi$ for $\xi \in \Msi$ and $\xi(\{x\}) = 1$.
A map $\lab$ from $\Msi$ to $S^{\N}$ is called a label if $\lab(\xi)=(x^j)_{j\in\N}$ satisfies $\xi=\sum_{j\in\N}\delta_{x^j}$.
We set $\unl((x^j)_{j\in\N})=\sum_{j\in\N}\delta_{x^j} $ for $(x^j)_{j\in\N}\in S^{\N}$ and call this an ``unlabel".
It is clear that $\unl (\lab(\xi))=\xi$ for $\xi\in\Msi$.
For any $\k\in\N$, a map $\lab^{[\k]}$ from $\Msi$ to $S^\k\times\Msi$ is called a $\k$-label if $\lab^{[\k]}(\xi)=((x^j)_{j=1}^\k,\eta)$ satisfies $\sum_{j=1}^\k \delta_{x^j}+\eta=\xi$.
We also set $\unl^{[\k]}(((x^j)_{j=1}^\k,\eta))=\sum_{j=1}^\k \delta_{x^j}+\eta$,
and often write $\unl(((x^j)_{j=1}^\k,\eta))$ instead of $\unl^{[\k]}(((x^j)_{j=1}^\k,\eta))$ when this will not cause any confusion.

Let $W_{\rm SIN}$ denote the set of elements $\Xi=\{\Xi_t\}_{t\ge 0}$ of $W(\Msi)$ such that 

\noindent (1) any pair of particles do not jump simultaneously, 

\noindent (2) there is no collision among particles,

\noindent (3) no particle explodes.

For a given label $\lab$, we can determine the labeled process $\X=(X^{j})_{j\in \N}$ associated with $\Xi \in W_{\rm SIN}$ in the following way.
We initially label the process $\Xi$ as $\lab(\Xi_0)=(X_0^{j})_{j\in\N}$, which implies that
$\Xi_0 = \sum_{j=1}^{\infty} \delta_{X_0^j}$.
By virtue of conditions (1), (2), and (3) above, we can follow the label of each particle and determine the label $\lab_{\rm path}$ for $\Xi$, 
$$
\lab_{\rm path}(\Xi)=\X=(X^j)_{j\in\N}\in W(S^{\N}),
$$
satisfying $\unl(\X_t)=\Xi_t$, $t\ge 0$.

Let $(\Xi_t, \{ \mathbb{P}_{\xi} \}_{\xi \in \M})$ be the unlabeled process constructed in Lemma \ref{L:2a}.
We state the following condition:

\vskip 2mm
\noindent
\As[SIN] $\P_{\xi}(\Xi \in W_{\rm SIN})=1 \quad \text{for q.e. $\xi$}$.
\vskip 2mm
\noindent
Note that, for $\P_\xi$-a.s. $\Xi$, any pair of particles do not jump simultaneously.
Then, $(\Xi_t, \{ \mathbb{P}_{\xi} \}_{\xi \in \M})$ satisfies \As[SIN] if and only if the following two conditions hold:
\vskip 2mm
\noindent 
\As[NCL] \quad There is no collision among particles, that is,
$\Capa^{\mu}(\Msi^c) = 0$.
\\ 
\As[NEX] \quad Any tagged particle never explodes, that is,
\begin{align*}
\P_{\xi} \left( \sup_{0 \leq t \leq T} |\lab_{\rm path}(\Xi)_t^j| < \infty \ \text{for all $T$, $j \in \N$} \right) = 1 \quad \text{for q.e. $\xi$}. 
\end{align*}
\vskip 3mm \noindent
Condition \As[SIN] implies that the labeled process $(\X,  \mathbb{P}_{\xi})$ can be constructed 
for q.e. $\xi$.
From the construction above, $\X$ is not a Dirichlet process of $\Xi$ (see, e.g., \cite{Fol81} for the definition of Dirichlet processes.)
We present another representation of $\X=(X^j)_{j \in \N}$ as a Dirichlet process of the process $\{(\X^{[\k]}, \Xi^{[\k]})\}_{\k\in\N}$ defined below following the argument used by Osada \cite{o.tp}.

Let $m_{r,T}:W(S^{\N})\to \N\cup \{\infty\}$ be a function defined by
\begin{align}\notag 
m_{r,T}(\w)=\inf \{ m\in\N
: \min_{t\in [0,T]}|w_t^j|>r \mbox{ for all $j\in\N$ such that $j>m$} \}.
\end{align}
We introduce the following assumption:
\vskip 3mm
\noindent\As[NBJ]  Only a finite number of particles visit a given bounded set during any finite time interval, that is,
$$
P(m_{r,T}(\X)<\infty)=1, \mbox{ for each $r, T \in \N$.}
$$
In Lemma \ref{L:53b}, we show that if $(\Xi_t, \P_\xi)$ satisfies \As[NCL], then $\mathfrak{l}_{\rm path}(\Xi)$ satisfies \As[NBJ].

For a given $\lab$, let $\{\lab^{[\k]}\}_{\k\in\N}$ be a consistent sequence  of $\lab^{[\k]} : \Msi \to S^{\k} \times \Msi$ such that
$$
\lab^{[\k]}(\xi)^j=\lab(\xi)^j, \quad j=1,2,\dots,\k,
\quad \text{and} \quad
\lab^{[\k]}(\xi)^{\k+1}=\sum_{j=\k+1}^{\infty}\delta_{\lab(\xi)^j},
\quad \k\in\N.
$$
We also define the consistent sequence $\lab_{\rm path}^{[\k]} : W_{\rm SIN} \to W(S^\k \times \M)$ given by
$$
\lab_{\rm path}^{[\k]}(\Xi)^j=\lab_{\rm path}(\Xi)^j, 
\quad j=1,2,\dots, \k,
\quad \text{and} \quad
\lab_{\rm path}^{[\k]}(\Xi)^{\k+1}=\sum_{j=\k+1}^{\infty}\delta_{\lab_{\rm path}(\Xi)^j},
\quad \k\in\N.
$$
We denote $\lab_{\rm path}^{[\k]}(\Xi)=(\X^{[\k]},\Xi^{[\k]})$,
and call $\X^{[\k]}$ the $\k$-tagged particle process and $\Xi^{[\k]}$ the environment process associated with $\mathfrak{l}_{\rm path}(\Xi)$.
We construct the consistent sequence of processes
$\{(\X^{[\k]},\Xi^{[\k]})\}$, $\k\in\N$ using the Dirichlet form argument.

Let $\Dck = C_0^{\infty}(S^\k) \otimes \Dc$, where $C_0^{\infty}(S^\k)$ denotes the set of all smooth functions with compact supports. For $\f, \g \in \Dck$, we set 
\begin{equation*}
\nabla^{[\k]}[\f, \g](\x_\k, \eta) = \frac{1}{2}\sum_{j=1}^{\k} \int_{S} \nadeffsa{j}{y}\f(\x_\k, \eta) \nadeffsa{j}{y}\g(\x_\k, \eta) p(x_\k^j, y)dy, 
\end{equation*}
where  
\begin{equation}\label{:31b}
\nadeffsa{j}{y} \f(\x_\k, \eta) = \f(\x_\k^{j, y},\eta) - \f(\x_\k, \eta)
\end{equation}
with 
\begin{align}\label{:31c}
\x_\k^{j, y}=(x_\k^1, \ldots, x_\k^{j-1}, y, x_\k^{j+1}, \ldots, x_\k^\k)
\end{align}
for $j=1,2,\dots, \k$, $\x_\k = (x_\k^1, \ldots, x_\k^\k) \in S^\k$, and $y\in S$. 
For $f, g \in \Dck$, we set 
\begin{equation} \label{:31d}
\Dk[\f, \g](\x_\k, \eta) = \nabla^{[\k]}[\f, \g](\x_\k, \eta) + \D[\f(\x_\k, \cdot), \g(\x_\k, \cdot)](\eta)
\end{equation}
and
\begin{equation}\label{:31e}
\begin{split}
&\Emuk(\f, \g) = \int_{S^\k \times \M} \Dk[\f, \g](\x_\k, \eta) \muk(d\x_\k d\eta), 
\\ 
&\Dcmuk = \{ \f \in \Dck \cap L^2(S^\k \times \M, \muk) ; \Emuk(\f, \f) < \infty \}.  
\end{split}
\end{equation}
For $\k \in \N$, we introduce the following assumptions, which are analogous to \As[A.1] and \As[A.2], respectively. 
\vskip 1mm \noindent
\As[A.1.$\k$] \quad $(\Emuk, \Dcmuk)$ is closable on $L^2(S^\k \times \M, \muk)$. 
\vskip 2mm \noindent
\As[A.2.$\k$] \quad $\rho_{\x_\k}^1(y) = O\left( |y|^{\kappa} \right)$ as $|y| \to \infty$ for a.e. $\x_\k \in S^\k$ with $\kappa $ in \As[A.2].

\vskip 3mm
\noindent We write \As[A.1.0] and \As[A.2.0] for \As[A.1] and \As[A.2], respectively.
The following theorem is a generalization of Lemma \ref{L:2a}.

\vskip 3mm
\begin{theorem} \label{T:31a}
Let $\k\in\N$.
Assume that \As[A.1.$\k$] and \As[A.2.$\k$] hold. Let $(\Emuk, \Dmuk)$ be the closure of $(\Emuk, \Dcmuk)$ on $L^2(S^\k \times \M, \muk)$. Then
\begin{enumerate}
\item $(\Emuk, \Dmuk, L^2(S^\k \times \M, \muk))$ is a quasi-regular Dirichlet space. 
\item There exists a Hunt process $((\X_t^{[\k]},\Xi_t^{[\k]}), \{ \mathbb{P}_{(\x_\k,\eta)}^{[\k]} \}_{(\x_k,\eta) \in S^\k\times\M})$ associated with the quasi-regular Dirichlet space $(\Emuk, \Dmuk, L^2(S^\k \times \M, \muk))$. 

\item The process $((\X_t^{[\k]},\Xi_t^{[\k]}), \{ \mathbb{P}_{(\x_\k,\eta)}^{[\k]} \}_{(\x_\k,\eta) \in S^\k\times\M})$ is reversible with respect to $\muk$, that is, for $\Lambda_1, \Lambda_2 \in \Borel(S^\k\times \M)$ and $0\le s<t <\infty$, 
\begin{align*}
\mathbb{M}^{[\k]}((\X_s^{[\k]},\Xi_s^{[\k]})\in \Lambda_1, (\X_t^{[\k]},\Xi_t^{[\k]})\in \Lambda_2)
= \mathbb{M}^{[\k]}((\X_t^{[\k]},\Xi_t^{[\k]})\in \Lambda_1, (X_s^{[\k]},\Xi_s^{[\k]})\in \Lambda_2), 
\end{align*}
where
\begin{align}\label{:31f}
&\mathbb{M}^{[\k]}(d\X^{[\k]} d\Xi^{[\k]})=\int_{S\times \M}
\muk (d\x_{\k} d\eta) 
\mathbb{P}_{(\x_{\k},\eta)}^{[\k]}(d\X^{[\k]} d\Xi^{[\k]}).
\end{align}
\end{enumerate}
\end{theorem}

We give the proof of this theorem in Section \ref{SS:41}.

\vskip 3mm

We define the capacity $\Capa^{\muk}$ associated with $(\Emuk, \Dmuk, L^2(S^\k \times \M, \mu^{[\k]}))$ in the same way as \eqref{:2g}.

\noindent The following theorem is a generalization of Lemma \ref{L:2b}.
\begin{theorem} \label{T:31b}
Let $\k\in\N$. Assume that \As[QG] and \As[EQ] hold.
Then, \As[A.1.$\k$] holds.
\end{theorem}
The proof of this theorem is given in Section \ref{SS:41}.

\vskip 3mm

The following theorem ensures the consistency of the sequence of processes $\{(\X^{[\k]},\Xi^{[\k]})\}$.

\begin{theorem}\label{T:31c}
Assume that \As[A.1.$\k$] and \As[A.2.$\k$] for any $\k\in\N\cup\{0\}$ hold.
Suppose that $(\Xi_t, \{ \mathbb{P}_{\xi} \}_{\xi \in \M})$ satisfies \As[SIN]. 
\\ \noindent
(i)  There exists a set $\M_0 \subset \Msi$ with 
\begin{align} \label{:31g}
&\Capa^{\mu}(\M_0^c) = 0, 
\qquad \P_{\xi} (\Xi_t \in \M_0 \ \text{for $t\ge 0$}) = 1 \quad \text{for all $\xi \in \M_0$}, 
\end{align}
and for all $\k \in \N$, 
\begin{align} \label{:31h}
&\P_{(\x_\k,\eta)}^{[\k]} = \P_{\unl((\x_\k,\eta))} \circ (\lab_{\rm path}^{[\k]})^{-1} \quad \text{for all $(\x_\k,\eta)\in \lab^{[\k]}(\M_0)$}, 
\\ \label{:31k}
&\P_{\xi} = \P_{\lab^{[\k]}(\xi)}^{[\k]} \circ \unl^{-1} \quad \text{for all $\xi \in \M_0$}. 
\end{align}
\noindent
(ii) There exists an $S^{\N}$-valued process $\X=(X^j)_{j\in\N}$ such that
$$
\X^{[\k]}=(X^j)_{j=1}^\k
\quad \text{and} \quad
\Xi^{[\k]}= \sum_{j=\k+1}^\infty \delta_{X^j}.
$$
\end{theorem}
We give the proof of this theorem in Section \ref{SS:42}.

\vskip 3mm

\begin{remark}
From the above lemma, $X^j$, $j\in\N$ is a Dirichlet process of $(\X^{[\k]},\Xi^{[\k]})$ 
for $\k\ge j$, and hence $\X$ is a Dirichlet process of the consistent sequence of processes $\{(\X^{[\k]},\Xi^{[\k]})\}_{\k\in\N}$. 
\end{remark}

\subsection{ISDE representation}\label{SS:32}

Under condition \As[EQ], we can define $\mathsf{d}^\mu$ by
\begin{align}\notag 
\mathsf{d}^\mu(x,y,\eta)=1+ \frac{\rho^1 (y)}{\rho^1(x) }\frac{d\mu_y}{d\mu_x}(\eta),
\quad x,y\in S, \quad \eta\in\M.
\end{align}
Note that for alomost all $y$, $\mathsf{d}^\mu (\cdot, y, \cdot) \in L_{\rm loc}^1 (S\times\M,  \mu^{[1]})$.

We introduce the rate function defined as
\begin{equation}\label{:32b}
c(x,y,\xi)=
\begin{cases}
\frac{1}{2}p(x, y)\mathsf{d}^\mu(x,y,\xi\setminus x),
&\text{if $x\in\xi$}
\\
0,
&\text{if $x\notin\xi$}
\end{cases}
\end{equation}
and set
$$
a(u,r,x,\xi\setminus x)= \mathbf{1}\left(0\le r \le c(x,x +u,\xi)\right).
$$
We introduce the following measurable functions on $\R^d\times [0,\infty)\times W(S^\N)\times [0,\infty)$
\begin{align*}
a^j(u,r,\X,s)=a(u,r,X_{s-}^j, \Xi_{s-}-\delta_{X_{s-}^j}), \quad j\in\N,
\end{align*}
where $\Xi = \unl (\X)$. Then we have the following representation of the process $\X$, 
the proof of which is given in Section \ref{SS:43}.

\begin{theorem}\label{T:32a}
Suppose that the assumptions in Theorem \ref{T:31c} and \As[EQ] hold.
Let $\X=\mathfrak{l}(\Xi)$ be the labeled process on $(\Omega, \mathcal{F}, \{\P_\xi\}_{\xi\in \M})$ constructed in Theorem \ref{T:31c}.
Then, $\X$ satisfies
\begin{align}\label{ISDE}
&X_t^j=X_0^j+\int_0^t \int_S  \int_0^\infty N^j(dsdudr) \;
u a^j(u,r,\X,s),
\quad j\in\N,
\tag{ISDE}
\end{align}
where
$\bN=(N^j)_{j\in\N}$ are independent Poisson RPFs on $[0,\infty)\times \R^d \times [0,\infty)$ whose intensity measure is the Lebesgue measure $dsdudr$.
\end{theorem}

We consider (\ref{ISDE}) under the conditions that 
there exist subsets $\M_{\rm SDE}$ and $\mathfrak{H}$ of $\M_{0}$ such that $\frak{H} \subset \M_{\rm SDE}$ and $\mu(\M_{\rm SDE})=\mu(\mathfrak{H})=1$ with
\begin{align} \label{:32e}
&\X \in W(\mathbf{S}_{\rm SDE}),
\\ \label{:32f}
&\X_0 \in \mathbf{H}, 
\end{align}
where 
$\mathbf{S}_{\rm SDE}=\mathfrak{u}^{-1}(\M_{\rm SDE})$ and $\mathbf{H}= \mathfrak{u}^{-1}(\mathfrak{H})$. 

We introduce the increasing sub-$\sigma$-fields of $\Borel(W(S^{\N}))$ defined by
$\Borel_t=\sigma[\w_s, 0\le s\le t], \w\in W(S^{\N})$, $0\le t < \infty$.

\begin{definition}\label{D:32a}
A stochastic process $\X=(X^j)_{j\in\N}$ 
is called a weak solution of (\ref{ISDE}) with (\ref{:32e})-(\ref{:32f})
if $(\X, \bN)$ defined on a probability space $(\Omega, \mathcal{F}, P)$ with a reference family $\{\mathcal{F}_t\}_{t\ge 0}$
satisfies

\noindent (i)  $\X=(X^j)_{j\in\N}$ is $\mathcal{F}_t/\Borel_t$-measurable for each $0\le t<\infty$.

\noindent (ii) $\bN=\{N^j\}_{j\in\N}$ is an i.i.d. sequence of  $\{\mathcal{F}_t\}$-Poisson RPFs on $[0,\infty)\times \R^d\times [0,\infty)$ whose intensity measures are the Lebesgue measure $dsdudr$.

\noindent (iii) $a^j(u,r,\X,s)$, $j\in\N$ are predictable, and
\begin{align}\notag 
&E\Big[
\int_0^t \int_{U_1}  \int_0^\infty N^j (dsdudr) \;|u|^2 a^j(u,r,\X,s)
\Big]<\infty,
\\ \notag 
&E\Big[
\int_0^t \int_{U_1^c}  \int_0^\infty N^j(dsdudr) \; a^j(u,r,\X,s)
\Big]<\infty.
\end{align}

\noindent (iv) With probability $1$, the process $(\X, \bN)$ satisfies (\ref{ISDE}) with (\ref{:32e})--(\ref{:32f}).
\vskip 3mm
\noindent
Moreover, $\X=(X^j)_{j\in\N}$ 
is called a weak solution of (\ref{ISDE}) starting from $\x\in \mathbf{H}$ if $\X_0=\x$ holds.
\end{definition}

From Theorems \ref{T:32a} and \ref{T:31c}(i), if we take $\mathbf{H}=\mathbf{S}_{\rm SDE}\subset \mathfrak{u}^{-1}(\M_{\mu}\cap \M_{\rm s.i.})$, then we have the following result. Here $\M_{\mu}$ is given in Lemma \ref{L:2a}. 

\begin{corollary}\label{C:32a}
Suppose that the assumptions in Theorem \ref{T:32a} hold.
Let $\X=\mathfrak{l}(\Xi)$ be the labeled process on $(\Omega, \mathcal{F}, \{\P_\xi\}_{\xi\in \M})$ constructed in Theorem \ref{T:31c}.
Then, for $\mu\circ \lab^{-1}$ a.s. $\x=\mathfrak{l}(\xi)$, $\X$ is a weak solution of (\ref{ISDE}) with (\ref{:32e})-(\ref{:32f}).
\end{corollary}

\begin{remark}\label{R:32a}
In \cite{o.rm, o.rm2} a diffusion process describing a system of infinitely many Brownian particles with interactions was constructed using the Dirichlet form technique. In \cite{o.isde}, the following ISDE representation was given:
\begin{equation}\notag  
dX^j_t=dB^j_t+\frac{1}{2}
\mathrm{d}^\mu(X_t^j, \sum_{i\in\N, i\not=j}\delta_{X_t^i})dt, \quad j\in\mathbb{N}.
\end{equation}
Here, $\mathrm{d}^\mu \in L_{\rm loc}^1 (S\times \M, \mu^{[1]})$ is the logarithmic derivative of $\mu$, which is defined as
\begin{align}\notag 
\int_{S\times \M} \mathrm{d}^\mu(x,\eta)\varphi (x,\eta)\mu^{[1]}(dx d\eta)
=- \int_{S\times \M}\nabla_x \varphi(x,\eta)\mu^{[1]}(dxd\eta), \; \varphi \in \mathcal{D}^{1}_{\circ}.
\end{align}
In the case where $\mu$ is a canonical Gibbs measure associated with the potentials $\Phi$ and $\Psi$, its logarithmic derivative is represented as
$$
\mathrm{d}^\mu (x,\sum_{j\in\N}\delta_{y_j})
=-\nabla\Phi(x)-\sum_{j\in\N}\nabla\Psi(x,y_j).
$$
In \cite{BDO}, a relation between $\mathrm{d}^\mu(x,\eta)$ and $\mathsf{d}^\mu (x,y ; \eta)$ was discussed and it was shown that 
\begin{align*}
&\mathrm{d}^\mu(x,\eta)=\nabla_y\mathsf{d}^\mu (x,y ; \eta)\Big|_{y=x}
\quad \mbox{ for } x\in \R^d, \eta\in \M
\end{align*}
holds under Assumption 1 \cite[Proposition 2.2]{BDO}.
\end{remark}

\subsection{Existence and uniqueness of strong solutions}\label{SS:33}

We examine solutions of (\ref{ISDE}). First, we define several notions of the uniqueness of solutions of (\ref{ISDE}). 
We can regard an i.i.d. sequence $\bN=\{N^j\}_{j\in\N}$ of $\{\mathcal{F}_t\}$-Poisson RPFs on $[0,\infty)\times \R^d\times [0,\infty)$ as a Poisson RPF on $([0,\infty)\times \R^d\times [0,\infty))^{\N}$.
Let $\X$ be a weak solution of (\ref{ISDE}) with the $\{\mathcal{F}_t\}$-Poisson RPF $\bN$ on $([0,\infty)\times \R^d\times [0,\infty))^{\N}$.
Set 
$$
\mathcal{M}=\M([0,\infty)\times \R^d\times [0,\infty)),
\quad
\mathcal{M}_t=\M([0,t]\times \R^d\times [0,\infty)).
$$
We can regard the $\sigma$-field $\Borel(\mathcal{M}_t^\N)$ as a sub-$\sigma$-field of 
$\mathcal{B}(\mathcal{M}^\N)$ and denote this by $\mathcal{B}_t(\mathcal{M}^\N)$. 
We denote the completions of the topological Borel fields $\mathcal{B}(\mathcal{M}^\N)$ and $\mathcal{B}_t(\mathcal{M}^\N)$ with respect to the distribution of $\bN$ by $\overline{ \mathcal{B}(\mathcal{M}^\N) }$ and $\overline{\mathcal{B}_t(\mathcal{M}^\N)}$.

\begin{definition}[uniqueness in law]\label{D:33a}
The uniqueness in law of weak solutions of (\ref{ISDE}) states that whenever $\mathbf{X}$  and $\mathbf{X}'$ are two weak solutions whose initial distributions coincide, then the laws of $\mathbf{X}$  and $\mathbf{X}'$ coincide.
\end{definition}

\begin{definition}[pathwise uniqueness]\label{D:33b}
The {\it pathwise uniqueness} of solutions to (\ref{ISDE}) with (\ref{:32e})--(\ref{:32f}) states that any solutions $\X$ and $\X'$ with (\ref{:32e})--(\ref{:32f}) on the same probability space with the same $\{\mathcal{F}_t\}$-Poisson random field $\bN$ on $([0,\infty)\times \R^d\times [0,\infty))^{\N}$ such that $\X_0=\X_0'$ satisfy $\X=\X'$ a.s.
\end{definition}

\begin{definition}[strong solution starting at $\x$]\label{D:33c}
A weak solution $\X$ of (\ref{ISDE}) with the i.i.d. sequence  $\bN$ of $\{\mathcal{F}_t\}$-Poisson RPFs is called a {\it strong solution} of (\ref{ISDE}) starting at $\x$ 
if $\X_0=\x$ a.s. and if there exists a function $F_\x : \mathcal{M}^\N\to W(S^{\N})$  such that $F_\x$ is 
$\overline{ \mathcal{B}(\mathcal{M}^\N) }/\mathcal{B}(W(S^{\N}))$-measurable,  
$\overline{ \mathcal{B}_t(\mathcal{M}^\N) }/\mathcal{B}_t(W(S^{\N}))$-measurable, and satisfies 
$\X=F_\x(\bN)$ a.s.
\end{definition}

\begin{definition}[unique strong solution starting at $\x$]\label{D:33d}
We say that (\ref{ISDE})  has a unique strong solution starting at $\x$ if there exists a 
$\overline{ \mathcal{B}(\mathcal{M}^\N) }/\mathcal{B}(W(S^{\N}))$-measurable function $F_\x : \mathcal{M}^\N \to W(S^{\N})$ such that
\\
(i) for any weak solution $(\hat{\X}, \hat{\bN})$ of (\ref{ISDE}) starting at $\x$
$\hat{X} = F_\x (\hat{\bN})$ a.s.

\noindent (ii) for any $\{\mathcal{F}_t\}$-Poisson RPF $\bN$ on $([0,\infty)\times \R^d\times [0,\infty))^\N$,
the process $F_\x (\bN)$ is a strong solution of (\ref{ISDE}) starting at $\x$. 
We call $F_\x$ a unique strong solution starting at $\x$.
\end{definition}

For a given condition \As[C] we give the following definitions.

\begin{definition}[unique strong solution under constraint]\label{D:33e}
For a condition \As[C], 
we say that (\ref{ISDE})  has a unique strong solution starting at $\x$ under constraint \As[C] if there exists a 
$\overline{ \mathcal{B}(\mathcal{M}^\N) }/\mathcal{B}(W(S^{\N}))$-measurable function $F_\x : \mathcal{M}^\N \to W(S^{\N})$ such that
\\
(i) for any weak solution $(\hat{\X}, \hat{\bN})$ of (\ref{ISDE}) starting at $\x$ satisfying \As[C],
$\hat{X} = F_\x (\hat{\bN})$ a.s.

\noindent
(ii) for any $\{\mathcal{F}_t\}$-Poisson RPF  $\bN$ 
on $([0,\infty)\times \R^d\times [0,\infty))^{\N}$,
the process $F_\x (\bN)$ is a strong solution of (\ref{ISDE}) starting at $\x$ satisfying \As[C]. 
We call $F_\x$ a unique strong solution starting at $\x$ under constraint \As[C].
\end{definition}

We introduce the following condition for a family of strong solutions $\{F_\x\}$ of (\ref{ISDE}) given for $P\circ \X_0^{-1}$-a.s. $\x$.

\vskip 3mm

\noindent \As[MF] \quad
$P(F_\x (\bN)\in A)$ is $\overline{\Borel(S^\N)}^{P\circ \X_0^{-1}}$-measurable in $\x$ for any $A \in \Borel(W(S^{N}))$.

\vskip 3mm
For a family of strong solutions $F_\x$ satisfying \As[MF] we set
\begin{equation}\notag 
P_{\{F_\x\}}= \int P(F_\x (\bN)\in \cdot) P\circ \X_0^{-1}(d\x).
\end{equation}

\begin{definition}\label{D:33f}
For a condition \As[C], we say that (\ref{ISDE}) has a family of unique strong solution $\{F_\x\}$ starting at $\x$  for $P\circ \X_0^{-1}$-a.s. $\x$, under the constraints of \As[MF] and \As[C] if $\{F_\x\}$ satisfies \As[MF] and \As[C]. Furthermore, (i) and (ii) must be satisfied:
\\
(i) for any weak solution $(\hat{\X}, \hat{\bN})$ under $\hat{P}$ of (\ref{ISDE}) with
$\hat{P}\circ \hat{\X}_0^{-1} \prec  P\circ \X_0^{-1}$
satisfying \As[C], it holds that, for $\hat{P} \circ \hat{\X}_0^{-1}$-a.s. $\x$, 
$\hat{\X}= F_\x(\hat{\bN})\quad \text{$\hat{P}_\x$- a.s.}$,
where $\hat{P}_\x= \hat{P}(\cdot | \hat{\X}_0=\x )$.
\\
(ii) for any $\{\mathcal{F}_t\}$-Poisson RPF $\bN$ on $([0,\infty)\times \R^d\times [0,\infty))^\N$,
$F_\x(\bN)$ is a strong solution of (\ref{ISDE}) satisfying \As[C] starting at $\x$ for $P\circ \X_0^{-1}$-a.s $\x$.
\end{definition}

\noindent
\begin{remark}\label{R:33a}
Similar to Definitions \ref{D:33e}--\ref{D:33f},  we can introduce the notion of a constrained version of uniqueness in Definitions \ref{D:33a}--\ref{D:33b}.
\end{remark}

Let $(\X,\bN)$ be a weak solution of (\ref{ISDE}). 
We introduce an infinite sequence of finite-dimensional SDEs associated with $(\X,\bN)$. 
We set 
$$
\X^\m=(X^1,X^2,\dots,X^\m),
\quad
\X^{\m*}=(X^{\m+1}, X^{\m+2},\dots).
$$
For each $\m \in \N$ we introduce the $\m d$-dimensional SDE of $\Y^{\m}$ for $(\X,\bN)$ starting from $\x_\m=(x_1,x_2,\dots x_{\m})$ defined on $(\Omega, \mathcal{F}, \P_\xi)$ such that
\begin{align}
&\label{:33b}
Y_t^{\m,j}=Y_0^j+\int_0^t \int_S  \int_0^\infty N^j(dsdudr) \;
u a^j(u,r,(\Y^\m,\X^{\m*}),s),
\quad 1 \leq j \leq \m, 
\\ \label{:33c}
&\Y_0^\m=\x_\m, 
\\ \label{:33d}
&\Y_t^\m \in \mathbf{S}_{\rm SDE}^\m (\X_t) \; \mbox{ for any $t>0$,}
\end{align}
where 
$$
\mathbf{S}_{\rm SDE}^\m(\x)=\{\y_\m\in S^\m :  
\mathfrak{u}(\y_\m)+\mathfrak{u}(\x^{\m*})\in \mathfrak{u}(\mathbf{S}_{\rm SDE})  \}
$$
The sequence of SDEs (\ref{:33b})--(\ref{:33d}) is called an infinite system of finite-dimensional SDEs associated with a solution $(\X,\bN)$ of (\ref{ISDE}).
We introduce the notion of a strong solution of SDE (\ref{:33b})--(\ref{:33d}) for $(\X, \bN)$.
Set $\mathcal{C}^\m$ and $\mathcal{C}_t^\m$ as the completions of 
$\Borel(\mathcal{M}^\m \times W(S^\N))$ and $\Borel_t(\mathcal{M}^\m \times W(S^\N))=\Borel_t(\mathcal{M}^\m)\otimes \Borel_t(W(S^\N))$
with respect to the distribution of $(\bN^\m=\{N^j\}_{j=1}^\m, \X^{\m*})$, respectively.

\begin{definition}\label{D:33g}
(i) $(\Y^\m, \bN^\m, \X^{\m *})$ is called a {\it strong solution} of (\ref{:33b})-(\ref{:33d}) defined on $(\Omega, \mathcal{F}, P, \{\mathcal{F}_t\})$ if it satisfies (\ref{:33b})--(\ref{:33d}), and there exists a function $F_{\x_\m}^\m : \mathcal{M}^\m\times W(S^\N) \to W(S^{\m})$ such that $F_{\x_\m}^\m$ is
$\mathcal{C}^\m$-measurable, $\mathcal{C}_t^\m /\Borel_t(W(S^{\m}))$ measurable, and $\Y^\m=F_{\x_\m}^\m(\bN^\m, \X^{m*})$.

\noindent (ii) We say that SDE (\ref{:33b})-(\ref{:33d}) for $(\X, \bN)$ has a {\it unique strong solution}  if
any (weak) solution $(\hat{\Y}^\m, \bN^\m, \X^{\m *})$ satisfies $\hat{\Y}^\m=F_{\x_\m}^\m(\bN^\m, \X^{\m*})$.

\noindent (iii)
We say that the {\it pathwise uniqueness} of solutions (\ref{:33b})--(\ref{:33d}) holds
if whenever $\Y^{\m}$ and $\hat{\Y}^{\m}$ are two solutions defined on the same probability space $(\Omega, \mathcal{F}, \P_\xi)$ with the same $\{\mathcal{F}_t\}$-Poisson RPF $\bN=\{N^j\}_{j\in\N}$ and the process $\X$ starting from $\x$, then
$P_\x(\Y_t^\m = \hat{\Y}_t^\m, \mbox{ for all $t\ge 0$})=1$.
\end{definition}

We now state the following condition.

\vskip 3mm

\noindent \As[IFC] For each $\m\in\N$,
(\ref{:33b})--(\ref{:33d}) has a unique strong solution $\Y^\m= F_{\x_\m}^\m (\bN^\m, \X^{\m *})$.

\vskip 3mm 
\noindent
We assume that \As[IFC] holds, and set
\begin{align*}
& \mathbb{F}_\x^\m(\bN, \X)=(\mathbb{F}_\x^{\m,j}(\bN, \X))_{j\in\N}=
(F_{\x_\m}^\m (\bN^\m, \X^{\m *}), \X^{\m*})=(\Y^\m,\X^{\m*}).
\end{align*}

\vskip 3mm

Then, the following lemma can be readily derived from \As[IFC].
The tag \As[IFC]  comes from the property in Lemma \ref{L:33a}(i) that implies $(\Y^\m, \bN^\m, \X^{\m *}), \m\in\N$ is an {\it infinite system of finite-dimensional SDEs with consistency}. 

\begin{lemma}(\cite[Lemma 4.2]{o-t.tail}) \label{L:33a}
Assume that \As[IFC] is satisfied. Then, the following hold.

\noindent (i)
The sequence $\{\mathbb{F}_\x^\m \}_{\m\in\N}$ is consistent in the sense that
for $P$-a.s.
\begin{align*}
\mathbb{F}_\x^{\m,j}(\bN, \X)=\mathbb{F}_\x^{\m+\n,j}(\bN, \X)
\quad \mbox{for all $1\le j\le \m$, \; $\m,\n\in\N$}
\end{align*}
\noindent (ii)
There exists a map $F_{\x}^\infty :  \mathcal{M}^\N \times W(S^\N) \to W(S^\N)$ such that for any $j\in\mathbb{N}$ and $P$-a.s. $(\X, \bN)$,
$$
\displaystyle{\lim_{\m\to\infty}}\mathbb{F}_{\x}^{\m,j}(\bN, \X)= F_{\x}^{\infty,j}(\bN, \X).
$$
\noindent (iii)
$(\X,\bN)$ is a fixed point of $F_{\x}^\infty$ in the sense that, for $P_\x$-a.s.
\begin{align*}
(\X, \bN)= (F_{\x}^\infty (\bN, \X), \bN).
\end{align*}
\end{lemma}

 The tail $\sigma$-field $\mathcal{T}(\M)$ on $\M$ is given by 
 \begin{align} \label{:33g}
 \mathcal{T}(\M) = \bigcap_{r\in\N}\sigma(\pi_r^c).
 \end{align}
 We state the following condition.
 
\vskip 3mm

\noindent
\As[TT]
The tail $\sigma$-field $\mathcal{T}(\M)$ is $\mu$-trivial, that is, $\mu(\mathfrak{A})\in \{0,1\}$ for $\mathfrak{A}\in \mathcal{T}(\M)$. 

\vskip 3mm

Let $\mu$ be a probability measure on $\M$ and let $\Xi_t$ be an $\M$-valued process.
We say that $\Xi_t$ satisfies the $\mu$-absolute continuity condition if

\vskip 3mm

\noindent \As[AC]\quad
$\mu\circ\Xi_t^{-1}$ is absolutely continuous with respect to $\mu$ for any $t>0$.

\vskip 3mm
\noindent
We also say that an $S^{\N}$-valued process $\X_t$ satisfies the $\mu$-absolute continuity condition if
$\frak{u}(\X_t)$ satisfies the $\mu$-absolute continuity condition. 

By applying a modification of \cite[Theorem 11.1]{o-t.tail}, we have the following.

\begin{theorem}\label{T:33a}
Suppose that the assumptions in Theorem \ref{T:31c} hold.
Let $\X=\mathfrak{l}_{\rm path}(\Xi)$ be the labeled process on $(\Omega, \mathcal{F}, \{\P_\xi\}_{\xi\in \M})$ constructed in Theorem \ref{T:31c}.
Assume that \As[TT] holds for $\mu$ and \As[IFC] is satisfied.
Then $(\X,\bN)$ is a strong solution of (\ref{ISDE}) with (\ref{:32e})-(\ref{:32f}) starting at $\x$ 
for $\mu\circ\mathfrak{l} ^{-1}$-a.s.  $\x$ and satisfies \As[MF], \As[AC] for $\mu$, \As[SIN] and \As[NBJ].
Moreover, (\ref{ISDE}) with (\ref{:32e})-(\ref{:32f}) has a family of unique strong solutions $\{F_\x\}$ starting at $\x$ 
for $\mu\circ\mathfrak{l} ^{-1}$-a.s.  $\x$ under the constraints of \As[MF], \As[IFC], \As[AC] for $\mu$, \As[SIN], and \As[NBJ].
\end{theorem}

The following corollary is a direct consequence of the above theorem.
\begin{corollary}\label{C:33a}
Assume that \As[TT] holds for $\mu$.

\noindent
(i) The uniqueness in law of weak solutions of (\ref{ISDE}) with (\ref{:32e})--(\ref{:32f}) holds
under the constraints of \As[MF], \As[IFC], \As[AC] for $\mu$, \As[SIN], and \As[NBJ].

\noindent
(ii) The pathwise uniqueness of weak solutions of (\ref{ISDE}) with (\ref{:32e})--(\ref{:32f}) holds
under the constraints of \As[MF], \As[IFC], \As[AC] for $\mu$, \As[SIN], and \As[NBJ].
\end{corollary}

All determinantal RPFs (DRPFs) on continuous spaces are tail-trivial \cite{bqs,ly.18,o-o.tail}. These results are a generalization of the properties of DRPFs 
on discrete spaces \cite{BLPS,ST03,ly.03}. 

\vskip 3mm

We relax \As[TT] for $\mu$ through the tail decomposition of $\mu$.
Because  $\M$ is a Polish space, for a probability measure $\mu$, 
there exists a regular conditional probability measure using the tail $\sigma$-field $\mathcal{T}(\M)$ defined by
(\ref{:33g}). 
We set $\mu_{\rm Tail}^{\zeta}(\cdot)=\mu(\cdot | \mathcal{T}(\M))(\zeta) $.
Then, the following tail decomposition holds:
\begin{align}\label{:33h}
\mu (\cdot)= \int_{\M} \mu_{\rm Tail}^{\zeta}(\cdot)\mu(d\zeta).
\end{align}
When Theorem \ref{T:33a} is applied by replacing $\mu$ in $\mu_{\rm Tail}^{\zeta}$,
we can remove condition \As[TT].
From \cite[Lemma 14.2]{o-t.tail}, there exists a Borel set $\tilde{\mathfrak{H}}\subset \M$ such that $\mu(\tilde{\mathfrak{H}})=1$, and for any $\zeta\in\tilde{\mathfrak{H}}$
\begin{align}
& \mu_{\rm Tail}^{\zeta}(\mathfrak{A})\in \{0,1\} \quad \text{ for all $\mathfrak{A}\in \mathcal{T}(\M)$},
\label{:33k}
\\
& \mu_{\rm Tail}^{\zeta}(\{\xi \in \M : \mu_{\rm Tail}^{\zeta}=\mu_{\rm Tail}^{\xi}\})=1,
\label{:33l}
\\
&\text{$\mu_{\rm Tail}^{\zeta}$ and $\mu_{\rm Tail}^{\xi}$ are mutually singular on $\mathcal{T}(\M)$ if $\mu_{\rm Tail}^{\zeta}\not=\mu_{\rm Tail}^{\xi}$}.
\label{:33m}
\end{align}

\begin{theorem}\label{T:33b}
Suppose that \As[QG], \As[EQ], \As[A.2.$\k$], $\k \in \{0\}\cup\N $ for $\mu$ are satisfied. 
Then, for $\mu$-a.s. $\zeta$, 

\noindent (i)
 \As[QG], \As[EQ], \As[A.2.$\k$], $\k \in \{0\}\cup\N $ for $\mu^\zeta$ are satisfied.

\noindent (ii) 
There exists an unlabeled process $(\Xi_t, \{\P_\xi^\zeta\}_{\xi\in\M})$ associated with the quasi-regular Dirichlet form $(\mathcal{E}^{\mu^{\zeta}}, \mathcal{D}^{\mu^{\zeta}}, L^2(\M,\mu^{\zeta}))$. 

\noindent (iii) 
Assume that \As[SIN] holds for $(\Xi_t, \mathbb{P}_{\mu})$. The labeled process $\mathfrak{l}_{\rm path}(\Xi)$ on $(\Omega, \mathcal{F}, \P_{\mu^\zeta}^{\zeta})$ 
is a weak solution of (\ref{ISDE})  with (\ref{:32e})--(\ref{:32f}).

\noindent (iv) 
 Assume that $(\mathfrak{l}_{\rm path}(\Xi), \bN)$ on $(\Omega, \mathcal{F}, \P_\mu)$ satisfies \As[IFC]. Then ,$(\mathfrak{l}_{\rm path}(\Xi), \bN)$ is a strong solution of (\ref{ISDE})  starting at $\x$ for $\mu_{\rm Tail}^{\zeta}\circ\mathfrak{l} ^{-1}$ a.s. $\x$, 
and satisfies \As[MF], \As[AC] for $\mu_{\rm Tail}^{\zeta}$, \As[SIN], and \As[NBJ].
Moreover, 
(\ref{ISDE}) with (\ref{:32e})--(\ref{:32f}) has a strong solution $F_\x^{\zeta}$ starting at $\x$ for $\mu_{\rm Tail}^{\zeta}\circ\mathfrak{l} ^{-1}$ a.s. $\x$
under the constraints of \As[MF], \As[IFC], \As[AC] for $\mu_{\rm Tail}^{\zeta}$, \As[SIN], and \As[NBJ].
\end{theorem}

\subsection{Examples}\label{SS:34}

In this section we present some examples in which we can apply Theorems \ref{T:33a} and \ref{T:33b}.
We assume that $p$ is translation-invariant, that is, the value of $p(x, y)$ depends only on $x-y$. In particular, by \As[p.1], \As[p.2], and \As[A.2], we can apply our theorems to interacting $\alpha$-stable systems for $\alpha \in (\kappa, 2)$, where $\kappa$ is a growth order of $\rho^1$ in \As[A.2]. Recall that interacting $\alpha$-stable systems are processes associated with $p(x, y) = |x-y|^{-d-\alpha}$.

The examples of the probability measure $\mu$ below satisfy conditions \As[QG], \As[EQ], \As[A.2.$\k$], $\k\in\N\cup \{0\}$.
In fact, quasi-Gibbs measures include canonical Gibbs measures by definition. Moreover, it has been shown that Sine$_\beta$ RPFs ($\beta=1,2,4$),
Bessel$_{\alpha,\beta}$ RPFs ($\alpha \ge 1, \beta=2$),
Airy$_\beta$ RPFs ($\beta=1,2,4)$,
Ginibre RPFs are not canonical Gibbs measures, but quasi-Gibbs measures in \cite{h-o.bes,o.rm,o.rm2,o-t.airy}. In addition, the condition \As[EQ] can be shown(see, e.g., \cite{BDO,OsaShi}).
Condition \As[A.1.$\k$], $\k\in \N\cup \{0\}$ is derived from Theorem \ref{T:31b}. 

We can show that the associated processes $\Xi$ satisfy \As[NCL], \As[NEX], \As[NBJ], and \As[IFC] by checking the sufficient conditions in Subsections 5.1--5.4, respectively.

\begin{example} (Canonical Gibbs measure)
Let $\mu$ be a canonical Gibbs measure with a self-potential $\Phi$ and Ruelle's class pair-potential $\Psi$ \cite{Ru}. 
We assume that $d\ge 2$ or $d=1$ satisfying (\ref{:51n}).
Moreover, we assume that, for any $r\in\N$, there exist positive constants $\Ct\label{Psi2}$ and $\Ct\label{Psi3}$ such that
\begin{align*}
|\nabla\Psi(x)|, \; |\nabla^2\Psi(x)| \le \frac{ \cref{Psi2}}{(1+|x|)^{\cref{Psi3}}}, 
\quad \text{ for any $x$ with $|x|\ge 1/r$}.
\end{align*}
Then, $c(x, y, \xi)$, $\xi\in \Msi$, is described as
$$
c(x, y, \xi)=\frac{1}{2}p(x,y)\Big( 1+
\exp\Big\{ -\Phi (y)+\Phi (x)-\sum_{w\in \xi\setminus x} \left\{ \Psi(y, w) - \Psi(x, w) \right\} \Big\}\Big).
$$
In particular, when $\Phi=0$,
$$
c(x,y,\xi)=\frac{1}{2}p(x, y) \Big\{ 1+ 
\prod_{w\in\xi\setminus x}\frac{\exp(\Psi(w, x))}{\exp(\Psi(w, y))}
\Big\}. 
$$
The following are two examples of long-range interaction potentials. 

\noindent (1) (Lennard-Jones 6-12 potential)
$$
\Psi(x, y)=|x-y|^{-12}-|x-y|^{-6}, \quad x,y \in \R^3.
$$
 
\noindent (2) (Riesz potential) \quad Let $a>d$.
$$
\Psi(x, y)=|x-y|^{-a}, \quad x,y \in \R^d.
$$  

\end{example}

\begin{example}(RPFs related to random matrix theory)
Let $\mu$ be one of the following (quaternion) DRPFs:
Sine$_\beta$ RPF ($\beta=1,2,4$),
Bessel$_{\alpha,\beta}$ RPF ($\alpha \ge 1, \beta=2$),
Airy$_\beta$ RPF ($\beta=1,2,4)$,
Ginibre RPF \cite{forrester, Meh04}.
These are quasi-Gibbs measures with interaction potentials given by
$$
\Psi(x,y)=-\beta\log |x-y|, \quad x,y\in\R^d, \; d=1,2.
$$
The following are explicit forms of the rate function $c$ for $\mu$.

\vskip 3mm

\noindent (1) 
Sine$_\beta$ RPF $\mu_{\sin ,\beta}$ ($\beta=1,2,4$, $S=\R$): 

\noindent This is a $(0, -\beta\log |x-y|)$-quasi-Gibbs measure \cite{o.isde}, and for $\xi \in \Msi$
$$
c(x, y, \xi) = \frac{1}{2}p(x,y)\Big(1+
\lim_{r\to\infty}
\prod_{w\in \pi_r(\xi\setminus x)}\frac{|y-w|^\beta}{|x-w|^\beta}\Big).
$$
Condition \As[A.2.$\k$] is satisfied with $\kappa =0$. 
For $\beta=2$ \As[TT] is satisfied.

\noindent (2)
Bessel RPF $\mu_{{\rm Bessel},\alpha,2}$ ($\alpha \in [1,\infty), S=[0,\infty)$):

\noindent This is a $(-\alpha \log |x|, -\beta\log |x-y|)$-quasi-Gibbs measure \cite{{h-o.bes}} and for $\xi \in \Msi$
$$
c(x, y, \xi) = \frac{1}{2}p(x,y)\Big(1+
\left|\frac{y}{x}\right|^\alpha \lim_{r\to\infty}
\prod_{w\in \pi_r(\xi\setminus x)}\frac{|y-w|^\beta}{|x-w|^\beta}\Big).
$$
Condition \As[A.2.$\k$] is satisfied with $\kappa =0$ and \As[TT] is satisfied.

\noindent (3)
Airy RPF $\mu_{{\rm Ai}, \beta}$ ($\beta=1,2,4$, $S=\R$):

\noindent This is a quasi-Gibbs measure whose logarithmic derivative is given by 
$$
\mathrm{d}^{\mu_{{\rm Ai}, \xi}}(x,\eta)= 
 \beta \lim_{r\to\infty}\Big\{
 \sum_{w\in \pi_r(\xi\setminus x)}
\frac{1}{x-w}- \int_{|u|<r}\frac{\hat{\rho}(u)}{-u}du
\Big\},
$$
where $\hat{\rho}(x)=\frac{\sqrt{-x}}{\pi}\mathbf{1}_{(-\infty,0)}(x)$ \cite[Theorem 5.7]{o-t.airy}.
For $\xi \in \Msi$
$$
c(x, y, \xi) = \frac{1}{2}p(x,y)\Big\{1+
\lim_{r\to\infty}\exp\Big(\beta (x-y)\int_{|u|<r}\frac{\hat{\rho}(u)}{-u}du \Big)
\prod_{w\in \pi_r(\xi\setminus x)}\frac{|y-w|^\beta}{|x-w|^\beta}\Big\}.
$$
Condition \As[A.2.$\k$] is satisfied with $\kappa =\frac{1}{2}$. For $\beta=2$ \As[TT] is satisfied.

\noindent (4)
Ginibre RPF ($S=\R^2)$:
\\
\noindent This is a $(0, -2\log |x-y|)$-quasi-Gibbs measure \cite{o.rm}, and for $\xi \in \Msi$, 
\begin{equation*} 
c(x, y, \xi) = \frac{1}{2}p(x,y)\Big( 1+ 
\lim_{r \to \infty} \prod_{w\in \pi_r(\xi\setminus x)} \frac{|y-w|^2}{|x-w|^2}\Big). 
\end{equation*}
Condition \As[A.2.$\k$] is satisfied with $\kappa =0$ and \As[TT] is satisfied.
\end{example}

\begin{remark}\label{R:34a}
For an infinite-particle system of interacting Brownian motions, the existence and uniqueness of strong solutions have already been studied for the RPFs in the above examples
\cite{o-t.tail, h-o.bes, o-t.airy}.
\end{remark}

\section{Proofs of Theorems}\label{S:4}
\subsection{Proofs of Theorems \ref{T:31a} and \ref{T:31b}} \label{SS:41}

Let $\bd{A}=\{ \bd{a} = \{ a_r \}_{r \in \N}; a_r \in \N, a_r \leq a_{r+1} \ \text{for all $r \in \N$} \}$. We set $\1 = \{ 1 \}_{r \in \N}$. For $\bd{a}= \{ a_r \}_{r \in \N} \in \bd{A}$, let $\M[\bd{a}]= \{ \xi \in \M; \xi(\Qr) \leq a_r \ \text{for all $r \in \N$} \}$. Then, $\M[\bd{a}]$ is a compact set. 
Suppose that $\bd{a}, \bd{b} \in \bd{A}$ and $c \in \R$.  We set $\bd{a}^+ = \{ a_{r+1} \}_{r \in \N}$, $c\bd{a} = \{ ca_r \}_{r \in \N}$, and $\bd{a} + \bd{b} = \{ a_r + b_r \}_{r \in \N}$.  
We set $\M_{\bm{a}}^{\bm{b}} = \M[\bd{b}] \setminus \M[\bd{a}]$ for $\bd{a}, \bd{b} \in \bd{A}$. 
We introduce the function $\chia$ defined by  
\begin{align}\notag 
\chi[\bd{a}](\xi) = \mathsf{v} \circ d_{\bm{a}}(\xi), \quad d_{\bm{a}}(\xi) = \sum_{r=1}^{\infty} \sum_{j \in J_{r, \xi}} \frac{(2^r-|x_j|) \wedge 2^{r-1}}{2^{r-1}a_r}, 
\end{align}
where $(x^j)_{j \in \N}$ is a sequence such that $|x^j| \leq |x^{j+1}|$ for all $j \in \N$,  
$\xi=\sum_{j} \delta_{x^j}$, 
\begin{align*}
J_{r, \xi}=\{ j \in \N; j>a_r, x^j \in \Qr \}, 
\end{align*}
and $\mathsf{v} : \mathbb{R} \to [0, 1]$ is a decreasing smooth function satisfying
\begin{equation*}
\mathsf{v}(t)= \begin{cases}
1 &\text{if $t<0$}, \\
0 &\text{if $1<t$}. 
\end{cases}
\end{equation*}
Take $\beta\in (\kappa, \alpha)$. In this section we assume that $\bd{a}_n$ is given by
\begin{equation}\label{:41b}
\bd{a}_n = \{ a_{n, r} \}_{r \in \N} = \{ n2^{(d+\beta)r} \}_{r \in \N},
\end{equation}
where $\alpha$ and $\kappa$ are the positive constants in  \As[p.1] and \As[A.2.$\k$], respectively. 

\begin{lemma} \label{L:41a}
Let $\k\in\N$. Assume that \As[A.2.$\k$] holds. 
Then we have 
\begin{equation}\notag 
\mu_{\x_\k} \Big( \bigcup_{n \in \N} \M [\bd{a}_n] \Big) = 1 \quad \text{for a.e. $\x_\k \in S^\k$. }  
\end{equation}
\end{lemma}
\begin{proof}
From \As[p.1] and \As[A.2.$\k$]
\begin{align}\notag 
&E^{\mu_{\x_\k}}[\int_{ U_1^c}\xi(dx) |x|^{-(d+\beta)}]=\int_{U_1^c}  |y|^{-(d+\beta)} \rho_{\x_\k}^1(y) dy < \infty.
\end{align}
By Fubini's theorem, there exists $M=M(\xi)<\infty$ $\mu_{\x_\k}$-a.s. $\xi$ such that
\begin{align*}
\int_{U_1^c}\xi(dx) |x|^{-(d+\beta)} \le M(\xi), \quad \text{$\mu_{\x_\k}$-a.s. $\xi$.}
\end{align*}
Hence for $\mu_{\x_\k}$-a.s. $\xi$
\begin{align*}
\frac{1}{2^{(d+\beta)r}}\int_{ U_{2^r}}\xi(dx)  \le M(\xi)+\xi(U_1), \quad \text{for all $r\in\N$}
\end{align*}
and 
\begin{align*}
\frac{\xi( U_{2^r})}{2^{(d+\beta)r}}  \le M(\xi)+\xi(U_1), \quad \text{ for all $r\in \N$}.
\end{align*}
This proves the lemma.
\end{proof}

\begin{lemma} \label{L:41b}
Let $\k\in\N$.
Assume that \As[A.1.$\k$] holds.
 Let $\bd{a}_n$ be given by (\ref{:41b}). Then, we have $\varphi \otimes \chian f \in \Dmuk$ for any $\varphi \in C_0^{\infty}(S^\k)$ and $f \in \Dc$ with $\varphi \otimes f \in \Dinfmuk$. 
\end{lemma}
\begin{proof}
Let $\varphi \in C_0^{\infty}(S^\k)$ and $f \in \Dc$ with $\varphi \otimes f \in \Dinfmuk$. Then, we have
\begin{align}\notag
&\Emuk(\varphi \otimes \chian f, \varphi \otimes \chian f)
\\ \label{:41e}
&= \int_{S^\k\times\M} \mu^{[\k]}(d\x_\k d\eta) \left( \nabla^{[\k]}[\varphi, \varphi](\x_\k) \chian^2 f(\eta)^2  + \varphi(\x_\k)^2 \D [\chian f, \chian f](\eta) \right). 
\end{align}
By a direct calculation (see, e.g., the proof of \cite[Lemma 5.5]{E.19}), we have
\begin{equation}\label{:41f}
\D [g f, g f](\eta) \leq 2(\D [g, g](\eta)f(\eta)^2+g(\eta)^2\D [f, f](\eta)). 
\end{equation}
In particular, by substituting $\chian$ into $g$, we have 
\begin{equation}\label{:41g}
\D [\chian f, \chian f](\eta) \leq 2(\D [\chian , \chian](\eta)f(\eta)^2+\D [f, f](\eta)). 
\end{equation}
From \cite[Lemma 5.6]{E.19}, there exists a constant $\Ct\label{4_1}
>0$ such that 
\begin{equation} \label{:41g2}
\D[\chian , \chian ](\eta) \leq \cref{4_1} \quad \text{for any $\eta \in \M_{\bm{a}_n-\1}^{2(\bm{a}_n)^++\1}$. } 
\end{equation}
In addition, by a simple observation, we have
\begin{align}\label{:41g3}
\D[\chian , \chian ](\eta) = 0, \text{ 
for $\eta \notin \M_{\bm{a}_n-\1}^{2(\bm{a}_n)^++\1}$.}
\end{align}
From (\ref{:41g})--(\ref{:41g3}) along with the fact that $\chian(\eta) \leq 1$, we obtain
\begin{align}\notag 
&\int_{S^\k\times\M} \mu^{[\k]}(d\x_\k d\eta) \left( \nabla^{[\k]}[\varphi, \varphi](\x_\k)\chian^2 f(\eta)^2 + \varphi(\x_\k)^2 \D [\chian f, \chian f](\eta) \right). 
\\ \notag 
&\le 2(\cref{4_1}+1)\int_{S^\k\times\M} \mu^{[\k]}(d\x_\k d\eta)
\left\{  \nabla^{[\k]}[\varphi, \varphi](\x_\k)f(\eta)^2 + \D [f, f](\eta)\varphi(\x_\k)^2 + f(\eta)^2 \varphi(\x_\k)^2\right\} 
\\ \label{:41h}
&= 2(\cref{4_1}+1) \left\{ \Emuk(\varphi \otimes f, \varphi \otimes f) + \| \varphi \otimes f \|_{L^2(S^\k \times \M, \muk)}^2 \right\}.
\end{align}
Because $\varphi \otimes f \in \Dinfmuk$, we have that $\Emuk(\varphi \otimes f, \varphi \otimes f)<\infty$ and $\| \varphi \otimes f \|_{L^2(S^\k \times \M, \muk)}^2 <\infty$. Thus, \eqref{:41e} and \eqref{:41h} yields
\begin{equation*}
\Emuk(\varphi \otimes \chian f, \varphi \otimes \chian f) < \infty. 
\end{equation*}
Because $\chian$ is a decreasing limit of local smooth functions, the lemma follows from \As[A.1.$\k$].
\end{proof}

\begin{lemma} \label{L:41c}
Let $\k\in\N$.
Assume that \As[A.1.$\k$] and \As[A.2.$\k$] hold. Let $\bd{a}_n$ be given by (\ref{:41b}). Then, we have
$$ \lim_{n \to \infty}\| \varphi \otimes (1-\chian)f \|_{1}=0 $$
for any $\varphi \in C_0^{\infty}(S^\k)$ and $f \in \Dc$ with $\varphi \otimes f \in \Dinfmuk$, where $\| \cdot \|_1^2 = \Emuk(\cdot, \cdot)+\| \cdot \|_{L^2(S^\k \times \M, \muk)}^2$. 
\end{lemma}
\begin{proof}
We estimate the following quantity:
\begin{align}\notag 
&\Emuk(\varphi \otimes (1-\chian)f, \varphi \otimes (1-\chian)f) 
\\ \notag 
&= \int_{S^\k\times \M}  \mu^{[\k]}(d\x_{\k} d\eta) \big\{ \nabla^{[\k]}[\varphi, \varphi](\x_\k)\{(1-\chian)f(\eta)\}^2  
\\ \label{:41k}
&\quad\qquad\qquad + \varphi(\x_\k)^2 \D [(1-\chian)f, (1-\chian)f](\eta) \big\}. 
\end{align}
Note that $\eta \in \M[\bd{a}_n-\1]$ implies $1-\chian(\eta)=0$. Substituting $1-\chian$ into $g$ in \eqref{:41f}, we have
\begin{equation*}
\D [(1-\chian)f, (1-\chian)f](\eta) \leq 2(\D [1-\chian , 1-\chian](\eta)f(\eta)^2+\D [f, f](\eta)).
\end{equation*}
Here, we use $0 \leq 1-\chian(\eta) \leq 1$ for any $\eta$. 
From this and $\D [1-\chian , 1-\chian](\eta) = 0$ for $\eta \in \M[\bd{a}_n-\1]$, we find that
\begin{align}\label{:41l}
\text{R.H.S. of \eqref{:41k}}
&\leq \int_{S^\k\times\M[\bd{a}_n-\1]^c} \mu^{[\k]}(d\x_\k d\eta) \Big\{  \nabla^{[\k]}[\varphi, \varphi](\x_\k)f(\eta)^2  \notag \\
&\quad\quad + \varphi(\x_\k)^2 \left( \D [f, f](\eta) +f(\eta)^2 \D [1-\chian, 1-\chian](\eta) \right) \Big\}. 
\end{align}

From \cite[Lemma 5.6]{E.19}, for the constant $\cref{4_1}>0$ in the above, we have 
$$ \D[1-\chian , 1-\chian ](\eta) \leq \cref{4_1} \quad 
\text{for any $\eta \in \M_{\bm{a}_n-\1}^{2(\bm{a}_n)^++\1}$. } 
$$
Combining this with \eqref{:41l} gives
\begin{align}\notag
\text{R.H.S. of \eqref{:41l}} &\leq \cref{4_2} \int_{S^\k\times\M[\bd{a}_n-\1]^c} \mu^{[\k]}(d\x_\k d\eta)
\Big\{ \nabla^{[\k]} [\varphi, \varphi](\x_\k)f(\eta)^2 
\\ \notag 
&\qquad\qquad\qquad + \varphi(\x_\k)^2\D[f, f](\eta)+\varphi(\x_\k)^2f(\eta)^2 \Big\}, 
\end{align}
where $\Ct\label{4_2}=(\cref{4_1} \vee 1)$.
From Lemma \ref{L:41a} we have that
\begin{equation}\label{:41n}
\mu_{\x_\k}\Big( \bigcap_{n \in \N} \M[\bd{a}_n-\1]^c \Big) = 0 \quad \text{for a.e. $\x_\k \in S^{\k}$}. 
\end{equation}
Because $\varphi \otimes f \in \Dinfmuk$, \eqref{:41k}--\eqref{:41n} yield 
\begin{equation} \label{:41o}
\lim_{n \to \infty} \Emuk(\varphi \otimes (1-\chian)f, \varphi \otimes (1-\chian)f) = 0.  
\end{equation}
As $\varphi \otimes f \in L^2(S^\k \times \M, \muk)$, we have $\varphi \otimes (1-\chian)f \in L^2(S^\k \times \M, \muk)$ for any $n \in \N$. Then by Lemma \ref{L:41a} we can show that
\begin{equation} \label{:41p}
\lim_{n \to \infty}\| \varphi \otimes (1-\chian)f \|_{L^2(S^\k \times \M, \muk)} = 0 
\end{equation}
By \eqref{:41o} and \eqref{:41p}, we have $\lim_{n \to \infty}\| \varphi \otimes (1-\chian)f \|_{1} = 0$. Thus, the proof is complete.  
\end{proof}

\begin{proof}[Proof of Theorem \ref{T:31a}]
Recall that $\M[\bd{a}]$ is a compact set for any $\bd{a} \in \bd{A}$ with the vague topology. As $\supp \chian \subset \M[2(\bd{a}_n)^+]$ we see that $\chian f$ has a compact support for any $f \in \Dc$. We set 
\begin{equation*}
\mathcal{D}_{{\rm cut}}^{\k} = \text{$\{ \varphi \otimes \chian f; f \in \Dc$ and $\varphi \in C_0^{\infty}(S^\k)$ with $\varphi \otimes f \in \Dinfmuk, n \in \N \}$. }
\end{equation*}
Let $\{ K_n \}_{n=1}^{\infty}$ be the sequence of compact sets of $S^\k$ such that $S^\k = \cup_{n=1}^{\infty} K_n$. We set $\mathcal{D}^\k(n) = \{ \f \in \Dmuk; \text{$\f=0$ a.e. $(\bd{x}_\k, \xi)$ on $(K_n \times \M[2(\bd{a}_n)^+])^c$} \}$. By the definition of $\chian$ we can see that
\begin{equation}\label{:41q}
\mathcal{D}_{{\rm cut}}^{\k} \subset \bigcup_{n=1}^{\infty} \mathcal{D}^\k(n). 
\end{equation}
Additionally, from Lemma \ref{L:41c}, we have that $\mathcal{D}_{{\rm cut}}^{\k}$ is dense in $\Dinfmuk$ with respect to the norm $\| \cdot \|_1$. 
Combining this with \eqref{:41q}, we can see that $\{ K_n \times \M[2(\bd{a}_n)^+] \}_{n \in \N}$ is a compact nest. 
We can then prove Lemma \ref{T:31a} using a similar argument to that in \cite[Theorem 2.1]{E.19}. Thus, the proof is complete. 
\end{proof}

Next, we prove Theorem \ref{T:31b}. 
For $\m \geq \n$, $\ell > r$, $\k\in\N$ and $\eta\in\M$,
we introduce the measure $\muxnm$ on $\Ur^{\n} \times (U_\ell \setminus \Ur)^{\m-\n}$ determined by 
$$
\int_{\M_{r,\n} \cap \M_{\ell,\m}}f(\eta)(\mu_{\x_\k})_{\ell, \eta}^{\m}(d\zeta
)=\int_{U_{r}^{\n} \times (U_{\ell}\setminus U_{r})^{\m-\n}}f_{\ell,\eta}^{\m}(\y_{\m})
\muxnm(d\y_{\m})
$$
for any bounded measurable function $f$ on $\M$.
Here, $f_{\ell, \eta}^{\m}$ is the $U_\ell^{\m}$-representation of $f$,
 and $ (\mu_{\x_\k})_{\ell, \eta}^{\m}$ is the regular conditional probability defined by
$$
(\mu_{\x_\k})_{\ell, \eta}^{\m}(d\zeta) = \mu_{\x_\k}(\pi_\ell \in d\zeta | \pi_\ell^c (\eta)= \pi_\ell^c(\zeta), \zeta(U_\ell)=\m).
$$
Set $(\mu_{\x_\k})_{\ell}^{\m}(\cdot) = \mu_{\x_\k}(\cdot | \M_{\ell, \m})$ for $\ell, \m \in \N$. 
By (\ref{:2h}), 
$\muxnm$ has a density $\chesig_{\x_\k, r,\ell, \eta}^{\n,\m}(\y_{\m})$ with respect to $e^{-\mathcal{H}_\ell(\y_{\m}|\x_\k)}d\y_{\m}$ for $(\mu_{\x_\k})_{\ell}^{\m}$-a.s. $\eta$.
Here, $d\y_{\m}$ denotes the Lebesgue measure on $U_{\ell}^{\m}$ and 
$$ 
\mathcal{H}_{\ell}(\y_{\m}|\x_\k) = \sum_{i=1}^{\m} \Phi(y_\m^i) + \sum_{1 \leq i<j \leq \m} \Psi(y_\m^i, y_\m^j) + \sum_{i=1}^{\m} \sum_{j=1}^\k \Psi(y_\m^i, x_\k^j). 
$$ 
For $q \in \N$, set $p_q(x, y)=p(x, y)\1 (p(x, y)\le q )$. It is readily seen that
$\{ p_q \}$ is nondecreasing sequence in $q$ such that $\lim_{q\to\infty}p_q(x, y) = p(x, y)$ for a.e. $x, y \in S$. 
We set
\begin{align*}
&\Gamma_q(\x_\k) = \1 \Big( q^{-1} \le \rho^\k(\x_\k) \le q \Big) 
\\
&\Gamma_q(z| \x_\k, \eta) = \prod_{j=1}^\k \1 \Big( q^{-1} \leq \frac{\rho^\k(\x_\k)}{\rho^\k(\x_\k^{j,z})} \frac{d\mu_{\x_\k}}{d\mu_{\x_\k^{j,z}}}(\eta) \leq q \Big) 
\\
&\nabla_{q, \eta}^{[\k]}[\varphi, \psi](\x_\k)
= \Gamma_q(\x_\k) \frac{1}{2} \sum_{j=1}^{\k} \int_{U_{\ell}} \Gamma_q(z| \x_\k, \eta) \nadeffsa{j}{z}\varphi(\x_\k) \nadeffsa{j}{z}\psi(\x_\k) p_q(x_\k^j, z)dz,
\end{align*}
where $\nabla_j^z$ and $\x_{\k}^{j,z}$ are introduced in \eqref{:31b} and \eqref{:31c}, respectively.
For $\varphi, \psi \in C_0^{\infty}(S^\k)$ and $f, g \in \Dc$ with $\varphi \otimes f, \psi \otimes g \in \Dinfmuk$ we define 
\begin{align*}
&\mathcal{E}_{r, \ell, q}^{(1), \n, \m}(\varphi \otimes f, \psi \otimes g) 
= \int_{U_{\ell}^\k}\rho^\k(\x_\k) d\x_\k \int_{\M} (\mu_{\x_\k})_{\ell}^\m(d\eta)
\\&\quad \times 
\int_{U_{r}^{\n} \times (U_{\ell}\setminus U_{r})^{\m-\n}}  \nabla_{q, \eta}^{[\k]}[\varphi, \psi](\x_\k) f_{r, \eta}^{\m}(\y_{\m})g_{r, \eta}^{\m}(\y_{\m}) \muxnm(d\y_{\m}). 
\end{align*}

Then we can prove the following lemma. 

\begin{lemma} \label{L:41c_2}
Let $\k\in\N$. Assume that \As[QG] and \As[EQ] hold. Then, $(\mathcal{E}_{r, \ell, q}^{(1), \n, \m}, \Dinfmuk)$ is closable on $L^2(S^\k \times \M_{\ell, \m}, (\muk)_{\ell}^{\m})$ for $\m \ge \n$, $\ell \ge r$, $q\in\N$, where $(\muk)_{\ell}^{\m}$ is given as $(\muk)_{\ell}^{\m}(\cdot ) \equiv \muk(\cdot | S^{\k} \times \M_{\ell, \m})$. 
\end{lemma}
\begin{proof}
We calculate $\mathcal{E}_{r, \ell, q}^{(1), \n, \m}$. Then, for $\varphi \in C_0^{\infty}(S^\k)$ and $f \in \Dc$ with $\varphi \otimes f \in \Dinfmuk$, we see that
\begin{align}\notag
&\mathcal{E}_{r, \ell, q}^{(1), \n, \m}(\varphi \otimes f, \varphi \otimes f)  
\\ \notag
&= \frac{1}{2}\sum_{j=1}^{\k} \int_{U_{\ell}^\k} \Gamma_q(\x_\k) \rho^\k(\x_\k) d\x_\k \int_{U_{\ell}} (\varphi(\x_\k^{j, z}) - \varphi(\x_\k))^2 p_q(x_\k^j, z)dz  
\\ \notag
&\quad \times \int_{\M_{r, \n} \cap \M_{\ell, \m}} \Gamma_q(z| \x_\k, \eta) f(\eta)^2 \mu_{\x_\k}(d\eta) \\
&\leq q\sum_{j=1}^{\k} \int_{U_{\ell}^{\k+1}} \Gamma_q(\x_\k) \varphi(\x_\k^{j,z})^2 \frac{\rho^\k(\x_\k)}{\rho^\k(\x_\k^{j,z})}\rho^\k(\x_\k^{j,z}) d\x_\k dz  \notag 
\\ \notag
&\quad \quad \times \int_{\M_{r, \n} \cap \M_{\ell, \m}} \Gamma_q(z| \x_\k, \eta) f(\eta)^2 \frac{d\mu_{\x_\k}}{d\mu_{\x_\k^{j,z}}}(\eta) \mu_{\x_\k^{j,z}}(d\eta)  
\\ \label{:41r}
&\quad + \k q|U_{\ell}| \int_{S^\k} \varphi(\x_\k)^2 \rho^\k(\x_\k) d\x_\k \int_{\M_{r, \n} \cap \M_{\ell, \m}} f(\eta)^2 \mu_{\x_\k}(d\eta). 
\end{align}
Using the definitions of $\Gamma_q(\x_\k)$ and $\Gamma_q(z| \x_\k, \eta)$, we can see that the right-hand side of \eqref{:41r} is bounded by
$$ \k (q^3+q)|U_{\ell}| \int_{S^\k \times \M_{\ell, \m}} (\varphi \otimes f)^2(\x_\k, \eta) \mu^{[\k]}(d\x_\k d\eta). $$
Hence we have 
\begin{equation} \label{:41s}
\mathcal{E}_{r, \ell, q}^{(1), \n, \m}(\varphi \otimes f, \varphi \otimes f) \leq \k (q^3+q)|U_{\ell}| \| \varphi \otimes f \|_{L^2(S^\k \times \M_{\ell, \m}, \muk)}. 
\end{equation}
Let $\{ \varphi_n \otimes f_n \}$ be a $\mathcal{E}_{r, \ell, q}^{(1), \n, \m}$-Cauchy sequence in $\Dinfmuk$ such that 
$$
\displaystyle \lim_{n \to \infty} \| \varphi_n \otimes f_n \|_{L^2(S^\k \times \M_{\ell, \m}, (\muk)_{\ell}^{\m})} = 0.
$$ 
Then by \eqref{:41s} we have that 
\begin{equation*}
\lim_{n \to \infty} \mathcal{E}_{r, \ell, q}^{(1), \n, \m}(\varphi_n \otimes f_n, \varphi_n \otimes f_n) = 0. 
\end{equation*}
Hence, $(\mathcal{E}_{r, \ell, q}^{(1), \n, \m}, \Dinfmuk)$ is closable on $L^2(S^\k \times \M_{\ell, \m}, (\muk)_{\ell}^{\m})$. Thus, the proof is complete.
\end{proof}

We use the following lemma, which is given in \cite[Lemma 2.1(1),(2)]{O96}(see also \cite[Prop. I.3.7]{MR}). 
\begin{lemma} \label{L:41d} {\rm (\cite[Lemma 2.1(1),(2)]{O96})} 
Let $\{ (\mathfrak{E}^{(n)}, \mathfrak{D}^{(n)}) \}_{n \in \N}$ be a sequence of positive-definite, symmetric bilinear forms on $L^2(\M, \mu)$. \\
{\rm (i)} \  Suppose that $(\mathfrak{E}^{(n)}, \mathfrak{D}^{(n)})$ is closable for any $n \in \N$ and $\{ (\mathfrak{E}^{(n)}, \mathfrak{D}^{(n)}) \}$ is increasing, i.e. for any $n \in \N$ we have $\mathfrak{D}^{(n)} \supset \mathfrak{D}^{(n+1)}$ and $\mathfrak{E}^{(n)}(f, f) \leq \mathfrak{E}^{(n+1)}(f, f)$ for any $f \in \mathfrak{D}^{(n+1)}$. Let $\mathfrak{E}^{\infty}(f, f) = \lim_{n \to \infty} \mathfrak{E}^{(n)}(f, f)$ with the domain $\mathfrak{D}^{\infty} = \{ f \in \cap_{n \in \N} \mathfrak{D}^{(n)}; \sup_{n \in \N}\mathfrak{E}^{(n)}(f, f) < \infty \}$. Then $(\mathfrak{E}^{\infty}, \mathfrak{D}^{\infty})$ is closable on $L^2(\M, \mu)$. \\
{\rm (ii)} \  In addition to assumption (i), assume that $(\mathfrak{E}^{(n)}, \mathfrak{D}^{(n)})$ are closed. Then, $(\mathfrak{E}^{\infty}, \mathfrak{D}^{\infty})$ is closed. 
\end{lemma}

We set $\mathcal{E}_{r, \ell}^{(1), \n, \m} = \lim_{q \to \infty} \mathcal{E}_{r, \ell, q}^{(1), \n, \m}$. Using Lemma \ref{L:41d} we can state the following lemma. 

\begin{lemma} \label{L:41e}
Suppose that $(\mathcal{E}_{r, \ell, q}^{(1), \n, \m}, \Dinfmuk)$ is closable on $L^2(S^\k \times \M_{\ell, \m}, (\muk)_{\ell}^{\m})$ for $\m \ge \n$, $\ell \ge r$, $\k, q\in\N$. Then $(\mathcal{E}_{r, \ell}^{(1), \n, \m}, \Dinfmuk)$ is closable on $L^2(S^\k \times \M_{\ell, \m}, (\muk)_{\ell}^{\m})$ for $\m \ge \n$, $\ell \ge r$, $\k\in\N$.
\end{lemma}

From Lemmas \ref{L:41c_2} and \ref{L:41e}, we have the following lemma. 

\begin{lemma} \label{L:41f}
Let $\k\in\N$. Assume that \As[QG] and \As[EQ] hold. Then $(\mathcal{E}_{r, \ell}^{(1), \n, \m}, \Dinfmuk)$ is closable on $L^2(S^\k \times \M_{\ell, \m}, (\muk)_{\ell}^{\m})$ for $\m \ge \n$, $\ell \ge r$, $k\in\N$.
\end{lemma}

For $\mathbf{u}_\n=(u_\n^1, \dots, u_\n^\n)$, we set
\begin{align*}
&\Lambda_q(\mathbf{u}_{\n})= \prod_{i=1}^{\n}\1 
\Big( |\Phi(u_\n^i)+\sum_{\substack{j \neq i \\ 1 \leq j \leq n}} \Psi(u_\n ^i, u_\n ^j)|\le q \Big), 
\\
&\Lambda_q(y|\mathbf{u}_{\n}) = \1\Big( |\Phi(y)+\sum_{j =1}^{\n} \Psi(y, u_\n ^j)|\le q \Big). 
\end{align*}
It is readily seen that $\{ \Lambda_q(\x_{\n}) \}$, $\{ \Lambda_q(y|\x_{\n}) \}$ are nondecreasing sequences in $q$. 
We set
\begin{align*}
&\Dqban{1}[f, g](\y_{\m}, \x_\k) \\
&= \Lambda_q(\y_{\m}, \x_\k) \sum_{j=1}^{\n} \int_{U_{r}} \Lambda_q(z|\y_{\m}^{\hiku{j}}, \x_\k)\nadeffsa{j}{z}f_{\ell, \eta}^{\m}(\y_{\m}) \cdot \nadeffsa{j}{z}g_{\ell, \eta}^{\m}(\y_{\m}) p_q(y_\m ^j, z) dz, \\
&\Dqban{2}[f, g](\y_{\m}, \x_\k) \\
&= \Lambda_q(\y_{\m}, \x_\k) \sum_{j=1}^{\n} \int_{U_{\ell}\setminus U_{r}} \Lambda_q(z|\y_{\m}^{\hiku{j}}, \x_\k) \nadeffsa{j}{z}f_{\ell, \eta}^{\m}(\y_{\m}) \cdot \nadeffsa{j}{z}g_{\ell, \eta}^{\m}(\y_{\m}) p_q(y_\m ^j, z) dz, \\
&\Dqban{3}[f, g](\y_{\m}, \x_\k) \\
&= \Lambda_q(\y_{\m}, \x_\k) \sum_{j=\n+1}^{\m}
\int_{U_{r}} \Lambda_q(z|\y_{\m}^{\hiku{j}}, \x_\k) \nadeffsa{j}{z}f_{\ell, \eta}^{\m}(\y_{\m}) \cdot  \nadeffsa{j}{z}g_{\ell, \eta}^{\m}(\y_{\m}) p_q(y_\m ^j, z)dz, 
\end{align*}
where $\y_{\m}^{\hiku{j}}= (y_\m^1, \ldots, y_\m^{j-1}, y_\m^{j+1}, \ldots, y_\m^{\m}) \in S^{\m-1}$, and 
\begin{align*}
\nadeffsa{j}{z}f_{r, \eta}^{\m}(\y_{\m}) &= \begin{cases} f_{r, \unl (z,\eta)}^{{\m}-1}(\y_{\m}^{\hiku{j}}) - f_{r, \eta}^{\m}(\y_{\m}), \quad \text{if $\y_{\m} \in U_{r}^{\m}$, $z \notin U_{r}$, }
\\
f_{r, \eta}^{{\m}}(\y_{\m}^{j,z}) - f_{r, \eta}^{\m}(\y_{\m}), \quad \text{if $\y_{\m} \in U_{r}^{\m}$, $z \in U_{r}$, }
\end{cases}
\end{align*}
for $j=1,2,\dots, \m$.
 We set $\Dq = \Dqban{1}+\Dqban{2}+\Dqban{3}$. 
For $\varphi, \psi \in C_0^{\infty}(S^\k)$ and $f, g \in \Dc$ with $\varphi \otimes f, \psi \otimes g \in \Dinfmuk$ we set 
\begin{align*}
\mathcal{E}_{\x_\k, r, \ell, q, \eta}^{(2), \n, \m}(\varphi \otimes f, \psi \otimes g) &= \int_{U_{r}^{\n} \times (U_{\ell}\setminus U_{r})^{\m-\n}} \varphi(\x_\k) \psi(\x_\k) \Dq[f, g](\y_{\m}, \x_\k) \muxnm(d\y_{\m}) \\
\mathcal{E}_{\x_\k, r, \ell, \eta}^{(2), \n, \m}(\varphi \otimes f, \psi \otimes g) &= \lim_{q \to \infty} \mathcal{E}_{\x_\k, r, \ell, q, \eta}^{(2), \n, \m}(\varphi \otimes f, \psi \otimes g). 
\end{align*}
Then we can prove the following lemma. 

\begin{lemma} \label{L:41g}
Let $\k\in\N$. Assume that \As[QG] and \As[EQ] hold. Then $(\mathcal{E}_{\x_\k, r, \ell, q, \eta}^{(2), \n, \m}, \Dinfmuk)$ is closable on $L^2(S^\k \times \M_{\ell, \m}, \delta_{\x_\k} \otimes (\mu_{\x_\k})_{\ell, \eta}^{\m})$ for $\m \ge \n$, $\ell \ge r$, $\k, q\in\N$, $d\x_\k$-a.e. $\x_\k$ and $(\mu_{\x_\k})_{\ell}^{\m}$-a.e. $\eta$. 
\end{lemma}

\begin{proof}
By the property of $\muxnm$ and using a similar approach as for the proof of \cite[Lemma 4.1]{E.19} we have 
\begin{equation} \label{:41t}
\mathcal{E}_{\x_\k, r, \ell, q, \eta}^{(2), \n, \m}(\varphi \otimes f, \varphi \otimes f) \leq 2(\cref{qg}^2e^{2q}+1)\m q |U_{\ell}| \| \varphi \otimes f \|_{L^2(S^\k \times \M_{\ell, \m}, \delta_{\x_\k} \otimes (\mu_{\x_\k})_{\ell, \eta}^{\m})}^2, 
\end{equation}
where $\cref{qg}$ is the constant in \eqref{:2h}. 
Let $\{ \varphi_n \otimes f_n \}$ be a $\mathcal{E}_{\x_\k, r, \ell, q, \eta}^{(2), \n, \m}$-Cauchy sequence in $\Dinfmuk$ such that $\displaystyle \lim_{n \to \infty} \| \varphi_n \otimes f_n \|_{L^2(S^\k \times \M_{\ell, \m}, \delta_{\x_\k} \otimes (\mu_{\x_\k})_{\ell, \eta}^{\m})} = 0$. Then, \eqref{:41t} yields 
\begin{equation*}
\lim_{n \to \infty} \mathcal{E}_{\x_\k, r, \ell, q, \eta}^{(2), \n, \m}(\varphi_n \otimes f_n, \varphi_n \otimes f_n) = 0 \quad \text{for $(\mu_{\x_\k})_{\ell}^{\m}$-a.s. $\eta$.} 
\end{equation*}
Hence, $(\mathcal{E}_{\x_\k, r, \ell, q, \eta}^{(2), \n, \m}, \Dinfmuk)$ is closable on $L^2(S^\k \times \M_{\ell, \m}, \delta_{\x_\k} \otimes (\mu_{\x_\k})_{\ell, \eta}^{\m})$. Thus, the proof is complete. 
\end{proof}

Using Lemma \ref{L:41d}, we can state the following lemma. 

\begin{lemma} \label{L:41h}
Suppose that $(\mathcal{E}_{\x_\k, r, \ell, q, \eta}^{(2), \n, \m}, \Dinfmuk)$ is closable on $L^2(S^\k \times \M_{\ell, \m}, \delta_{\x_\k} \otimes (\mu_{\x_\k})_{\ell, \eta}^{\m})$ for $\m \ge \n$, $\ell \ge r$, $\k, q\in\N$, $d\x_\k$-a.e. $\x_\k$ and $(\mu_{\x_\k})_{\ell}^m$-a.e. $\eta$. Then, 
$(\mathcal{E}_{\x_\k, r, \ell, \eta}^{(2), \n, \m}, \Dinfmuk)$ is closable on $L^2(S^\k \times \M_{\ell, \m}, \delta_{\x_\k} \otimes (\mu_{\x_\k})_{\ell}^{\m})$ for $\m \ge \n$, $\ell \ge r$, $\k\in\N$, $d\x_\k$-a.e. $\x_\k$ and $(\mu_{\x_\k})_{\ell}^\m$-a.e. $\eta$.
\end{lemma}

For $r, \ell, \k, \n, \m \in \N$ and $\x_\k \in S^\k$, we set
\begin{align*}
\mathcal{E}_{\x_\k, r, \ell}^{(2), \n, \m}(\varphi \otimes f, \psi \otimes g) &= \int_{\M} \mathcal{E}_{\x_\k, r, \ell, \eta}^{(2), \n, \m}(\varphi \otimes f, \psi \otimes g) (\mu_{\x_\k})_{\ell}^\m(d\eta), 
\\
\mathcal{E}_{r, \ell}^{(2), \n, \m}(\varphi \otimes f, \psi \otimes g) &= \int_{S^\k} \mathcal{E}_{\x_\k, r, \ell}^{(2), \n, \m}(\varphi \otimes f, \psi \otimes g) \rho^\k(\x_\k) d\x_\k.
\end{align*}
In the same way as the proof of \cite[Lemma 4.3]{E.19}, we have the following lemmas. 

\begin{lemma} \label{L:41k}
Suppose that $(\mathcal{E}_{\x_\k, r, \ell, \eta}^{(2), \n, \m}, \Dinfmuk)$ is closable on $L^2(S^\k \times \M_{\ell, \m}, \delta_{\x_\k} \otimes (\mu_{\x_\k})_{\ell, \eta}^{\m})$ for $\m \ge \n$, $\ell \ge r$, $\k\in\N$, $d\x_\k$-a.e. $\x_\k$ and $(\mu_{\x_\k})_{\ell}^{\m}$-a.e. $\eta$. 
Then, $(\mathcal{E}_{\x_\k, r, \ell}^{(2), \n, \m}, \Dinfmuk)$ is closable on $L^2(S^\k \times \M_{\ell, \m}, \delta_{\x_\k} \otimes (\mu_{\x_\k})_{\ell}^{\m})$ for $\m \ge \n$, $\ell \ge r$, $\k\in\N$, and $d\x_\k$-a.e. $\x_\k$.
\end{lemma}

\begin{lemma} \label{L:41l}
Suppose that $(\mathcal{E}_{\x_\k, r, \ell}^{(2), \n, \m}, \Dinfmuk)$ is closable on $L^2(S^\k \times \M_{\ell, \m}, \delta_{\x_\k} \otimes (\mu_{\x_\k})_{\ell}^{\m})$ for $\m \ge \n$, $\ell \ge r$, $\k\in\N$, and $d\x_\k$-a.e. $\x_\k$. Then, $(\mathcal{E}_{r, \ell}^{(2), \n, \m}, \Dinfmuk)$ is closable on $L^2(S^\k \times \M_{\ell, \m}, (\muk)_{\ell}^{\m})$ for $\m \ge \n$, $\ell \ge r$, $k\in\N$. 
\end{lemma}

From Lemmas \ref{L:41g}--\ref{L:41l}, we have the following lemma. 

\begin{lemma} \label{L:41m}
Let $\k\in\N$. Assume that \As[QG] and \As[EQ] hold. Then  $(\mathcal{E}_{r, \ell}^{(2), \n, \m}, \Dinfmuk)$ is closable on $L^2(S^\k \times \M_{\ell, \m}, (\muk)_{\ell}^{\m})$ for $\m \ge \n$, $\ell \ge r$, $k\in\N$. 
\end{lemma}

We set $\mathcal{E}_{r, \ell}^{\n, \m, [\k]} = \mathcal{E}_{r, \ell}^{(1), \n, \m} + \mathcal{E}_{r, \ell}^{(2), \n, \m}$. From Lemmas \ref{L:41f} and \ref{L:41m}, we have the following lemma. 

\begin{lemma} \label{L:41n}
Let $\k\in\N$. Assume that \As[QG] and \As[EQ] hold. Then, $(\mathcal{E}_{r, \ell}^{\n, \m, [\k]}, \Dinfmuk)$ is closable on $L^2(S^\k \times \M_{\ell, \m}, (\mu^{[\k]})_{\ell}^{\m})$ for $\m \ge \n$, $\ell \ge r$, and $\k\in\N$. 
\end{lemma} 

\begin{proof}[Proof of Theorem \ref{T:31b}]
We set 
\begin{align*}
\mathcal{E}_r^{[\k]} (\varphi \otimes f, \varphi \otimes f) &= \lim_{\ell \to \infty} \sum_{\m = 0}^{\infty} \sum_{\n = 0}^{\m} \mathcal{E}_{r, \ell}^{\n, \m, [\k]}(\varphi \otimes f, \varphi \otimes f), \\
\mathcal{E}_{\infty}^{[\k]} (\varphi \otimes f, \varphi \otimes f) &= \lim_{r \to \infty} \mathcal{E}_r^{[\k]} (\varphi \otimes f, \varphi \otimes f). 
\end{align*}
Note that $(\mathcal{E}_r^{[\k]}, \Dinfmuk)$ is increasing in $r$. Set
$$ 
\mathcal{D}_{\infty}^{[\k]} = \Big\{ \varphi \otimes f \in \Dinfmuk; \sup_{r \in \N}\mathcal{E}_r^{[\k]}(\varphi \otimes f, \varphi \otimes f) < \infty \Big\}. 
$$
From Lemma \ref{L:41d} $(\mathcal{E}_\infty^{[\k]}, \mathcal{D}_{\infty}^{[\k]})$ is closable.
Additionally, for any $\varphi \in C_0^{\infty}(S^\k)$ and $f \in \Dc$ with $\varphi \otimes f \in \Dinfmuk$,
there exists some $r\in\N$ such that
\begin{equation*} 
\mathcal{E}_{\infty}^{[\k]} (\varphi \otimes f, \varphi \otimes f) = 
\mathcal{E}_{r}^{[\k]} (\varphi \otimes f, \varphi \otimes f) =\Emuk(\varphi \otimes f, \varphi \otimes f)< \infty, 
\end{equation*}
and  so $\mathcal{D}_{\infty}^{[\k]} = \Dinfmuk$. 
Therefore $(\Emuk, \Dinfmuk)$ is closable and the proof is complete. 
\end{proof}

\subsection{Proof of Theorem \ref{T:31c}} \label{SS:42}

In this section, we prove Theorem \ref{T:31c}.

\begin{lemma}\label{L:42a}
Let $\k\in \N$ and $\mathfrak{A} \subset S^\k \times \M$ be such that $(\unl^{[\k]})^{-1}(\unl^{[\k]}(\mathfrak{A})) = \mathfrak{A}$. Then, $\Capa^{\mu}(\unl^{[\k]}(\mathfrak{A}))=0$ implies $\Capa^{\muk}(\mathfrak{A})=0$. (Here, $\unl^{[\k]}$ is defined in Section 3.1.)
\end{lemma}
\begin{proof}
This lemma is proved in the same way as \cite[Lemma 4.1]{o.tp}. 
\end{proof}

For a process $\{Z_t\}_{t\ge 0}$ and a random time $\sigma\ge 0$, we set 
$Z^{\sigma}_{t} = Z_{\sigma \wedge t}$, and 
 $$
 \P_{\xi}^{\sigma}(\Xi\in \cdot) = \P_{\xi}(\Xi^{\sigma} \in \cdot), 
 \quad 
 \P_{(\x_\k, \eta)}^{[\k], \sigma}((\X^{[\k]},\Xi^{[\k]})\in\cdot ) 
 = \P_{(\x_\k, \eta)}^{[\k]}((\X^{[\k], \sigma}, \Xi^{[\k], \sigma}) \in \cdot).
$$
We set  
\begin{align*}
\sigma_{r} &= \inf \{ t>0; \text{$\Xi_t(\Ur) \neq \Xi_{t-}(\Ur)$ or $\Xi_t(\partial \Ur) \geq 1$} \}, \\
\sigma_{r}^{[\k]} &= \inf \{ t>0 ; \text{$\X^{[\k]}_t \notin \Ur^\k$ or $\Xi^{[\k]}_t(\Ur) \neq \Xi^{[\k]}_{t-}(\Ur)$ or $\Xi^{[\k]}_t(\partial \Ur) \geq 1$} \}. 
\end{align*}
Then $\P_{\xi}^{\sigma_{r}}$ and $\P_{(\x_\k, \eta)}^{[\k], \sigma_{r}^{[\k]}}$ are strong Markov processes. Note that the  sets 
\begin{align*}
\{ \xi; \xi(\partial U_r) = 0, \xi(U_r) = m \} &\subset \M, \\
\{ (\x_{\k}, \xi) ; \ \x_{\k} \in S^{\k}, \xi(\partial U_r) = 0, \xi(U_r) = m\} &\subset S^{\k} \times \M,
\end{align*}
are open for any $m \in \N$. Then, we can see that $\P_{\xi}^{\sigma_{r}}$ and $\P_{(\x_\k, \eta)}^{[\k], \sigma_{r}^{[\k]}}$ are associated with parts of the Dirichlet forms $(\Emu, \Dmu)$ and $(\Emuk, \Dmuk)$, respectively. 
We denote the capacity associated with the process $\P_{\xi}^{\sigma_{r}}$ by $\Capa^{\mu, \sigma_{r}}$. 

The following lemma is a modification of \cite[Lemma 4.2]{o.tp}, and the proof is obtained using the same procedure.

\begin{lemma} \label{L:42b} 
Suppose that the assumptions in Theorem \ref{T:31c} hold. Then there exists $\mathfrak{A}_{r}^\k \subset \Ur^\k \times \M$ such that 
\begin{align} \notag 
&(\Ur^\k \times \M) \cap (\unl^{[\k]})^{-1}\left( \unl^{[\k]}(\mathfrak{A}_{r}^\k) \right) = \mathfrak{A}_{r}^\k, \quad \unl^{[\k]}(\mathfrak{A}_{r}^\k) \subset \Msi, 
\\ \label{:42b}
&\Capa^{\mu, \sigma_{r}}\left( \unl^{[\k]}(\Ur^\k \times \M) \setminus \unl^{[\k]}(\mathfrak{A}_{r}^\k) \right) = 0, 
\\ \notag 
&\P^{\mu, \sigma_{r}}_{\unl^{[\k]}(\x_\k, \eta)}\left( \Xi_t \in \text{$\unl^{[\k]}(\mathfrak{A}_{r}^\k)$ for all $t<\sigma_{r}$} \right) = 1 \quad \text{for all $(\x_\k, \eta) \in \mathfrak{A}_{r}^\k$}, 
\\ \notag 
&\P^{\sigma_{r}}_{\unl^{[\k]}(\x_\k, \eta)} = \P^{[\k], \sigma_{r}^{[\k]}}_{(\x_\k, \eta)} \circ (\unl^{[\k]})^{-1} \quad \text{for all $(\x_\k, \eta) \in \mathfrak{A}_{r}^\k$}. 
\end{align}
\end{lemma}

We fix $\e>0$ and $r>0$, and introduce smooth functions $\psi_+$ and $\psi_- : S \to [0, \infty)$ such that 
\begin{equation}\label{:42e}
\psi_+(x) = \begin{cases} 1, & \text{$|x| \leq r$, } \\
\in (0,1), & \text{$r < |x| < r+\e$, } \\
0, & \text{$|x| \ge r+\e$, } 
\end{cases}
\quad 
\psi_-(x) = \begin{cases} 1, & \text{$|x| \leq r-\e$, } \\
\in (0,1), & \text{$r-\e < |x| < r$, } \\
0, & \text{$|x| \ge r$. } 
\end{cases}
\end{equation}
In the following, we treat $\psi_+$ and $\psi_-$ together and use the notation $\psi_{\pm}$. 
In addition, we set $\tilde{\psi}_{\pm}(\xi) = \int_S \psi_{\pm}(x) \xi(dx)$, respectively. Note that the integrations in the definitions of $\tilde{\psi}_{\pm}$ are finite. Thus, $\tilde{\psi}_{\pm}$ are local functions on $\M$. By a straightforward calculation, we have the following for $f \in \Dc$
\begin{align*}
\Emu(f, \tilde{\psi}_{\pm}) 
&= -\int_{\M} \mu(d\xi) f(\xi) \int_S \xi(dx) \int_S (\tilde{\psi}_{\pm}(\xi^{x, y})-\tilde{\psi}_{\pm}(\xi))c(x, y, \xi)dy
\\
&=-\int_{\M} \mu(d\xi) f(\xi) \int_S \xi(dx) \int_S ({\psi_{\pm}}(y)-{\psi_{\pm}}(x))c(x, y, \xi)dy.
\end{align*}
We consider an additive functional of $\Xi_t=\sum_{j=1}^\infty \delta_{X_t^j}$ defined by  $A^{[\tilde{\psi}_{\pm}]}_t=\tilde{\psi}_{\pm}(\Xi_t)-\tilde{\psi}_{\pm}(\Xi_0)$. 
By Fukushima's decomposition, we have $A^{[\tilde{\psi}_{\pm}]}_t = M^{[\tilde{\psi}_{\pm}]}_t + N^{[\tilde{\psi}_{\pm}]}_t$, where $M^{[\tilde{\psi}_{\pm}]}_t$ are the martingale additive functionals of finite energy and $N^{[\tilde{\psi}_{\pm}]}_t$ are the additive functionals of zero energy. We have
\begin{align*}
N^{[\tilde{\psi}_{\pm}]}_t 
&= \int_0^t ds \sum_{j\in\N} \int_S (\psi_{\pm}(y)-\psi_{\pm}(X_{s-}^j)) c(X_{s-}^j, y, \Xi_{s-}) dy,
\end{align*}
and $A^{[\tilde{\psi}_{\pm}]}$ are given by
\begin{equation*}
A^{[\tilde{\psi}_{\pm}]}_t = \int_0^t \int_{\N} \int_{S} \int_0^{\infty} 
(\psi_{\pm}( y)-\psi_{\pm}(X_{s-}^j)) \hat{a}(r, j, y, \Xi_{s-}) \hat{N}(dsdjdydr), 
\end{equation*}
where we set 
$$ 
\hat{a}(r, j, y, \xi) = \1\left\{ 0 \leq r \leq c(\lab(\xi)^j, y, \xi) \right\} 
$$
and $\hat{N}$ is the Poisson point process on $[0, \infty) \times \N \times S \times [0, \infty)$ with intensity $ds\zeta_{\N}(dj)dydr$ and $\zeta_{\N} = \sum_{j \in \N} \delta_{j}$ is the counting measure on $\N$. 
We consider another additive functional defined by
\begin{equation*}
\hat{A}^{[\tilde{\psi}_{\pm}]}_t = \int_0^t \int_{\N} \int_{S} \int_0^{\infty} 
\mathbf{1}_{[\varepsilon,\infty)}(|y-X_{s-}^j|)
|\psi_{\pm}( y)-\psi_{\pm}(X_{s-}^j)|  \hat{a}(r, j, y, \Xi_{s-}) \hat{N}(dsdjdydr).
\end{equation*}
Then we see that
\begin{align*}
\E_{\mu}[\hat{A}^{[\tilde{\psi}_{\pm}]}_t ]= t\int_{\M}\mu(d\xi) \nu_{\pm}(\xi)
\end{align*}
in the case where
\begin{align*} 
\nu_{\pm}(\xi) = \int_S \xi(dx) \int_{|x-y| \geq \e} |{\psi_{\pm}}(y)-{\psi_{\pm}}(x)| c(x, y, \xi)dy  
\end{align*}
are elements of $L^1(\M, \mu)$.
We can now prove the following lemma. 

\begin{lemma}\label{L:42c}
Suppose that the assumptions in Theorem \ref{T:31c} hold. Then $\nu_{\pm} \in L^1(\M, \mu)$ and $\E_{\mu}[\hat{A}^{[\tilde{\psi}_{\pm}]}_t ]<\infty$ for any $t>0$.
\end{lemma}

\begin{proof}
It is sufficient to show the lemma for $\nu_+$: the proof for $\nu_-$ uses the same method. 
From (\ref{:42e}) we can easily see that
\begin{align*}
\ \nu_+(\xi)  &\leq \int_{U_{r+\e}^c} \xi(dx) \int_{U_{r+\e} \cap \{ |x-y| \geq \e \}}c(x, y, \xi)dy  + \int_{U_r} \xi(dx) \int_{U_{r}^c \cap \{ |x-y| \geq \e \}}c(x, y, \xi)dy \\
&\quad + \int_{A_{r, \e}} \xi(dx) \int_{|x-y| \geq \e} c(x, y, \xi)dy \\
&=: I_1(\xi)+I_2(\xi)+I_3(\xi),  
\end{align*}
where we set $A_{r, \e}=\{ x \in S; r<|x| \leq r+\e \}$. 
By the definition of the reduced Palm measure, we have
\begin{align}\notag
&\int_{\M} I_1(\xi) \mu(d\xi) = \int_{U_{r+\e}^c} dx \rho^1(x) \int_{\M} \mu_x(d\eta) \int_{U_{r+\e} \cap \{ |x-y| \geq \e\}} 
c(x, y,\eta) dy  
\\ \notag
&= \int_{U_{r+\e}^c} dx \rho^1(x) \int_{\M} \mu_x(d\eta) \int_{U_{r+\e} \cap \{ |x-y| \geq \e\}} p(x, y) dy \\ \notag
&\quad + \int_{U_{r+\e}} dx \rho^1(x) \int_{\M} \mu_x(d\eta) \int_{U_{r+\e}^c \cap \{ |x-y| \geq \e\}} p(x, y) dy 
\\ \label{:42f}
&\leq \int_{U_{r+\e}^c} dx \rho^1(x) \int_{U_{r+\e} \cap \{ |x-y| \geq \e\}} p(x, y) dy + \int_{U_{r+\e}} dx \rho^1(x) \int_{U_{r+\e}^c \cap \{ |x-y| \geq \e\}} p(x, y) dy, 
\end{align}
where we use the reversibility, that is, $\frac{\rho^1(y)}{\rho^1(x)}\frac{d\mu_y}{d\mu_x} p(x, y) \rho^1(x)d\mu_xdx  = p(x, y) \rho^1(y)d\mu_ydy$ in the last line and we use the fact that $\mu_x$ is a probability measure on $\M$ in the last inequality.
By the assumptions \As[p.1], we can see that there exists some constant $C_{\e}>0$ such that 
\begin{equation}\label{:42g}
\int_{\{ |x-y| \geq \e \}}   p(x, y) dx \leq C_{\e}. 
\end{equation}
In addition, by \As[p.1] and \As[A.2] we obtain
\begin{equation*} 
\int_{U_r} dy\int_{\{ |x-y| \geq \e \}} \rho^1(x)p(x, y) dx < \infty, \quad \text{for any $r>0$}. \label{prCe}
\end{equation*}
Hence, we can show that
\begin{equation} \label{:42h}
\int_{U_{r+\e}^c} dx \rho^1(x) \int_{U_{r+\e} \cap \{ |x-y| \geq \e\}} p(x, y) dy
 = \int_{U_{r+\e}} dy\int_{U_{r+\e}^c \cap \{ |x-y| \geq \e \}} \rho^1(x) p(x, y) dx < \infty.
\end{equation}
In additiion, we have that
\begin{equation}\label{:42i}
\int_{U_{r+\e}} dx \rho^1(x) \int_{U_{r+\e}^c \cap \{ |x-y| \geq \e\}} p(x, y) dy \leq C_{\e} \int_{U_{r+\e}} \rho^1(x)dx < \infty,  
\end{equation}
where we use \eqref{:42g}. 
Thus, by \eqref{:42f}, \eqref{:42h} and \eqref{:42i} we have $\int_{\M} I_1(\xi) \mu(d\xi) < \infty$. In the same way, we can show that $\int_{\M} I_2(\xi) \mu(d\xi)<\infty$ and $\int_{\M} I_3(\xi) \mu(d\xi) < \infty$. Therefore we obtain $\nu_+ \in L^1(\M, \mu)$. This completes the proof.
\end{proof}

From the above lemma, we have
\begin{align*}
&\int_0^t \int_{\N} \int_{S} \int_0^{\infty}  \1\left\{ \left|\psi_{\pm}(y)-\psi_{\pm}(X_{s-}^j) \right| = 1 \right\} \hat{a}(r, j, y, \Xi_{s-}) \hat{N}(dsdjdydr) \leq \hat{A}^{[\tilde{\psi}_{\pm}]}_t < \infty
\end{align*}
for any $t \in [0, \infty)$ $\P_\mu$-a.s. Note that 
\begin{align*}
\left| \psi_+( y)-\psi_+(X_{s-}^j) \right| = 1 
&\Leftrightarrow [X_{s-}^j \in \Ur, y \in U_{r+\e}^c] \ \text{or} \ [X_{s-}^j \in U_{r+\e}^c, y \in \Ur],  \\
\left| \psi_-( y)-\psi_-(X_{s-}^j) \right| = 1 
&\Leftrightarrow [X_{s-}^j \in U_{r-\e}, y \in U_{r}^c] \ \text{or} \ [X_{s-}^j \in U_{r}^c, y \in U_{r-\e}].
\end{align*}
Hence, we have the following lemma. 
\begin{lemma} \label{L:42d}
Suppose that the assumptions in Theorem \ref{T:31c} hold.
For any $T \in \N$, we have
\begin{align*}
&\#\big\{ (s, j) \in [0, T] \times \N; [X_{s-}^j \in \Ur, X_s^j \in U_{r+\e}^c] \ \text{or} \ [X_{s-}^j \in U_{r+\e}^c, X_{s}^j \in \Ur] \\
&\quad \ \text{or} \ [X_{s-}^j \in U_{r-\e}, X_s^j \in U_{r}^c] \ \text{or} \ [X_{s-}^j \in U_{r}^c, X_{s}^j \in U_{r-\e}] \big\} < \infty, \ \text{$\P_\mu$-a.s.} 
\end{align*}
\end{lemma}
\vskip 3mm

We set $r(i)=r$ if $i$ is odd and $r(i)=r+1$ if $i$ is even. 
We set $\bar{\sigma}_0=\bar{\sigma}_0^{[\k]}=0$ and
for $i \in\N$ let
\begin{align*}
\bar{\sigma}_i &= \bar{\sigma}_i(\Xi) := \inf \{ t>\bar{\sigma}_{i-1};  \text{$\Xi_t(\Uri) \neq \Xi_{t-}(\Uri)$ or $\Xi_t(\partial \Uri) \geq 1$}\}, \\
\bar{\sigma}_i^{[\k]} &= \bar{\sigma}_i^{[\k]}(\X^{[\k]},\Xi^{[\k]}) \\
&:= \inf \{ t>\bar{\sigma}_{i-1}^{[\k]}; \text{$\X_t^{[\k]} \notin \Uri^\k$ or $\Xi_t^{[\k]}(\Uri) \neq \Xi_{t-}^{[\k]}(\Uri)$ or $\Xi^{[\k]}_t(\partial \Uri) \geq 1$}  \}.
\end{align*}
For $\Xi \in W(\Msi)$, we set $\lab_{\rm path}^{[\k]}(\Xi)= (\bar{\X}^{\k}, \bar{\Xi}^{[\k]})$.
For $T\in \N$ we set
$$
\Omega_i = \{ \Xi ; \bar{\X}_{\bar{\sigma}_i}^{[\k]} \in \Ur^\k \},  \quad 
\Omega_i^{[\k]} = \{ (\X^{[\k]}, \Xi^{[\k]}); \X_{\bar{\sigma}_i^{[\k]}}^{[\k]} \in \Ur^\k \}. 
$$ 
We set $\Omega_{\infty} = \bigcap_{i=1}^{\infty} \Omega_i$ and $\Omega_{\infty}^{[\k]} = \bigcap_{i=1}^{\infty} \Omega_i^{[\k]}$. In addition, we set $\bar{\sigma}_{\infty}(\omega) = \lim_{i \to \infty} \bar{\sigma}_i(\omega)$ on $\Omega_{\infty}$ and $\bar{\sigma}_{\infty}^{[\k]}(\omega) = \lim_{i \to \infty} \bar{\sigma}_i^{[\k]}(\omega)$ on $\Omega_{\infty}^{[\k]}$.

\begin{lemma} \label{L:42e}
Suppose that the assumptions in Theorem \ref{T:31c} hold. Let $\mathfrak{A}_r^\k$ be as in Lemma \ref{L:42b}. Then, for all $(\x_\k, \eta) \in \mathfrak{A}_r^\k$ the following holds:
\begin{enumerate}
\item $\bar{\sigma}_{\infty} =\infty$ for $\P_{\unl^{[\k]}(\x_\k, \eta)}(\cdot; \Omega_{\infty})$-a.s. 
\item $\bar{\sigma}_{\infty}^{[\k]} =\infty$ for $\P_{(\x_\k, \eta)}^{[\k]}(\cdot; \Omega_{\infty}^{[\k]})$-a.e. 
\end{enumerate}
\end{lemma}
\begin{proof}
Let $T\in\N$.
Assume that $\bar{\sigma}_{\infty} \le T$. By \As[NBJ], which is derived from Lemma \ref{L:53b} and \As[NCL], we see that $\mathcal{N}:=m_{r+1,T}(\lab_{\rm path}(\Xi))$ is finite $\P_{\mu}$-a.e.
Hence, if $n \geq \mathcal{N}$ then $X_t^n \notin U_{r+1}$ for any $t \in [0, T]$. Thus, it is sufficient to deal with $\bar{\X} =( X^j )_{j=1}^{\mathcal{N}}$. 
By definition, on $\Omega_{\infty}$ we see that $\bar{\X}_t^{[\k]} \in \Ur^\k$ for any $t < \bar{\sigma}_{\infty}$.
 Then the cardinality of the set of particles in $\{ X_t^j \}_{j=\k+1}^\mathcal{N}$ which jumps from $\Uri$ to $\Uri^c$ or from $\Uri^c$ to $\Uri$ at $\bar{\sigma}_i$ is infinite. 
From the assumption that $\bar{\sigma}_{\infty} \le T$,
 there exist $j_1, j_2 \in \{k+1,k+2,\dots, \mathcal{N}\}$ such that 
\begin{equation*}
\#\{ i \in \N; X_{\bar{\sigma}_{2i}(\Xi)}^{j_1} \neq X_{\bar{\sigma}_{2i}(\Xi)-}^{j_1} \} = \infty, \quad 
\#\{ i \in \N; X_{\bar{\sigma}_{2i-1}(\Xi)}^{j_2} \neq X_{\bar{\sigma}_{2i-1}(\Xi)-}^{j_2} \} = \infty.
\end{equation*}
Then we choose sequences 
$\{ i_{\ell}^1 \}_{\ell=1}^{\infty} =  \{ i \in 2\N; X_{\bar{\sigma}_{i}(\Xi)}^{j_1} \neq X_{\bar{\sigma}_{i}(\Xi)-}^{j_1} \}$ with $i_{\ell}^1 < i_k^1$ ($\ell < k$) and $\{ i_{\ell}^2 \}_{\ell=1}^{\infty} =  \{ i \in 2\N-1; X_{\bar{\sigma}_{i}(\Xi)}^{j_2} \neq X_{\bar{\sigma}_{i}(\Xi)-}^{j_2} \}$ with $i_{\ell}^2 < i_k^2$ ($\ell < k$). 
For any $\e > 0$ let $R^a(\e)$ be the cardinality of the set
\begin{align*}
&\big\{ \ell \in \N; X_{\bar{\sigma}_{i_{\ell}^a}(\Xi)-}^{j_a} \in U_{r(i_{\ell}^a)}, X_{\bar{\sigma}_{i_{\ell}^a}(\Xi)}^{j_a} \in U_{r(i_{\ell}^a)+\e}^c \ \text{or} \ X_{\bar{\sigma}_{i_{\ell}^a}(\Xi)-}^{j_a} \in U_{r(i_{\ell}^a)+\e}^c, X_{\bar{\sigma}_{i_{\ell}^a}(\Xi)}^{j_a} \in U_{r(i_{\ell}^a)} \\
&\quad \  \text{or} \ X_{\bar{\sigma}_{i_{\ell}^a}(\Xi)-}^j \in U_{r(i_{\ell}^a)-\e}, X_{\bar{\sigma}_{i_{\ell}^a}(\Xi)}^j \in U_{r(i_{\ell}^a)}^c \ \text{or} \ X_{\bar{\sigma}_{i_{\ell}^a}(\Xi)-}^j \in U_{r(i_{\ell}^a)}^c, X_{\bar{\sigma}_{i_{\ell}^a}(\Xi)}^j \in U_{r(i_{\ell}^a)-\e} \big\}
\end{align*}
for $a=1, 2$. If $R^a(\e)=\infty$, we have a contradiction of Lemma \ref{L:42d}. Thus, $R^a(\e)<\infty$ a.s. for any $\e > 0$. Hence, for any $\e > 0$ there exists $\ell_{\e} \in \N$ such that if $\ell > \ell_{\e}$, then 
\begin{equation} \notag 
X_{\bar{\sigma}_{i_{\ell}^a}(\Xi)}^{j_a} \in U_{r(i_{\ell}^a)+\e} \setminus  U_{r(i_{\ell}^a)-\e} \ \text{for $a=1, 2$}.
\end{equation}
Because $\Xi$ is a Hunt process, there exists an admissible filtration $\{ \mathcal{F}_t \}$ such that $\Xi$ is strong Markov and quasi-left-continuous with respect to $\{ \mathcal{F}_t \}$. In particular, $\{ \mathcal{F}_t \}$ is right-continuous. 
As $\{ \bar{\sigma}_{i_{\ell}^a}(\Xi) \}_{\ell > \ell_{\e}}$ is a sub-sequence of $\{ \bar{\sigma}_i \}$, we have that $\{ \bar{\sigma}_{i_{\ell}^a}(\Xi) \}_{\ell > \ell_{\e}}$ is a sequence of $\{ \mathcal{F}_t \}$-stopping times and $\lim_{\ell \to \infty} \bar{\sigma}_{i_{\ell}^a}(\Xi) = \bar{\sigma}_{\infty}$. In addition, we have $d(X_{\bar{\sigma}_{i_{\ell}^a}(\Xi)}^{j_a}, \partial U_{r(i_{\ell}^a)}) \leq \e$ for $\ell > \ell_{\e}$, where we set $d(x, A) := \inf_{y \in A}|x-y|$ for a set $A$. Hence, by the quasi-left-continuity of $\Xi$, we have $d(X_{\bar{\sigma}_{\infty}}^{j_a}, \partial U_{r(i_{\ell}^a)}) \leq \e$ a.s. for any $\e>0$. Thus, we see that $X_{\bar{\sigma}_{\infty}}^{j_1} \in \partial \Ur$ and $X_{\bar{\sigma}_{\infty}}^{j_2} \in \partial U_{r+1}$. Then, we have $\Xi_{\bar{\sigma}_{\infty}} \in \partial \mathfrak{U}_r \cap \partial \mathfrak{U}_{r+1}$, where $\partial \mathfrak{U}_r = \{ \xi \in \M; \xi(\partial \Ur) \geq 1 \}$. Suppose that $\P_{\mu}(\bar{\sigma}_{\infty}<\infty; \Omega_{\infty})>0$. Then by the above arguments, we have
$$ 
P_{\mu}\left( \Xi_{\bar{\sigma}_{\infty}(\Xi)} \in \partial \mathfrak{U}_r \cap \partial \mathfrak{U}_{r+1}; \Omega_{\infty} \right) > 0. 
$$
That is $P_{\mu}\left( \sigma_{\partial \mathfrak{U}_r \cap \partial \mathfrak{U}_{r+1}} < \infty \right) > 0$, where $\sigma_{\partial \mathfrak{U}_r \cap \partial \mathfrak{U}_{r+1}}$ is the first hitting time for the set $\partial \mathfrak{U}_r \cap \partial \mathfrak{U}_{r+1}$. By the general theory of Dirichlet forms, we have $\Capa^{\mu}(\partial \mathfrak{U}_r \cap \partial \mathfrak{U}_{r+1}) > 0$. Additionally, by the locally boundedness of the ${\sf n}$-correlation functions $\rho^{\sf n}$ we can show that $\Capa^{\mu}(\partial \mathfrak{U}_r \cap \partial \mathfrak{U}_{r+1}) = 0$ (see, e.g., the argument of Lemma \ref{L:51a}). These statements contradict one another. 
Therefore we have $P_{\mu}(\bar{\sigma}_{\infty}\le T; \Omega_{\infty})=0$ for any $T\in\N$. 
We can then show (i). Statement (ii) is obtained by similar arguments. Thus, the proof is complete.
\end{proof}
Let
\begin{align*}
\tau_{r}=\tau_{r}(\Xi) &= \inf \{ t>0; \bar{\X}_t^{[\k]} \notin \Ur^\k \}, \\
\tau_{r}^{[\k]}=\tau_{r}^{[\k]}(\X^{[\k]}, \Xi^{[\k]}) &= \inf \{ t>0; \X_t^{[\k]} \notin \Ur^\k \}
\end{align*}
and
$$
\tau_{\infty} = \lim_{r \to \infty} \tau_{r}, \quad
\tau_{\infty}^{[\k]} = \lim_{r \to \infty} \tau_{r}^{[\k]}.
$$

\begin{lemma}\label{L:42f} 
Suppose that the assumptions in Theorem \ref{T:31c} hold. Let $\mathfrak{A}_r^\k$ be as in Lemma \ref{L:42b}. Then we have
\begin{equation} \label{:42l}
\P_{\unl^{[\k]}(\x_\k, \eta)}^{\tau_{r}} = \P_{(\x_\k, \eta)}^{[\k], \tau_{r}^{[\k]}} \circ (\unl^{[\k]})^{-1} \quad \text{for any $(\x_\k, \eta) \in \mathfrak{A}_r^\k$}. 
\end{equation}
Moreover,
\begin{equation} \label{:42m}
\P_{\unl^{[\k]}(\x_\k, \eta)}^{\tau_{\infty}} = \P_{(\x_\k, \eta)}^{[\k], \tau_{\infty}^{[\k]}} \circ (\unl^{[\k]})^{-1} \quad \text{for any $(\x_\k, \eta) \in \liminf_{r \to \infty} \mathfrak{A}_r^\k$}. 
\end{equation}
\end{lemma}
\begin{proof}
From Lemma \ref{L:42e}, we see $\X_t^{[\k]} \in \Ur^\k$ for any $t \in [0, \infty)$ on $\Omega_{\infty}^{[\k]}$. Hence combining Lemma \ref{L:42b} with the strong Markov property repeatedly, we have for any $(\x_\k, \eta) \in \mathfrak{A}_r^\k$
\begin{equation*}
\P_{\unl^{[\k]}(\x_\k, \eta)}^{\bar{\sigma}_i}(\cdot; \Omega_{\infty}) = \P_{(\x_\k, \eta)}^{[\k], \bar{\sigma}_i^{[\k]}}(\cdot; \Omega_{\infty}^{[\k]}) \circ (\unl^{[\k]})^{-1} \quad \text{for any $i \in \N$}. 
\end{equation*}
Therefore, by Lemma \ref{L:42e} we can show that
\begin{equation} \label{:42n}
\P_{\unl^{[\k]}(\x_\k, \eta)}(\cdot; \Omega_{\infty}) = \P_{(\x_\k, \eta)}^{[\k]}(\cdot; \Omega_{\infty}^{[\k]}) \circ (\unl^{[\k]})^{-1} \quad \text{for any $(\x_\k, \eta) \in \mathfrak{A}_r^\k$}. 
\end{equation}

Next, suppose that $(\X^{[\k]}, \Xi^{[\k]}) \notin \Omega_{\infty}^{[\k]}$. Then, there exists some $i$ such that $\X_{\bar{\sigma}_i^{[\k]}}^{[\k]} \notin \Ur^\k$ and $\X_{\bar{\sigma}_j^{[\k]}}^{[\k]} \in \Ur^\k$ for any $j<i$. Set 
\begin{equation*}
\Omega_{i*}^{[\k]} = \{\X_{\bar{\sigma}_i^{[\k]}}^{[\k]} \notin \Ur^\k, \X_{\bar{\sigma}_j^{[\k]}}^{[\k]} \in \Ur^\k \ \text{for any $j<i$} \}. 
\end{equation*}
In addition, set $\Omega_{i*} = \unl^{[\k]}(\Omega_{i*}^{[\k]})$. By Lemma \ref{L:42b} and the strong Markov property we have
\begin{equation} \label{:42o}
\P_{\unl^{[\k]}(\x_\k, \eta)}^{\bar{\sigma}_i}(\cdot; \Omega_{i*}) = \P_{(\x_\k, \eta)}^{[\k], \bar{\sigma}_i^{[\k]}}(\cdot; \Omega_{i*}^{[\k]}) \circ (\unl^{[\k]})^{-1} \quad \text{for any $(\x_\k, \eta) \in \mathfrak{A}_r^\k$}. 
\end{equation}
By construction we see that $\tau_{r} = \bar{\sigma}_i$ on $\Omega_{i*}$ and $\tau_r^{[\k]} = \bar{\sigma}_i^{[\k]}$ on $\Omega_{i*}^{[\k]}$. 
Hence \eqref{:42o} implies that
\begin{equation} \label{:42p}
\P_{\unl^{[\k]}(\x_\k, \eta)}^{\tau_{r}}(\cdot; \Omega_{i*}) = \P_{(\x_\k, \eta)}^{[\k], \tau_r^{[\k]}}(\cdot; \Omega_{i*}^{[\k]}) \circ (\unl^{[\k]})^{-1} \quad \text{for any $(\x_\k, \eta) \in \mathfrak{A}_r^\k$}. 
\end{equation}
We see that the state spaces $\Omega=\Omega_{\infty} + \sum_{i=1}^{\infty} \Omega_{i*}$ and $\Omega^{[\k]}=\Omega_{\infty}^{[\k]} + \sum_{i=1}^{\infty} \Omega_{i*}^{[\k]}$ have a probability of 1. Hence, using \eqref{:42n} and \eqref{:42p} we can show \eqref{:42l}: then \eqref{:42m} follows immediately. 
\end{proof}

As the conclusion of the above arguments, we can now prove Theorem \ref{T:31c}. 

\begin{proof}[Proof of Theorem \ref{T:31c}]
Let 
\begin{equation*}
\M_0 = \bigcap_{\k \in \N} \liminf_{r \to \infty} \unl^{[\k]}\left( \mathfrak{A}_r^\k \right) . 
\end{equation*}
Then, from \eqref{:42b}, we have \eqref{:31g}. In addition, \As[NEX] implies that $\tau_{\infty} = \infty$ for $\P_{\xi}$-a.s. and $\tau_\infty^{[\k]}=\infty$ for $\P_{(\x_\k, \eta)}^{[\k]}$- a.s. Hence, \eqref{:42m} leads to \eqref{:31h} and \eqref{:31k}. Claim (ii) is derived from the consistency \eqref{:31h}. Thus, the proof is complete.
\end{proof}

\subsection{Proof of Theorem \ref{T:32a}}\label{SS:43}

In this subsection, we prove Theorem \ref{T:32a}.
Let  $\varphi\in C_0^\infty(S^\k)$ and $f\in \mathcal{D}_{\circ}^\k$, $\k\in\N$.
From the definition of $\mathbb{D}^\k$ in \eqref{:31d},
$$
\mathbb{D}^\k[f,\varphi](\x_\k,\eta)=\nabla^{[\k]}[f,\varphi](\x_\k,\eta)=\frac{1}{2}\sum_{j=1}^\k
\int_S dy \;p(x_\k^j,y)\nadeffsa{j}{y}f(\x_\k,\eta)\nadeffsa{j}{y}\varphi(\x_\k).
$$
Then, 
\begin{align}\nonumber
&-\mathcal{E}^{\mu^{[\k]}}(f, \varphi)
= -\int_{\M\times S^\k}d\mu^{[\k]}(\x_\k,\eta)\mathbb{D}^\k[f,\varphi](\x_\k,\eta)
\\ \nonumber
&=\frac{1}{2}\sum_{j=1}^\k\int_{\M\times S^\k}\int_S 
d\mu^{[\k]}(\x_\k^{j,y},\eta)dx_\k^j
 \;p(y,x_\k^j)f(\x_\k,\eta)\nadeffsa{j}{y}\varphi(\x_\k)
\\ \label{:43a}
&\qquad +\frac{1}{2}\sum_{j=1}^\k\int_{\M\times S^\k}\int_S
d\mu^{[\k]}(\x_\k,\eta)dy \;p(x_\k^j,y)f(\x_\k,\eta)\nadeffsa{j}{y} \varphi(\x_\k).
\end{align}
By simple observation, it is clear that 
\begin{equation} \label{:43b}
\frac{d\mu^{[\k]}(\x_\k^{j,y},\eta)dx_\k^j}{d\mu^{[\k]}(\x_\k,\eta)dy}
=\frac{d\mu^{[1]}(y,\unl(\x_\k^{\hiku{j}},\eta))dx_\k^j}{d\mu^{[1]}(x_\k^j,\unl(\x_\k^{\hiku{j}},\eta))dy}=\frac{\rho^1(y)d\mu_{y}}{\rho^1(x_\k^j)d\mu_{x_\k^j}}(\unl(\x_\k^{\hiku{j}},\eta)),
\mbox{ $d\mu^{[\k]}dy$-a.s.}
\end{equation}
where 
$\x_{\k}^{\hiku{j}}= (x_\k^1, \ldots, x_\k^{j-1}, x_\k^{j+1}, \ldots, x_\k^{\k}) \in S^{\k-1}$ 
for $\x_\k = (x_\k^1, \ldots, x_\k^\k)$. 
Combining \eqref{:43a} and \eqref{:43b} with definition of $c(x,y,\xi)$ in (\ref{:32b}), we have
\begin{align}\label{:43c}
&-\mathcal{E}^{\mu^{[\k]}}(f, \varphi)
=\sum_{j=1}^\k\int_{\M\times S^\k}
d\mu^{[\k]}(\x_\k,\eta) f(\x_\k,\eta)
\int_S dy \; c(x_\k^j,y,\unl(\x_\k,\eta)) \nadeffsa{j}{y} \varphi(\x_\k)
\end{align}
Similarly, for $\varphi, \psi \in C_0^\infty (S^k)$
\begin{align}\label{:43d}
&\mathcal{E}^{\mu^{[\k]}}(\varphi,\psi)
=\frac{1}{2}\sum_{j=1}^\k\int_{\M\times S^\k}
d\mu^{[\k]}(\x_\k,\eta) 
\int_S dy \; c(x_\k^j,y,\unl(\x_\k,\eta)) \nadeffsa{j}{y} \varphi(\x_\k)\nadeffsa{j}{y} \psi(\x_\k).
\end{align}

Let $\varphi_L^{i} \in C_0^\infty(S)$, $i=1,2,\dots,d$
such that
$$
\varphi_L^{i}(x)=\varphi_L^{i}(x^1,x^2,\dots, x^d)=x^{i}, \quad x^{i}\in [-L,L]
$$
and set $\varphi_L^{j,i}(\x_\k)=\varphi_L^{i}(x_\k^j)$, $j=1,2,\dots,\k$, $i=1,2,\dots,d$.
We introduce the additive functional $A^{[\varphi_L^{j,i}]}=\varphi_L^{j,i}(\X^{[\k]}_t)-\varphi_L^{j,i}(\X^{[\k]}_0)$ of the process $(\X^{[\k]}_t,\Xi^{[\k]}_t)$. By Fukushima's decomposition \cite[Theorem 5.2.5]{fot.2}, we have
\begin{align}\label{:43d2}
A_t^{[\varphi_L^{j,i}]}=N_t^{[\varphi_L^{j,i}]}+M_t^{[\varphi_L^{j,i}]},
\end{align}
where $M_t^{[\varphi_L^{j,i}]}$ is a martingale additive functional of finite energy and
$N_t^{[\varphi_L^{j,i}]}$ is an additive functional of zero energy.
From (\ref{:43c}), we find that
\begin{align}\notag
N_t^{[\varphi_L^{j,i}]}&=\int_0^t ds
\int_S dy \ c(X_{s-}^j,y,\Xi_{s-})(\varphi_L^{i}(y)-\varphi_L^{i}(X_{s-}^j) )
\\ \notag 
&=\int_0^t ds\int_S du \int_0^\infty dr \;
\left\{\varphi_L^{i}(X_{s-}^j+u)-\varphi_L^{i}(X_{s-}^j)\right\}
a^j(u,r,\X,s),
\end{align}
where $\Xi= \Xi^{[\k]}+ \unl (\X^{[\k]})$.
Recall that $\Xi$ and $X^j$ do not depend on $\k$ by virtue of the consistency of the sequence $\{ (\X^{[\k]}, \Xi^{[\k]})\}$, as proved in Theorem \ref{T:31c}(ii).

From (\ref{:43d}), the martingale part can be represented as
\begin{align}\label{:43f}
&M_t^{[\varphi_L^{j,i}]}=\int_0^t \int_S  \int_0^\infty M^j(dsdudr) \;
\left\{\varphi_L^{i}(X_{s-}^j+u)-\varphi_L^{i}(X_{s-}^j)\right\}
 a^j(u,r,\X,s),
\end{align}
where
\begin{align}\label{:43f}
M^j(dsdudr)=N^j(dsdudr)-dsdudr,
\end{align}
and $\bN=\{N^j\}_{j=1}^\k$ is a sequence of $\{\mathcal{F}_t\}_{t\ge 0}$ adapted independent Poisson RPFs  on $[0,\infty)\times \R^d \times [0,\infty)$ with the common intensity measure $dsdudr$.  
Then, from \eqref{:43d2}--\eqref{:43f}, we have the following representations
\begin{align*}
A_t^{[\varphi_L^{j,i}]}=\int_0^t \int_S  \int_0^\infty N^j(dsdudr) \;
\left\{\varphi_L^{i}(X_{s-}^j+u)-\varphi_L^{i}(X_{s-}^j)\right\}
 a^j(u,r,\X,s).
\end{align*}

Setting $\displaystyle{\tau_L^j=\min_{1\le i \le d}\inf_{t>0}\{X_t^{j,i} \notin [-L,L]\}}$, $L>0$,
we have
$$
A_{t\wedge\tau_L^j}^{[x^{j,i}]}=N_{t\wedge\tau_L^j}^{[x^{j,i}]}+M_{t\wedge\tau_L^j}^{[x^{j,i}]},
\quad i=1,2,\dots,d, \; t\ge 0.
$$
Hence,
the process $X_t^j$ solves the following SDE
\begin{equation}\nonumber
X_{t\wedge\tau_L^j}^j=X_{0}^j+\int_0^{t\wedge\tau_L^j} \int_S  \int_0^\infty N^j(dsdudr) \;
u a^j(u,r,\X,s),
\quad t\ge 0
\end{equation}
for any $L>0$.
Because $X_t^j$ is conservative, we have $\lim_{L\to\infty}\tau_{L}^j=\infty$ and see that
the process $X_t^j$ solves the following ISDE
\begin{equation}\nonumber
X_t^j=X_0^j+\int_0^t \int_S  \int_0^\infty N^j(dsdudr) \;
u a^j(u,r,\X,s).
\end{equation}
Hence, we obtain (\ref{ISDE}). 
This completes the proof of Theorem \ref{T:32a}.
\qed

\subsection{Proof of Theorems \ref{T:33a} and \ref{T:33b}} \label{SS:44}

To prove Theorems \ref{T:33a} and \ref{T:33b} we use the following proposition, which are modifications of \cite[Theorem 3.1 and Theorem 3.2]{o-t.tail}.

\begin{proposition}\label{P:44a}
Assume that \As[TT] for $\mu$ holds. 
Assume that (\ref{ISDE}) with (\ref{:32e})-(\ref{:32f}) has a weak solution $(\X, \bN)$ satisfying \As[IFC], \As[AC] for $\mu$, \As[SIN] and \As[NBJ].
Then (\ref{ISDE}) with (\ref{:32e})-(\ref{:32f}) has a family of unique strong solutions $\{F_{\x}\}$ starting at $\x$  for $\mu\circ\mathfrak{l} ^{-1}$-a.s.  $\x$ under the constraints of \As[MF],\As[IFC], \As[AC] for $\mu$, \As[SIN] and \As[NBJ].
\end{proposition}
\vskip 3mm

Suppose that $\mu$, $\mu_{\rm Tail}^{\zeta}$, and $\tilde{\mathfrak{H}}$ satisfy (\ref{:33k})--(\ref{:33m})
for any $\zeta \in \tilde{\mathfrak{H}}$ with some $\tilde{\mathfrak{H}}$ satisfying $\mu(\tilde{\mathfrak{H}})=1$.

For a labeled process $\X$ we set $\Xi=\mathfrak{u}(\X)$. 
We assume that $\mu=P\circ\Xi^{-1}$ as before. We set
$$
P^\zeta(\cdot) := P(\cdot | \Xi_0^{-1}(\mathcal{T}(\M)))\big|_{\Xi_0=\zeta}
=\int P(\cdot | \X_0=\x)\mu_{tail}^{\zeta}\circ \mathfrak{l}^{-1}(d\x).
$$
\vskip 3mm
\begin{proposition}\label{P:44b}

Assume that $(\X, \bN)$ is a weak solution of (\ref{ISDE}) with (\ref{:32e})-(\ref{:32f}) satisfying \As[IFC], \As[SIN] and \As[NBJ]. Assume that, for $\mu$-a.s. $\zeta$, $(\X, \bN)$ under $P^\zeta$ satisfies \As[AC] for  $\mu_{\rm Tail}^{\zeta}$. 
Then, for $\mu$-a.s. $\zeta$, (\ref{ISDE}) with (\ref{:32e})-(\ref{:32f}) has a family of unique strong solutions $\{F_{\x}^\zeta\}$ starting at $\x$ for $P^\zeta\circ \X_0^{-1}$-a.s. $\x$
under the constraints of \As[MF], \As[IFC], \As[AC] for $\mu_{\rm Tail}^{\zeta}$, \As[SIN], and \As[NBJ].
\end{proposition}

The proofs of the above propositions are obtained in the same way as those of \cite[Theorem 3.1 and Theorem 3.2]{o-t.tail}, which are derived from the first tail theorem \cite[Theorem 4.1]{o-t.tail} and  the second tail theorem \cite[Theorem 5.1]{o-t.tail}. Although these two theorems refer to the systems of interacting Brownian motions, the arguments used in their proofs are robust and applicable to systems of interacting particles with jumps.

\vskip 3mm

\noindent{\it Proof of Theorem \ref{T:33a}}.
From Corollary \ref{C:32a} $\X=\mathfrak{l}(\Xi)$ is a weak solution of (\ref{ISDE}) with (\ref{:32e})--(\ref{:32f}).
As $\mu$ is a reversible measure of $\Xi$, \As[AC] holds.
\As[SIN], \As[TT] and \As[IFC] hold from the assumptions.
From Lemma \ref{L:53b} \As[NBJ] is satisfied.
Hence, we can apply Proposition \ref{P:44a} and obtain Theorem \ref{T:33a}.
This completes the proof.
\qed

\vskip 3mm

\noindent{\it Proof of Theorem \ref{T:33b}}.
Assume that \As[QG] and \As[EQ] for $\mu$ hold. From the definition of  quasi-Gibbs measures and the decomposition in (\ref{:33h}), we see that \As[QG] for $\mu_{\rm Tail}^{\zeta}$ holds for $\mu$-a.s. $\zeta$.
For $\mu$-a.s. $\zeta$, \As[A.2.$\k$] for $\mu_{\rm Tail}^{\zeta}$ is derived from that for $\mu$ by Fubini's theorem. Hence, we have (i).
From Lemma \ref{L:2a} and Theorem \ref{T:31b}, we can construct the unlabeled process 
$(\Xi_t, \{\P_\xi^\zeta\}_{\xi\in\M})$ associated with the quasi-regular Dirichlet form 
$(\mathcal{E}^{\mu^{\zeta}}, \mathcal{D}^{\mu^{\zeta}}, L^2(\M,\mu^{\zeta}))$. 
 Thus, we have (ii).
Assume that \As[SIN] holds for the process.
From Corollary \ref{C:32a} $\X=\mathfrak{l}(\Xi)$ is a weak solution of (\ref{ISDE}).
Thus, we obtain (iii).
Because $\mu_{\rm Tail}^{\zeta}$ is a reversible measure of $\Xi$, \As[AC] holds.
\As[IFC] holds from the assumptions. From Lemma \ref{L:53b} \As[NBJ] is satisfied.
Hence, we can apply Proposition \ref{P:44b} and obtain Theorem \ref{T:33b}.
This completes the proof.
\qed

\section{Sufficient conditions for assumptions} \label{S:5}
\subsection{Sufficient conditions for \As[NCL]}\label{SS:51}
In this subsection, we discuss sufficient conditions for \As[NCL].
For systems of interacting diffusion processes, sufficient conditions \As[NCL] were given in \cite{O04}.

First, we introduce the definition of density functions of $\mu$. 
Permutation-invariant functions $\sigma_r^{\n} : U_r^{\n} \to [0,\infty)$ are called density functions of $\mu$ if
\begin{equation} \label{:51a}
\frac{1}{{\n}!}\int_{U_r^{\n}}f_r^{\n}(\x_{\n})\sigma_r^{\n}(\x_{\n})d\x_{\n} = \int_{\M_{r, \n}}f(\xi)\mu(d\xi)
\end{equation}
for all bounded $\sigma[\pi_r]$-measurable functions $f$, where $f_r^{\n}$ is a $U_r^{\n}$-representation of $f$ defined in Section 2. 
Then, a sufficient condition for \As[NCL] is given as follows. 

\begin{lemma}\label{L:51a}
{\rm (i)} Let $d=1$. Assume that the conditions in Lemma \ref{L:2a} are satisfied.
If there exist positive constants $\Ct\label{con:NCL}=\cref{con:NCL}(r, \sf n)$ such that 
\begin{equation} \label{:51b}
\sigma_r^{\sf n}(x^1, \ldots, x^{\sf n}) \leq \cref{con:NCL}\min_{i \neq j}|x^i-x^j|, \quad \text{for any $x^1, \ldots, x^{\sf n} \in U_r$, $r, {\sf n} \in \N$, }
\end{equation}
then \As[NCL] holds. 

\noindent{\rm (ii)} 
Let $d \geq 2$. Assume that the conditions in Lemma \ref{L:2a} are satisfied. If the density functions $\sigma_r^{\n}$ are bounded for any $r, \n \in \N$, then \As[NCL] holds. 
\end{lemma}

\begin{proof}
We first prove (i).
For $\n, m, r, q \in \N$, we set
\begin{align*}
&H_{\sf n}(m, r, q) = \{ \xi \in \mathfrak{M} ;  \xi(U_{2^r}) = \n, \ \xi(\mathcal{A}_{r, m}) = 0, \min_{x,y \in \pi_{2^r-1}(\xi), x\not= y}|x-y| < 2^{-q},  \}, 
\end{align*}
where we set $\mathcal{A}_{r, m} = U_{2^r+\frac{1}{m}}\setminus U_{2^r}=\{ x \in S; 2^r \leq |x| < 2^r + \frac{1}{m} \}$. 
From the monotonicity and the sub-additivity of the capacity,
if 
\begin{equation}\label{:51c}
\inf_{q \in \N} {\rm Cap}^{\mu}(H_\n(m, r, q)) = 0,
\; \text{for any $\n, m, r \in \N$}
\end{equation}
is satisfied, then \As[NCL] holds.

We prepare some functions.
Let ${\sf w}$ be an increasing smooth function on $\R$ such that  that ${\sf w}(x)=1$, $x\ge 1$, ${\sf w}(x)=0$, $x\le 0$, and  ${\sf w}'(x) \le \Ct\label{con:w}$ on $\R$ for some $\cref{con:w}>0$.  
We set 
\begin{equation*}
h_q(x) = \begin{cases}
 {\sf w}\left( \frac{\log |x|}{\log 2^{-q}} \right), & x \neq 0, \\
1, & x = 0, 
\end{cases}
\quad 
g(\xi) = \begin{cases}
 \prod_{x \in \pi_{\mathcal{A}_{r, m}}} {\sf w}(m(|x|-2^r)), & \xi( \mathcal{A}_{r, m}) \not= 0, \\
1, & \xi( \mathcal{A}_{r, m}) = 0.
\end{cases}
\end{equation*}
Then, we introduce the function $\hat{h}_q(\xi) : \mathfrak{M} \to \mathbb{R}$ defined by
\begin{equation*}
\hat{h}_q(\xi) 
= \begin{cases} 
\left\{ \int_{U_{2^r}} \xi(dx) \int_{U_{2^r}} ( \xi \setminus x )(dy)h_q(x-y) {\sf w}(2^r-|x|) {\sf w}(2^r-|y|)  \right\} g(\xi) , &\text{if $\xi \in \mathfrak{M}_{2^r, \n}$. } 
\\
&
\\
0, & \text{otherwise.}
\end{cases}
\end{equation*}
By definition we have $\hat{h}_q (\xi) =0$, $\xi \notin \mathfrak{M}_{2^r, \n} $, and $\hat{h}_q(\xi) \geq 1$ for $\xi \in H_\n(m, r, q)$.
 
For $\mathfrak{N}\subset \mathfrak{M}$, $A,B \subset \R^d$, we set
$$
J(\mathfrak{N},A,B)=\int_{\mathfrak{N}} \mu(d\xi) \int_{A} \xi(dx) \int_{B} (\hat{h}_q(\xi^{x, y})-\hat{h}_q(\xi))^2 p(x, y) dy.
$$
We divide $\mathcal{E}(\hat{h}_q, \hat{h}_q)$ into as the following six terms: 
$\displaystyle{
\mathcal{E}(\hat{h}_q, \hat{h}_q) 
= \sum_{k=1}^6 J_q^k
}$, 
where
\begin{align*}
&J_q^1=J(\mathfrak{M}_{2^r, \n}, U_{2^r}, U_{2^r}), \quad J_q^2=J(\mathfrak{M}_{2^r, \n}, U_{2^r}^c, U_{2^r}),
\quad J_q^3=J(\mathfrak{M}_{2^r, \n-1}, U_{2^r}^c, U_{2^r}),
\\
&J_q^4=J(\mathfrak{M}_{2^r, \n}, U_{2^r}, U_{2^r}^c), \quad J_q^5=J(\mathfrak{M}_{2^r, \n+1}, U_{2^r}, U_{2^r}^c),
\quad J_q^6=J(\mathfrak{M}_{2^r, \n}, \mathcal{A}_{r, m}, \mathcal{A}_{r, m}).
\end{align*}

To show \eqref{:51c}, we need to show that
\begin{align}\label{:51d}
\lim_{q \to \infty} J_q^k = 0, \quad 1\le k \le 6,
\end{align}
for fixed $\n, m, r \in\N$.
We first consider the case $k=2$. 
In $J_q^2$, note that $\hat{h}_q(\xi^{x, y})=0$, because $\xi^{x,y}\notin \mathfrak{M}_{2^r, \n}$. In addition, $||x|-2^r| \leq |x-y|$ for $x \in \mathcal{A}_{r, m}$ and $y \in U_{2^r}$. Hence, by the definition of $\hat{h}_q$, we have
\begin{align}\notag
J_q^2 &= \int_{\mathfrak{M}_{2^r, \n}}\mu(d\xi) \int_{\mathcal{A}_{r, m}\cup U_{2^r+\frac{1}{m}}^c} \xi(dx)  \int_{U_{2^r}} \hat{h}_q(\xi)^2 p(x, y)dy  
\\ \notag 
&\leq \cref{con:w}^2 \int_{\mathfrak{M}_{2^r, \n}}\mu(d\xi) \; \hat{f}_q^\n(\xi)^2 \int_{\mathcal{A}_{r, m}} \xi(dx) \int_{U_{2^r}}  |x-y|^2 p(x, y)dy  
\\ \label{:51e}
&\quad + \int_{\mathfrak{M}_{2^r, \n}}\mu(d\xi) \; \hat{f}_q^\n(\xi)^2 
\int_{U_{2^r+\frac{1}{m}}^c} \xi(dx) \int_{U_{2^r}}  p(x, y)dy,  
\end{align}
where
\begin{equation*}
\hat{f}_q^\n(\xi) = 
\begin{cases}
\displaystyle{\int_{U_{2^r}} \xi(dx) \int_{U_{2^r}}}(\xi\setminus x)(dy)
 h_q(x-y), & \text{if $\xi \in \mathfrak{M}_{2^r, \n}$,} \\
0, & \text{otherwise. }
\end{cases}
\end{equation*} 
In addition, by \As[p.1], \As[p.2], and \As[A.2], we see that $\int_S (|x-y|^2 \wedge 1) p(x, y)\rho^1(y)dy < \infty$ for any $x \in S$. Then,
\begin{align}\notag 
&\int_{\mathfrak{M}_{2^r, \n}}\mu(d\xi) \; \int_{\mathcal{A}_{r, m}} \xi(dx) \int_{U_{2^r}}  |x-y|^2 p(x, y)dy <\infty,
\\ \label{:51g}
&\int_{\mathfrak{M}}\mu(d\xi)\int_{U_{2^r+\frac{1}{m}}^c} \xi(dx) \int_{U_{2^r}}  p(x, y)dy  < \infty. 
\end{align}  
Note that $\hat{f}_q^\n(\xi) \leq \n(\n-1)/2$ for any $\n, q \in \N$ and $\xi \in \mathfrak{M}$, and $\lim_{q \to \infty} \hat{f}_q^\n(\xi) = 0$ for a.s. $\xi$. Thus, 
from \eqref{:51e}, \eqref{:51g}, and Lebesgue's dominated convergence theorem, we have \eqref{:51d} for $k=2$. Using a similar approach, we can obtain \eqref{:51d} for $k=3, 4, 5, 6$. 

It remains to show the case $k=1$. Note that $\xi, \xi^{x, y} \in \mathfrak{M}_{2^r, \n}$ in $J_q^1$. 
Hence, in $J_q^1$, we have
\begin{equation}\label{:51h} \textstyle 
\Big|\hat{h}_q(\xi^{x, y}) - \hat{h}_q(\xi)\Big|
\leq 2 \  \Big| \int_{U_{2^r}}(\xi \setminus x)(dz) {\sf H}(x,y,z) {\sf w}(2^r-|z|)   \Big| 
\end{equation}
with
$$
{\sf H}(x,y,z)=h_q(z-y){\sf w}(2^r-|y|) - h_q(z-x){\sf w}(2^r-|x|).
$$
Note that 
\begin{align*}
\xi &\mapsto 4 \int_{U_{2^r}} \xi(dx)
\int_{U_{2^r}} p(x, y) dy  \Big( \int_{U_{2^r}}(\xi \setminus x)(dz) {\sf H}(x,y,z) {\sf w}(2^r-|z|)  \Big)^2
\end{align*}
is a $\sigma(\pi_{2^r})$-measurable function. 
Thus, the upper bound of $J_q^1$ can be given by the integral using the $\n$-density function $\sigma_{2^r}^{\sf n}$ of $\mu$ defined in \eqref{:51a}. 
By \eqref{:51h}, we have
\begin{align}\notag 
J_q^1 &\leq \frac{4}{\n!} \int_{U_{2^r}^{\n}}d\x_{\n} \sum_{k=1}^{\n} \int_{U_{2^r}} \ dy \Big( \sum_{j \neq k} {\sf H}(x^k,y,x^j) {\sf w}(2^r-|x^j|) \Big)^2 p(x^k, y) \sigma_{2^r}^{\n}(\x_{\n})
\\ \label{:51j}
&\leq \frac{4(\n-1)}{\n!} \int_{U_{2^r}^{\n}} d\x_{\n} \int_{U_{2^r}} dy \ 
 \sum_{k=1}^{\n} \sum_{j \neq k}{\sf H}(x^k,y,x^j)^2  p(x^k, y) \sigma_{2^r}^{\n}(\x_{\n}), 
\end{align}
where we use the inequalities $(\sum_{i=1}^{\n-1} a_i)^2 \leq (\n-1)\sum_{i=1}^{\n-1} a_i^2$ and ${\sf w}(x) \leq 1$. From \eqref{:51b}, we know that there exists a positive constant $\Ct\label{cnI} = \cref{cnI}(r, \n)$ such that
\begin{align}\notag 
&\text{R.H.S. of \eqref{:51j}} 
\leq \cref{cnI} \int_{U_{2^r}^3}dxdydz {\sf H}(x,y,z)^2  p(x, y) |x-z| 
\\  \notag 
&\leq 2\cref{cnI} \int_{U_{2^r}^3}dxdydz \{ h_q(z-y)-h_q(z-x) \}^2{\sf w}(2^r-|y|)^2 p(x, y) |x-z|
\\ \notag 
&\quad +2\cref{cnI}\int_{U_{2^r}^3}dxdydz h_q(z-x)^2\{ {\sf w}(2^r-|y|)-{\sf w}(2^r-|x|) \}^2  p(x, y) |x-z|dx \\ \label{:51k}
&=: J_q^{11} + J_q^{12}. 
\end{align}
Because ${\sf w}(x) \leq 1$ for any $x \in \R$, we have
\begin{equation} \label{:51l}
J_q^{11} \leq 2\cref{cnI} \int_{U_{2^r}^3}dxdydz \{ h_q(z-y)-h_q(z-x) \}^2 p(x, y) |x-z|\to 0, \quad q\to\infty, 
\end{equation}
from a straightforward calculation. Additionally, from the triangle inequality and ${\sf w}'(x) \leq \cref{con:w}$ for any $x \in \R$ we see that
\begin{equation} \label{:51m}
J_q^{12} \leq 2\cref{cnI}\cref{con:w} \int_{U_{2^r}^3} dxdydz \ h_q(z-x)^2|y-x|^2 p(x, y) |x-z|
\to 0, \quad q\to\infty, 
\end{equation}
by the fact that $h_q(x) \to 0$ a.e. as $q \to \infty$ and Lebesgue's dominated convergence theorem. 
Therefore, by \eqref{:51j}--\eqref{:51m}, we have \eqref{:51d} for $k=1$.
From \eqref{:51c} and \eqref{:51d}, we have (i).

For the case when $d\ge 2$ we obtain (ii) through a similar calculation as for (i). 
This completes the proof.
\end{proof}

Note that $0 \leq \sigma_r^\m(\x_\m) \leq \rho^\m(\x_\m)$ in general.
This is a direct consequence of Lemma \ref{L:51a}.  

\begin{corollary} \label{C:51a}
Let $d=1$. Assume that the conditions in Lemma \ref{L:2a} are satisfied.
If there exist positive constants $\Ct\label{51}=\cref{51}(\m, r)$ such that
\begin{equation}\label{:51n}
\rho^\m(x^1,x^2,\dots,x^\m)\le \cref{51}\min_{\substack{i, j \\ i \neq j}}|x^i-x^j|
\quad \mbox{ for all $(x^j)_{j=1}^\m \in U_{r}^{\m}$ with $x^i\not= x^j$},
\end{equation}
then \As[NCL] holds.
\end{corollary}

For DRPFs, the same argument gives the following result, from which we obtain \cite[Theorem 2.1]{O04}.

\begin{corollary} \label{C:51b}
Let $d=1$. In addition, suppose that $\mu$ is a DRPF with kernel $\mathbb{K}$.
Assume that $\mathbb{K}$ is locally Lipschitz continuous. Then \As[NCL] holds.
\end{corollary}
\begin{proof}
By \cite[Lemma 4.3]{O04} we have \eqref{:51b}. Then, by Lemma \ref{L:51a}(i),  we have \As[NCL] for $\mu$. Thus the proof is complete. 
\end{proof}

\begin{example}\label{E:51a}
${\rm Sine}_2$, ${\rm Bessel}_{\alpha, 2}$ $(\alpha \geq 1)$ and ${\rm Airy}_2$ are DRPFs on $\R$ with locally Lipschitz continuous kernels given by 
\begin{align}\label{:51o}
{\sf K}_{{\rm sin}, 2}(x, y) &= S(x-y), \quad \text{$x, y \in \R$, $x \neq y$},  
\\ \notag
{\sf K}_{{\rm Bessel}, \alpha, 2}(x, y) &= \frac{J_{\alpha}(\sqrt{x})\sqrt{y}J_{\alpha}'(\sqrt{y}) - \sqrt{x}J_{\alpha}'(\sqrt{x})J_{\alpha}(\sqrt{y})}{2(x-y)}, \quad \text{$x, y \in [0, \infty)$, $x \neq y$},   
\\ \label{:51q}
{\sf K}_{\Ai, 2}(x, y) &= \frac{\Ai(x)\Ai'(y)-\Ai'(x)\Ai(y)}{x-y}, \quad \text{$x, y \in \R$, $x \neq y$},  
\end{align}
respectively, where $S(x)=\sin(\pi x)/(\pi x)$, $J_{\alpha}$ is the Bessel function of order $\alpha$, and $\Ai$ is the Airy function. 
By Corollary \ref{C:51b}, \As[NCL] hold for these DRPFs. 

The Ginibre RPF is a DRPF on $\C (\simeq \R^2)$ with a locally Lipschitz kernel given by 
\begin{equation*}
{\sf K}_{{\rm Gin}}(z, w) = \frac{1}{\pi}e^{-(|z|^2+|w|^2)/2}e^{z\bar{w}}, \quad \text{$z, w \in \C$}, 
\end{equation*}
where $\bar{w}$ is the complex conjugate of $w$. By the definition of the correlation functions for DRPFs, we can easily see that $\rho^{\m}_{{\rm Gin}}$ is locally bounded for each $\m \in \N$, where $\rho^{\m}_{{\rm Gin}}$ is the $\m$-correlation function of the Ginibre RPF. Therefore, by Lemma \ref{L:51a}(ii), if $\mu$ is the Ginibre RPF, \As[NCL] holds. 
\end{example}

We also discuss \As[NCL] for quaternion determinantal RPFs (QDRPFs), which are RPFs whose correlation functions are given by quaternion determinants. 
We give the definition of quaternion determinants in (\ref{:6c}). 

We consider quaternion-valued kernels ${\sf K}$ on $\R \times \R$.
We assume 

\vskip 3mm

\noindent \As[CQ-1] $\sf K$ is represented using some local Lipschitz continuous functions ${\sf K}_1$, ${\sf K}_2$, ${\sf K}_3$ as 
\begin{equation*} 
{\sf K}(x, y) = \Upsilon\left( \begin{bmatrix} {\sf K}_1(x, y) & {\sf K}_2(x, y) \\ {\sf K}_3(x, y) - \e(x-y) & {\sf K}_1(y, x) \end{bmatrix} \right), 
\end{equation*}
or 
\begin{equation*} 
{\sf K}(x, y) = \Upsilon\left( \begin{bmatrix} {\sf K}_1(x, y) & {\sf K}_2(x, y) \\ {\sf K}_3(x, y) & {\sf K}_1(y, x) \end{bmatrix} \right), 
\end{equation*}
where $\Upsilon$ is the map defined in (\ref{:6b})
\vskip 3mm

\noindent \As[CQ-2]
${\sf K}_2$, ${\sf K}_3$ are antisymmetric, that is ${\sf K}_2(x, y)=-{\sf K}_2(y, x)$ and ${\sf K}_3(x, y)=-{\sf K}_3(y, x)$ hold. 

\vskip 3mm

We obtain the following lemma from a straightforward calculation. 
\begin{lemma} \label{L:51b}
Assume that \As[CQ-1] and \As[CQ-2] hold.
 Let $I$ be a bounded interval of $\R$. Then there exists a positive constant $\Ct\label{q1}$ depending on $I$ such that for any $i=0, 1, 2, 3$ and $x, y \in I$  
\begin{equation*} 
|[{\sf K}(x, x){\sf K}(y, y)-{\sf K}(x, y){\sf K}(y, x)]^{(i)}| \leq \cref{q1}|x-y|,  
\end{equation*}
where $[ \cdot ]^{(i)}$ for $i=0, 1, 2, 3$ is defined in \eqref{:6a}. 
\end{lemma}

From this lemma, with the expansion of a quaternion determinant, the locally boundedness of ${\sf K}_i$, and combinatorial arguments, we can show the following proposition. 
\begin{proposition} \label{P:51a}
Assume that \As[CQ-1] and \As[CQ-2] hold. Then there exists a positive constant $\Ct\label{q2}$ depending on ${\sf n}$ and $I$ such that  
\begin{align*}
\left| \qdet_{1 \leq i, j \leq {\sf n}}\left[ {\sf K}(x^i, x^j)\right] \right| \leq \cref{q2} \min_{\substack{1 \leq i, j \leq {\sf n} \\ i \neq j}}|x^i-x^j|.
\end{align*}
\end{proposition}
By this proposition, we can verify \eqref{:51n} for the QDRPFs associated with kernels satisfying \As[CQ-1] and \As[CQ-2]. 

\begin{example}(Sine, Airy RPF for $\beta=1, 4$) \label{E:51b}
Let $D(x) = dS(x)/dx$, $I(x) = \int_0^x S(x) dx$ for $x \in \R$, where $S$ is the function in \eqref{:51o}. In addition, we set 
\begin{align*}
J_1(x, y) &= {\sf K}_{\Ai, 2}(x, y)+\frac{1}{2}\Ai(x)\left( 1-\int_y^{\infty} \Ai(u)du \right), \quad x, y \in \R,  \\
J_4(x, y) &= {\sf K}_{\Ai, 2}(x, y) -\frac{1}{2}\Ai(x)\int_y^{\infty} \Ai(u)du, \quad x, y \in \R,
\end{align*}
where  ${\sf K}_{\Ai, 2}$ is defined in \eqref{:51q}. Note that these are locally bounded functions. 
The Sine and Airy RPFs for $\beta=1, 4$ are QDRPFs with the following kernels \cite{AGZ10, Meh04}: 
\begin{align*}
{\sf K}_{{\rm sin}, 1}(x, y) &= \Upsilon \left( \begin{bmatrix} S(x-y) & D(x-y) \\ I(x-y)-\e(x-y) & S(x-y) \end{bmatrix} \right), \quad x, y \in \R, \\
{\sf K}_{{\rm sin}, 4}(x, y) &= \Upsilon \left( \begin{bmatrix} S(2(x-y)) & D(2(x-y)) \\ I(2(x-y)) & S(2(x-y)) \end{bmatrix} \right), \quad x, y \in \R, \\
{\sf K}_{\Ai, 1}(x, y) &= \Upsilon \left( \begin{bmatrix} J_1(x, y) & -\frac{\partial}{\partial y}J_1(x, y) \\ \int_y^x J_1(u, y)du-\e(x-y) & J_1(y, x) \end{bmatrix} \right), \quad x, y \in \R, \\
\frac{1}{2^{\frac{2}{3}}} {\sf K}_{\Ai, 4}\left( \frac{x}{2^{\frac{2}{3}}}, \frac{y}{2^{\frac{2}{3}}} \right) &= \frac{1}{2} \Upsilon \left( \begin{bmatrix} J_4(x, y) & -\frac{\partial}{\partial y}J_4(x, y) \\ \int_y^x J_4(u, y)du & J_4(y, x) \end{bmatrix} \right), \quad x, y \in \R,
\end{align*}
where $\e(t)$ is the function such that $\e(t)=-1/2$ $(t > 0)$, $\e(t)=0$ $(t = 0)$, $\e(t)=1/2$ $(t < 0)$. 
Note that the quaternion matrices $[{\sf K}_{\ast, \beta}(x^i, x^j)]_{1 \leq i, j \leq {\sf n}}$ with $\ast = {\rm sin}$ or $\Ai$ and $\beta=1, 4$ are self-dual. In particular, $D(x-y)$, $D(2(x-y))$, $-\frac{\partial}{\partial y}J_1(x, y)$, $-\frac{\partial}{\partial y}J_4(x, y)$ are antisymmetric functions. Hence, the assumptions in Lemma \ref{L:51b} hold for ${\sf K}_{\ast, \beta}$ with $\ast = {\rm sin}$ or $\Ai$ and $\beta=1, 4$.

By Proposition \ref{P:51a} and the definition of the correlation functions for QDRPFs, we can check \eqref{:51n} for these RPFs. Therefore, by Corollary \ref{C:51a}, we can see that \As[NCL] holds for these RPFs. 
\end{example}

\subsection{Sufficient conditions for \As[NEX]} \label{SS:52}

In this subsection, we discuss sufficient conditions for \As[NEX].
We write $(X_t, \Xi_t^{[1]})$ for $(X_t^{[1]}, \Xi_t^{[1]})$ to simplify the notation.

\begin{lemma}\label{L:52a}
Suppose that the conditions in Theorem \ref{T:31a} hold.
Assume that $p$ is translation-invariant, i.e. $p(x,y)= p(y-x)$.
Then, 
\begin{align}\label{:52a}
\P^{[1]}_{(x, \xi)} \left( \sup_{t\in [0,T]} |X_t| < \infty \ \text{for all $T \in \N$} \right) = 1 \quad \text{for $\mu^{[1]}$-a.s. $(x, \xi) \in S \times \M$}. 
\end{align}
In particular, \As[NEX] holds.
\end{lemma}

\begin{proof}
Let $\tau_M(X)= \inf\{t>0 ; |X_t|>M\}$, $M\in\N$, and $\tau_\infty (X)=\lim_{M\to\infty}\tau_M(X)$. Then, by the argument in the proof of Theorem \ref{T:32a},
we have 
\begin{align*}
&X_t=X_0+\int_0^t \int_S  \int_0^\infty N^1(dsdudr) \;
u a(u,r,X_{s-},\Xi_{s-}^{[1]}), \quad 0\le t <\tau_\infty(X).
\end{align*}
Set
$$
a_{\mathsf{M}}(u,r,x,\eta)= \mathbf{1}\left(0\le r \le p(x,x+u)\right)
$$
and 
$$
a_{\mathsf{A}}(u,r,x,\eta)= \mathbf{1}\left(p(x,x+u) \le r \le p(x,x+u)\left(1+\frac{\rho^1(x+u)}{\rho^1(x)}
\frac{d\mu_{x+u}}{d\mu_x}(\eta)\right)\right)
$$
We introduce the processes defined by
\begin{align*}
&X_t^{(\mathsf{M})}=\int_0^t \int_S  \int_0^\infty N(dsdudr) \;
u a_\mathsf{M}(u,r,X_{s-},\Xi_{s-}^{[1]}),
\\
&X_t^{(\mathsf{A})}=\int_0^t \int_S  \int_0^\infty N(dsdudr) \;
u a_\mathsf{A}(u,r,X_{s-},\Xi_{s-}^{[1]}).
\end{align*}
Obviously, 
$$
X_t=X_0+X^{(\mathsf{M})}_t+X^{(\mathsf{A})}_t, \quad 0\le t <\tau_\infty(X).
$$
and $X^{(\mathsf{M})}$ is the jump-type Markov process with the generator
$$
Lf(x)=\int_S dy \; p(x,y)\{f(y)-f(x)\}=\int_S dy \; p(y-x)\{f(y)-f(x)\}
$$
and is conservative by virtue of \As[p.1] and \As[p.2]. 
We introduce a Markov process $(P, Y_t)$ with the generator $L$ starting from $0$.
By simple observation, we see that when $X_{t-}=x$ and $X_t=y$,
$\Xi_t^{[1]}=\Xi_{t-}^{[1]}$ and
\begin{enumerate}
\item 
$X_t^{(\mathsf{M})}-X_{t-}^{(\mathsf{M})}=y-x$ with probability
$
\displaystyle{\frac{\rho^1(x)d\mu_x}{\rho^1(x)d\mu_x+\rho^1(y)d\mu_y}(\Xi_t^{[1]})}$,

\item $X_t^{(\mathsf{A})}-X_{t-}^{(\mathsf{A})}=y-x$ with probability
$
\displaystyle{\frac{\rho^1(y)d\mu_y}{\rho^1(x)d\mu_x+\rho^1(y)d\mu_y}(\Xi_t^{[1]})}$.
\end{enumerate}

For a fixed $T>0$ we set $\hat{X}_t=X_{T-t}$ and $\hat{\Xi}_t = \Xi_{T-t}$.
Note that $\mu^{[1]}$ is a reversible measure of $(X_t,\Xi_t^{[1]})$.
Then we can represent process $\hat{X}$ as
\begin{align*}
&\hat{X}_t=\hat{X}_0+\int_0^t \int_S  \int_0^\infty \tilde{N}(dsdudr) \;
u a(u,r,\hat{X}_{s-},\hat{\Xi}_{s-}^{[1]}),
\quad 0\le t < \hat{\tau}_\infty(\hat{X})
\end{align*}
with another Poisson RPF $\tilde{N}$. Here, $\hat{\tau}_M(\hat{X})= \inf\{t>0 ; |\hat{X}_t|>M\}$, $M\in\N$, and $\hat{\tau}_\infty (\hat{X})=\lim_{M\to\infty}\hat{\tau}_M(\hat{X})$. 
We define the process $\hat{X}^{(\mathsf{M})}$ and $\hat{X}^{(\mathsf{A})}$ as
$$
\hat{X}_t^{(\mathsf{M})}=X_{T-t}^{(\mathsf{A})}-X_{T}^{(\mathsf{A})} 
\quad \text{ and } \quad 
\hat{X}_t^{(\mathsf{A})}=X_{T-t}^{(\mathsf{M})}-X_{T}^{(\mathsf{M})}.
$$
When $\hat{X}_{(T-t)-}=y$ and $\hat{X}_{T-t}=x$,
$\hat{\Xi}_{T-t}^{[1]}=\hat{\Xi}_{(T-t)-}^{[1]}=\Xi_t^{[1]}=\Xi_{t-}^{[1]}$ and
\begin{enumerate}
\item 
$\hat{X}_{T-t}^{(\mathsf{M})}-\hat{X}^{(\mathsf{M})}_{(T-t)-}=x-y$ with probability
$
\displaystyle{\frac{\rho^1(y)d\mu_y}{\rho^1(x)d\mu_x+\rho^1(y)d\mu_y}(\Xi_t^{[1]})}$,

\item $\hat{X}_{T-t}^{(\mathsf{A})}-\hat{X}^{(\mathsf{A})}_{(T-t)-}=x-y$ with probability
$
\displaystyle{\frac{\rho^1(x)d\mu_x}{\rho^1(x)d\mu_x+\rho^1(y)d\mu_y}(\Xi_t^{[1]})}$.
\end{enumerate}
Then, we see that $X^{(\mathsf{M})}_0=0$, $\hat{X}^{(\mathsf{M})}_0=0$, and
\begin{align}\label{:52b}
X_t-X_0= X^{(\mathsf{M})}_t + \hat{X}^{(\mathsf{M})}_{(T-t)-}-\hat{X}_T^{(\mathsf{M})},
\end{align}
and $\hat{X}^{(\mathsf{M})}$ is also a jump-type Markov process  with the generator $L$ starting from $0$, that is, the processes $X^{(\mathsf{M})}$, $\hat{X}^{(\mathsf{M})}$, and $Y$ are identical in  distribution.

Let $\mathbb{M}^{[1]}$ be the measure on $S\times \M$ defined by (\ref{:31f}). We prove that
\begin{align} \label{:52c}
\lim_{R\to\infty}\mathbb{M}^{[1]}(X_0\in U_r, \; \sup_{t\in [0,T]}|X_t|>R)=0.
\end{align}

We divide the event $\{X_0\in U_r, \; \sup_{t\in [0,T]}|X_t|>R\}$ into infinitely many events, and have that, for any $\delta>0$,
\begin{align} \notag
&\mathbb{M}^{[1]}(X_0\in U_r, \; \sup_{t\in [0,T]}|X_t|>R)
\\ \notag
&= \sum_{m=0}^\infty\mathbb{M}^{[1]}(X_0\in U_r, \; \sup_{t\in [0,T]}|X_t|>R, \; m <  |X_T|\le m+1)
\\ \notag
&= \sum_{m=0}^\infty \mathbb{M}^{[1]}(X_0\in U_r, \;  \sup_{t\in [0,T]}|X_t|>R, \; m <  |X_T|\le m+1, \; \sup_{t\in[0,T]}|X^{(\mathsf{M})}_t|> R^\delta )
\\ \notag
&+\sum_{m=0}^\infty\mathbb{M}^{[1]}(X_0\in U_r,  \sup_{t\in [0,T]}|X_t|>R, \; m <  |X_T|\le m+1, \; \sup_{t\in [0,T]}|X^{(\mathsf{M})}_t|\le R^\delta)
\\ \label{:52d}
&\equiv I_1 + I_2.
\end{align}
Then, we have
\begin{align}\notag 
I_1 
&\le \mathbb{M}^{[1]}(X_0\in U_r, \;  \sup_{t\in[0,T]}|X^{(\mathsf{M})}_t|> R^\delta)
\le \rho^1[U_r] P(\sup_{t\in[0,T]}|Y_t|> R^\delta)
\end{align}
and, from (\ref{:52b})  
\begin{align} \notag
I_2
&\le \sum_{m=0}^{R-1} 
\mathbb{M}^{[1]}( \sup_{t\in [0,T]}|\hat{X}_t^{(\mathsf{M})}|>R-R^\delta, \; m\le |\hat{X}_0|< m+1, \; m - R^\delta - r< |\hat{X}^{(\mathsf{M})}_T|\le m+ R^\delta +r)
\\ \notag
&+ \sum_{m=R}^\infty 
\mathbb{M}^{[1]}( m\le |\hat{X}_0|< m+1, \; m - R^\delta - r< |\hat{X}^{(\mathsf{M})}_T|\le m+ R^\delta +r)
\\ \notag
&\le \sum_{m=0}^{R-1} 
\rho^1[U_{m+1}\setminus U_m]
P(\sup_{t\in [0,T]}|Y_t|>R-R^\delta, \; m - R^\delta - r< |Y_T|\le m + R^\delta +r)
\\ \notag
&+\sum_{m=R}^\infty \rho^1[U_{m+1}\setminus U_m] P(m - R^\delta - r< |Y_T|\le m+ R^\delta +r)
\\ \notag
&\le \rho^1[U_R]P(\sup_{t\in [0,T]}|Y_t|>R-R^\delta)
\\ \notag
&+\sum_{m=R}^\infty \rho^1[U_{m+1}\setminus U_m] \int_{m - R^\delta - r< |x|\le m+ R^\delta +r}p_T(x)dx
\\ \label{:52f}
&\equiv I_{21} +I_{22},
\end{align}
where $\rho^1[A]=\int_A \rho^1(x) dx$ for $A\in \Borel(\R^d)$ and $p_T$ is the density of $Y_T$.
Because $Y_t$ is a jump process with conditions \As[p.1] and \As[p.2], we see that
\begin{align}\label{:52g}
&p_T(x) = O(|x|^{-d-\alpha}), \quad |x|\to\infty.
\end{align}
Using a similar argument to show the reflection principle of a simple random walk, 
we see that 
\begin{align*}
&P(\sup_{t\in[0,T]} Y_t^j > R) \le 2 P(Y_T^j > R), \quad R>0
\end{align*}
for $j=1,2,\dots,d$, and so 
\begin{align}\label{:52h}
&P(\sup_{t\in [0,T]} |Y_t| > R) \le 2d P(Y_T^j > R/\sqrt{d}), \quad R>0.
\end{align}
Then, from (\ref{:52h}) there exists some $\Ct \label{5_23}>0$ such that
\begin{align*}
P(\sup_{t\in [0,T]}|Y_t|>R-R^\delta) 
\le \cref{5_23}\int_{|x|>R-R^\delta-r} p_T(x)dx.
\end{align*}
Then, by \As[A.2],
\begin{align}\label{:52k}
I_{21} 
\le \cref{5_23} R^\kappa \int_{|x|>R-R^\delta-r} p_T(x)dx.
\end{align}
By simple calculation with \As[A.2] there exists some $\Ct \label{5_24}>0$
\begin{align} \label{:52l}
I_{22}&\le 
\cref{5_24} \int_{|x|>R-R^\delta-r}|x|^{\delta +\kappa} p_T(x)dx 
\end{align}
Combining  (\ref{:52d})--(\ref{:52l}), we have
\begin{align} \notag
\mathbb{M}^{[1]}(X_0\in U_r, \; \sup_{t\in [0,T]}|X_t|>R) &\le \rho^1[U_r] P(\sup_{t \in [0,T]}|Y_t|> R^\delta)
\\ \label{:52m}
& + (\cref{5_23}+\cref{5_24}) \int_{|x|>R-R^\delta-r}|x|^{\delta + \kappa}p_T(x)dx.
\end{align}
Then from (\ref{:52g}), for any $\delta\in (0,\alpha-\kappa)$, we have
\begin{align}\label{:52n}
&\lim_{R\to\infty} \int_{|x|>R-R^\delta-r}|x|^{\delta + \kappa} p_T(x)dx =0.
\end{align}
From (\ref{:52m}) and (\ref{:52n}), we obtain (\ref{:52c}), and (\ref{:52a}) follows. 
This completes the proof of Lemma \ref{L:52a}.
\end{proof}

\subsection{Sufficient conditions for \As[NBJ]}\label{SS:53}

In this subsection, we discuss sufficient conditions for \As[NBJ].
We set 
\begin{align*}
&\mathfrak{N}(p,q) = \{
\xi=\sum \delta_{x^j}\in \M : |x^i-x^j|\ge 2^{-p} \mbox{ if $x^i, x^j \in U_q$ }
\},
\quad \mathfrak{N}=\bigcap_{q\in \N}\bigcup_{p\in\N}\mathfrak{N}(p,q).
\end{align*}
We first show the following lemma.

\begin{lemma}\label{L:53a}
Assume that the conditions in Theorem \ref{T:31a} and \As[NCL] hold.
Then, for any $q, T \in \N$,
\begin{align}\label{:53b}
\P_{\mu}( \text{$\exists p \in \N$ \ s.t \ $\Xi_t\in\mathfrak{N}(p,q)$, $\forall t\in [0,T]$} )=1,
\end{align}
that is, for $\P_\mu$-a.s. $\omega$, there exists $p=p(\omega)$
such that $\Xi_t(\omega)$ does not hit the set $\mathfrak{N}(p,q)^c$ for $t\in [0,T]$.
\end{lemma}
\begin{proof} Fix $q, T \in \N$. 
Let $\tau_p(\omega) \equiv \inf\{ t \in [0, \infty) ; \Xi_t(\omega) \notin \mathfrak{N}(p, q) \}$ for $p \in \N$. By the definition, $\{ \mathfrak{N}(p, q) \}$ is increasing in $p$, and so $\tau_p(\omega)$ is as well. Hence, the limit $\tau_{\infty}(\omega) \equiv \displaystyle{\lim_{p \to \infty}} \tau_p(\omega) \in [0, \infty]$ exists. Equation \eqref{:53b} means that
\begin{align}\label{:53b2}
\P_{\mu} (\tau_\infty >T )=1.
\end{align}

Suppose that $\omega$ satisfies $\tau_{\infty}(\omega) \leq T$ and $\displaystyle{\lim_{p \to \infty} \Xi_{\tau_p}(\omega)} = \Xi_{\tau_{\infty}}(\omega)$ with the vague topology. 
We often omit $\omega$ for simplicity of notation, if this will not cause confusion.
From $\tau_{\infty} \leq T$ we have $\tau_p \leq T$ for any $p \in \N$. 
Combining this with the definition of $\tau_p$, we see that, for any $p \in \N$, there exists a sequence $\{ t_n^{(p)} \}_{n \in \N}$ such that $t_n^{(p)} \in \{ t; \Xi_t \notin \mathfrak{N}(p, q) \}$, $t_n^{(p)} \downarrow \tau_p$. 
As $\Xi$ has the right continuous path, we have
\begin{equation} \label{:53c}
\Xi_{\tau_p} 
= \lim_{n \to \infty} \Xi_{t_n^{(p)}} \in \overline{\mathfrak{N}(p, q)^c}, \quad \forall p \in \N. 
\end{equation}
Here, $\overline{\mathfrak{N}(p, q)^c} = \{ \xi \in \mathfrak{M} ; |x^i-x^j| \leq 2^{-p} \ \text{for some $x^i, x^j \in U_q$} \}$.

We assume that
$\Xi_{\tau_{\infty}} \notin \{ \xi \in \mathfrak{M} ; x^i = x^j \ \text{for some $x^i, x^j \in U_q$} \}$. 
Set $\Xi_{\tau_{\infty}} = \sum_i \delta_{z^i}$. 
Because the number of particles in $U_{q+1}$ of $\Xi_{\tau_{\infty}}$ is finite we can label particles in $U_q$ as $z^1, \ldots, z^{\n}$ and those in $U_{q+1} \setminus U_q$ as $z^{\n+1}, \ldots, z^{\hat{\n}}$.  
Noting that $U_q$ is a closed set, there exists some $\e>0$ such that $|z^i-z^j| > \e$ for $1\le i,j \le \n$ and $U_{\e}(z^i) \subset U_q^c$ for $i \geq \n+1$. 
As $\Xi_{\tau_p} \to \Xi_{\tau_\infty}$, $p\to\infty$ in the vague topology, we can set $\Xi_{\tau_p} = \sum_i \delta_{z^{p, i}}$ such that $\lim_{p \to \infty} z^{p, i} = z^i$ for any $i\in \N$. 
Then there exists some $M \in \N$ such that $|z^i-z^{p, i}| < \e/4$ for $1\le i \le \hat{\n}$ and $z^{p, i} \notin U_q$ for $i \geq \hat{\n}+1$, for any $p\ge M$. Thus, by the triangle inequality, we see that
\begin{align*}
|z^{p, i}-z^{p, j}| > |z^j-z^i| - |z^j-z^{p, j}| - |z^i-z^{p, i}| >  \frac{\e}{2}, 
\end{align*}
for $1 \leq i, j \leq \n$, $p\ge M$. 
Moreover, because $U_{\e}(z^i) \subset U_q^c$ and $|z^i-z^{p, i}| < \e/4$ for any $\n +1 \leq i \leq \hat{\n}$, we see that $z^{p, i} \notin U_q$. 
Take $p \in \N$ satisfying $p\ge M$ and $\e/2 > 2^{-p}$. Then, we have $|z^{p, i}-z^{p, j}| > 2^{-p}$ for any $i, j = 1, \ldots, \n$ and $z^{p, i} \notin U_q$ for any $i \geq \n+1$. Thus, we have $|z^{p, i}-z^{p, j}| > 2^{-p}$ for any $z^{p, i}, z^{p, j} \in U_q$. Therefore, $\Xi_{\tau_p} \in \mathfrak{N}(p, q)$. This contradicts \eqref{:53c}, and so we have $\Xi_{\tau_{\infty}} \in \{ \xi \in \mathfrak{M} ; x^i = x^j \ \text{for some $x^i, x^j \in U_q$} \}$. Then, we see that $\Xi_{\tau_{\infty}} \in \{ \xi \in \mathfrak{M} ; x^i = x^j \ \text{for some $x^i, x^j \in U_q$} \}$
if $\tau_{\infty} \leq T$ and $\displaystyle{\lim_{p \to \infty}} \Xi_{\tau_p} = \Xi_{\tau_{\infty}}$. 
Hence, by \As[NCL] and the quasi-left continuity of $\Xi$,  we have 
\begin{equation}\nonumber 
\P_{\mu} ( \tau_{\infty} \leq T )
=\P_{\mu} ( \lim_{p \to \infty} \Xi_{\tau_p} = \Xi_{\tau_{\infty}}, \tau_{\infty} \leq T ) = 0. 
\end{equation}
We then obtain \eqref{:53b2}, and \eqref{:53b} follows. This completes the proof.
\end{proof}

\begin{lemma}\label{L:53b}
Assume that the conditions in Theorem \ref{T:31a} and \As[NCL] hold.
Then \As[NBJ] holds. 
\end{lemma}

\begin{proof}
Suppose that 
$\P_{\mu}( \mathsf{m}_{q,T}(\mathfrak{l}_{\rm path}(\Xi)) = \infty )>0$
for some $q,T\in\N$.
From Lemma \ref{L:53a}, with positive probability,
we can take $\omega$ and $p=p(\omega)\in \N$ such that  
\begin{align}\label{:53g}
&\mathsf{m}_{q,T}(\mathfrak{l}_{\rm path}(\Xi(\omega)))=\infty,
\\ \label{:53h}
&\Xi_t(\omega)\in \mathfrak{N}(p,q), \quad  t\in [0,T].
\end{align}
From \eqref{:53g} for any small $\delta>0$ there exists $x_0\in U_q$ such that $U_\delta(x_0)$ intersects $U_q$, and
$$
\mathbb{I}(\delta,x_0,\omega) =\{ j\in \N ; X_t^j(\omega) \in U_\delta(x_0) \mbox{ for some $t\in [0,T]$ }\}
$$
is an infinite set. Take $\delta$ such that $2^{-p}> 5\delta$.
We take intervals $[a_j,b_j]\subset [0,T]$ such that
$$
[a_j,b_j] \subset \{ t\in [0,T] ; X_t^j(\omega) \in U_{2\delta}(x_0)\}, \quad j\in \mathbb{I}(\delta,x_0,\omega).
$$
From \eqref{:53h}, for each time $t\in [0, T]$, the number of points $\Xi_t(U_{2\delta}(x_0))$ in $U_{2\delta}(x_0)$ is at most $1$. Therefore, the intervals $[a_j,b_j]$, $j\in \mathbb{I}(\delta,x_0)$ are disjoint. Let $\varphi$ be a smooth function on $\R^d$ satisfying $0 \le \varphi (x) \le 1$, $x\in\R$, and
$$
\varphi(x)= \begin{cases}
1 & \mbox{ if $x\in U_\delta(x_0)$}
\\
0 & \mbox{ if $x\notin U_{2\delta}(x_0)$}.
\end{cases}
$$
Set $F(\xi)= \langle \varphi, \xi\rangle$. Then, $F(\Xi_t(\omega))$ is not a right continuous function with left limits in $t$, which contradicts the fact that $\Xi$ is a Hunt process.
This completes the proof.
\end{proof}

\subsection{Sufficient conditions for \as[IFC]}\label{SS:54}

Let $\X=\lab_{\rm path}(\Xi)$ in Theorem \ref{T:31c}. 
From Corollary \ref{C:32a}, $\X$ is a weak solution of \eqref{ISDE} with \eqref{:32e} and \eqref{:32f}.
In this subsection, we give sufficient conditions for \As[IFC]. 
Set
$$
\M_{\rm s.i}^{[\m]}=\{
(\x_\m,\eta)\in S^{\m}\times \M : \mathfrak{u}(\x_\m)+\eta \in \M_{\rm s.i.}
\}
$$
and
$$
R_{p,q}(\eta)=\{\x_\m \in \tilde{U}_q^\m : \min_{1\le j < k \le \m}|x_\m^j-x_\m^k| > 2^{-p}, \:
\min_{1\le j\le \m, y\in\eta}|x_\m^j-y| > 2^{-p} 
\}, 
$$
where $\tilde{U}_q = \{ |x|<q \}$. 
Let $\bd{a}_n=\{ a_{n,r}\}_{r\in\N}=\{n2^{(d+\beta)r}\}_{r\in \N}$ as in \eqref{:41b}, and set $\bd{a} = \{\bd{a}_n\}_{n\in\N}$.
For $(p,q,n)\in \N^3$, set
$$
\mathfrak{K}[\bd{a}]_{p,q,n}=\{
(\x_\m,\eta)\in \M^{[\m]}_{\rm s.i.} : \x_\m\in R_{p,q}(\eta), \: \eta \in \M[\bd{a}_n^+] \},
$$
and set
\begin{equation*}
\mathfrak{K}[\bd{a}]_{q,n}=\bigcup_{p\in\N}\mathfrak{K}[\bd{a}]_{p,q,n},
\quad
\mathfrak{K}[\bd{a}]_q=\bigcup_{n\in\N}\mathfrak{K}[\bd{a}]_{q,n},
\quad
\mathfrak{K}\equiv \mathfrak{K}[\bd{a}]=\bigcup_{q\in\N}\mathfrak{K}[\bd{a}]_q.
\end{equation*}
The following lemma is obtained in the same way as \cite[Lemma 8.1]{o-t.tail}
by using \As[NCL] and Lemma \ref{L:41a} and \ref{L:41c} (see also Lemma \ref{L:52a}).
%
\begin{lemma}\label{L:54a}
Suppose that the assumptions in Theorem \ref{T:31c} hold.
For each $\m\in \N$ 
\begin{align}\notag
\P_\xi ( \lim_{q\to\infty}\lim_{n\to\infty}\lim_{p\to\infty}
\tau_{\mathfrak{K}[\bd{a}_n]_{p,q,n}}(\X^{[\m]},\Xi^{[\m]})=\infty )=1, \quad \mbox{for $\mu$-a.s. $\xi$,}
\end{align}
where $\tau_{\mathfrak{K}[\bd{a}_n]_{p,q}}$ is the exit time from $\mathfrak{K}[\bd{a}]_{p,q,n}$.
\end{lemma}
Let $\{\mathfrak{I}_k^{[\m]} \}_{k\in\N}$ be an increasing sequence of closed subsets of $S^{\m}\times \M$
satisfying
\begin{align}\notag 
\mathrm{Cap}^{\mu^{[\m]}}\Big(
\Big(\bigcup_{k\in\N}\mathfrak{I}_k^{[\m]}\Big)^c
\Big)=0.
\end{align}
Let $\Pi : S^{\m}\times \M \to  \M$ be the projection such that $\Pi(\x,\eta)=\eta$.
Then, 
\begin{align}\notag
\Pi (\mathfrak{K}[\bd{a}]_{p,q,n}\cap \mathfrak{I}_k^{[\m]})
=\{\eta\in \M : \mathfrak{K}[\bd{a}]_{p,q,n} \cap \mathfrak{I}_k^{[\m]} \cap (S^{\m}\times \{\eta \})\not= \emptyset \}.
\end{align} 
Let $p,q,n, k\in\N$ and $L>0$.
For $\x_\m, \y_\m \in S^{\m}$, $\eta\in \M$, and $1\le j \le \m$, we set
$$
C_{p,q,j}^L[\x_\m,\y_\m,\eta]
=\int_{|u|\le L} du \; |u||c(x_\m^j,x_\m^j+u,\eta+\mathfrak{u}(\x_\m)) -c(y_\m^j,y_\m^j+u,\eta+\mathfrak{u}(\y_m))|H_{\x_\m,\y_\m,\eta}^j(u),
$$
where 
$$
H_{\x_{\m},\y_{\m},\eta}^j(u)= \mathbf{1}\Big(
(\x_\m,\eta)\underset{p,q}{\sim} (\x_\m+\u^j,\eta), \;
(\y_\m,\eta) \underset{p,q}{\sim} (\y_\m+\u^j,\eta) \Big)
$$
with $\u^j= (0,\dots,0,\underset{j\mathrm{th}}u, 0,\dots 0) \in \R^{\m d}$, $1\le j \le \m$,
and $(\x_\m,\eta)\underset{p,q}{\sim} (\y_\m,\eta)$ means that $\x_\m$ and $\y_\m$ belong to the same connected component of $R_{p,q} (\eta)$. 
We also set
\begin{align}\notag
&\mathsf{C}_{1,L}(p,q,n,k)
=\sup \bigg\{
\frac{\max_{1\le j\le \m} C_{p,q,j}^L[\x_\m,\y_\m,\eta]}{|\x_\m-\y_\m|} 
; 
\\ \notag
&\qquad \qquad \qquad
\eta\in\Pi (\mathfrak{K}[\bd{a}]_{p,q,n}\cap \mathfrak{I}_k^{[\m]}), \;  
\x_\m\not=\y_\m, \; (\x_\m,\eta)\underset{p,q}{\sim} (\y_\m,\eta), 
\bigg\}
\\ \notag
&\mathsf{C}_{2,L}(p,q,n,k)
=\sup \bigg\{
\frac{\displaystyle{\max_{1\le j\le \m}}\; |\mathsf{d}^\mu (x_\m^j,x_\m^j+u,\eta+\mathfrak{u}(\x_\m^{\hiku{j}})) -2|}{|u|} 
; 
\\ \notag
&\qquad \qquad\qquad \qquad
\eta\in\Pi (\mathfrak{K}[\bd{a}]_{p,q,n} \cap \mathfrak{I}_k^{[\m]}), \;  (\x_\m, \eta) \underset{p,q}{\sim}(\x_\m+\u^j, \eta), \; |u|\le L
\bigg\}
\end{align}
Recall that $\mathsf{d}^\mu (x_\m^j,x_\m^j,\eta+\mathfrak{u}(\x_\m^{\hiku{j}}))=2$.

We set
$$
\Kpqn =\bigcup_{\eta\in \Pi(\mathfrak{K}[\bd{a}]_{p,q,n} \cap \mathfrak{I}_k^{[\m]})}
R_{p,q} (\eta)\times \{\eta\}
$$
For $(\X, \bN)=(\frak{l}_{\rm path}(\Xi), \bN)$ we make the following assumption.
\vskip 3mm

\noindent
\As[B1] For each $\m\in\N$, there exist a $\mu^{[\m]}$-version $\tilde{c}$ of the speed function $c$ 
and an increasing sequence $\{\mathfrak{I}_k^{[\m]}\}_{k\in\N}$ of closed subset of $S^{\m}\times \M$ such that 
\begin{align}\label{:53k}
&\P_\xi (\lim_{q\to\infty}\lim_{n\to\infty}\lim_{p\to\infty}\lim_{k\to\infty}\tau_{\Kpqn}(\X^{[\m]},\Xi^{[\m]})=\infty)=1, \quad \mbox{for $\mu$-a.s. $\xi$}.
\end{align}
For each $k\in\N$ and $p,q,n\in\N$
\begin{align} \label{:53m}
&\mathsf{C}_{1,L}(p,q,n,k) <\infty \quad \mbox{ for any $L\in\N$,}
\\ \label{:53n}
&\mathsf{C}_{2,L}(p,q,n,k) <\infty \quad \mbox{ for any $L\in\N$,}
\end{align}

\begin{lemma} \label{L:54b}
Let $\X=\lab_{\rm path}(\Xi)$ be the labeled process in Theorem \ref{T:31c}.
Assume that \As[B1] holds and $p$ is translation-invariant, i.e., $p(x,y)=p(y-x)$. Then

\noindent (i) the pathwise uniqueness of solutions of SDE (\ref{:33b})--(\ref{:33d}) holds.

\noindent (ii) \As[IFC] holds.
\end{lemma}
\begin{proof}
For simplicity, we assume $\m=1$. The general case $\m\in\N$ can be proved by the same argument. When $\m = 1$, SDE (\ref{:33b})--(\ref{:33d}) ca be rewritten as 
\begin{align}
&Y_t=Y_0+\int_0^t \int_S  \int_0^\infty N^1(dsdudr) \;
u a(u,r,Y_{s-},\Xi^{[1]}_{s-}), \label{5.34N}
\\ 
&Y_0=\x_1=x_{1}, \label{5.35N}
\\ 
&Y_t \in \mathbf{S}_{\rm SDE}^1(\X_t), \; \mbox{ for any $t>0$}. \label{5.36N}
\end{align}
It is obvious that $X = X^1$ is a solution. Let $Y$ be another solution of \eqref{5.34N}--\eqref{5.36N}.  
Take $L=2q$. Set
\begin{align*}
&\tau_{p,q,n,k}= \tau_{\Kpqnmone}(X,\Xi^{[1]})\wedge \tau_{\Kpqnmone}(Y,\Xi^{[1]}).
\end{align*}
Then,
\begin{align*}
&\E_\xi [|X_{t\wedge \tau_{p,q,n,k}}-Y_{t\wedge \tau_{p,q,n,k}}|]
\\ \nonumber
&\le \E_\xi\Big[\int_0^{t\wedge \tau_{p,q,n,k}} \int_S  \int_0^\infty dsdudr \; |u| |a(u,r,X_{s-},\Xi^{[1]}_{s-})-a(u,r,Y_{s-},\Xi^{[1]}_{s-})|\Big]
\\  \nonumber
&= \E_\xi\Big[\int_0^{t\wedge \tau_{p,q,n,k}} \int_S   dsdu \; |u| |c(X_{s-},X_{s-}+u,\Xi^{[1]}_{s-})-c(Y_{s-},Y_{s-}+u,\Xi^{[1]}_{s-})|\Big]
\\
&\le \mathsf{C}_{1,L}(p,q,n,k)\int_0^t ds \ \E_\xi\Big[|X_{s\wedge \tau_{p,q,n,k}}-Y_{s\wedge \tau_{p,q,n,k}}|\Big].
\end{align*}
Here, we use (\ref{:53m}) in the last equation.
Hence, for $\P_\xi$ a.s.,
\begin{align*}
& X_{t} =Y_{t}, \quad t\in [0, \tau_{p,q,n,k})\quad \mbox{and} \quad
\tau_{\Kpqnmone}(Y,\Xi^{[1]})=\tau_{\Kpqnmone}(X,\Xi^{[1]}).
\end{align*}
From (\ref{:53k}), we obtain (i). 

For (ii), it is sufficient to show the existence of a strong solution.
For $\delta>0$, we consider the following SDE:
\begin{align*}
&Y_t^\delta=Y_0+\int_0^t \int_{|u|>\delta}  \int_0^\infty N^1(dsdudr) \;
u a(u,r,Y_{s-}^{\delta},\Xi^{[1]}_{s-}),
\\ 
&Y_0=x\in U_q,
\\ 
&Y_t^\delta \in \mathbf{S}_{\rm SDE}^\1(\X_t) \; \mbox{ for any $t>0$}.
\end{align*}
As $Y^{\delta}_t$, $t\in [0, \tau_{\langle \mathfrak{K}[\bd{a}]_{p,q,n} \cap \mathfrak{I}_k^{[1]}\rangle }(Y^{\delta},\Xi^{[1]})]$, almost surely has a finite number of jumps, we see that it is a functional of $(\bN^{1},\Xi^{[1]})$.
Set 
\begin{align*}
&\tau_{p,q,n,k}^{\delta}= \tau_{\langle \mathfrak{K}[\bd{a}]_{p,q,n} \cap \mathfrak{I}_k^{[1]}\rangle }(X,\Xi^{[1]})\wedge \tau_{\Kpqnponemone}(Y^{\delta},\Xi^{[1]})
.\end{align*}
From \eqref{:53m} and \eqref{:53n},
\begin{align*}
&\E_\xi[|X_{t\wedge \tau_{p,q,n,k}^{\delta}}-Y_{t\wedge \tau_{p,q,n,k}^{\delta}}^{\delta}|]
\\ \nonumber
&\le \E_\xi\Big[\big|\int_0^{t\wedge \tau_{p,q,n,k}^{\delta}} \int_S  \int_0^\infty dsdudr \; u(a(u,r,X_{s-},\Xi_{s-}^{[1]})-a(u,r,Y_{s-}^{\delta},\Xi^{[1]}_{s-}))\big|\Big]
\\  \nonumber
&\leq \E_\xi\Big[\int_0^{t\wedge \tau_{p,q,n,k}^{\delta}} \int_{|u|>\delta}   dsdu \; |u| |c(X_{s-},X_{s-}+u,\Xi^{[1]}_{s-})-c(Y_{s-}^{\delta},Y_{s-}^{\delta}+u,\Xi^{[1]}_{s-})|\Big]
\\
&+\E_\xi\Big[\int_0^{t\wedge \tau_{p,q,n,k}^{\delta}} \int_{|u|\le\delta}   dsdu \; |u|c(X_{s-},X_{s-}+u,\Xi^{[1]}_{s-})\Big]
\\
&\le \mathsf{C}_{1,L}(p,q,n,k)\int_0^t ds \E_\xi [|X_{s\wedge \tau_{p,q,n,k}^{\delta}}-Y_{s\wedge \tau_{p,q,n,k}^{\delta}}^{\delta}|]
+\mathsf{C}_{2,L}(p,q,n,k)\int_{|u|\le \delta}|u|^2 p(u)du.
\end{align*}
Then, for any $\delta>0$
\begin{align}\label{:53p}
\E_\xi[|X_{t\wedge \tau_{p,q,n,k}^{\delta}}-Y_{t\wedge \tau_{p,q,n,k}^{\delta}}^{\delta}|]
\le \mathsf{C}_{2,L}(p,q,n,k)e^{t\mathsf{C}_{1,L}(p,q,n,k)}\int_{|u|\le \delta}|u|^2 p(u)du.
\end{align}
Hence, from \As[p.2], for any $\varepsilon>0$, we can take $\delta>0$ sufficiently small such that
\begin{align*}
&\E_\xi[|X_{t\wedge \tau_{p,q,n,k}^{\delta}}-Y_{t\wedge \tau_{p,q,n,k}^{\delta}}^{\delta}|]
\le \varepsilon
.\end{align*}
If $|X_{t\wedge \tau_{p,q,n,k}^{\delta}}-Y_{t\wedge \tau_{p,q,n,k}^{\delta}}^{\delta}|< 2^{-(p+1)}$, we see that
\begin{align*}
\tau_{\Kpqnmone}(X,\Xi^{[1]}) < \tau_{\Kpqnponemone}(Y^{\delta},\Xi^{[1]}),
\end{align*}
and so
$\tau_{\Kpqnmone}(X,\Xi^{[1]}) =  \tau_{p,q,n,k}^{\delta}$.
Thus, by Chebyshev's inequality
\begin{align*}
&\P_{\xi}(\tau_{\Kpqnmone}(X,\Xi^{[1]}) =  \tau_{p,q,n,k}^{\delta}) 
\\
&\ge 1- \P_\xi (|X_{t\wedge \tau_{p,q,n,k}^{\delta}}-Y_{t\wedge \tau_{p,q,n,k}^{\delta}}^{\delta}|\ge 2^{-(p+1)})
\ge 1- 2^{p+1}\varepsilon.
\end{align*}
Then, $\tau_{p,q,n,k}^{\delta}$ converges to $\tau_{\Kpqnmone}(X,\Xi^{[1]})$ in probability as $\delta \to 0$.
From \eqref{:53p},
\begin{align*}
&\E_\xi[|X_{t\wedge \tau_{p,q,n,k}^{\delta}}-Y_{t\wedge \tau_{p,q,n,k}^{\delta}}^{\delta}|
:\tau_{\Kpqnmone}(X,\Xi^{[1]}) =  \tau_{p,q,n,k}^{\delta}]
\\
&\le \mathsf{C}_{2,L}(p,q,n,k)e^{t\mathsf{C}_{1,L}(p,q,n,k)}\int_{|u|\le \delta}u^2 p(u)du.
\end{align*}
Then, 
\begin{align}\notag
&\lim_{\delta\to 0}\E_\xi[|X_{t\wedge \tau_{\Kpqnmone}(X,\Xi^{[1]})}-Y_{t\wedge \tau_{\Kpqnmone}(X,\Xi^{[1]})}^{\delta}|]
\\ \notag
&=\lim_{\delta\to 0}\E_\xi[
|X_{ t \wedge \tau_{p,q,n,k}^{\delta} } -Y_{ t\wedge \tau_{p,q,n,k}^{\delta}}^{\delta}|
:\tau_{\Kpqnmone}(X,\Xi^{[1]}) =  \tau_{p,q,n,k}^{\delta}
]
\\\label{:53pp}
&=0.
\end{align}
Because $Y^\delta$ is a functional of $(\bN^{1},\Xi^{[1]})$,
from \eqref{:53k} and \eqref{:53pp}
$X$ is as well.
This completes the proof.
\end{proof}

We next examine sufficient conditions of \As[B1].
We set 
\begin{align}\label{:53q}
&\tilde{\mathrm{d}}^\mu(x,\eta)=\nabla_y\mathsf{d}^\mu (x,y , \eta)\Big|_{y=x}
\quad \mbox{ for } (x,\eta)\in \M^{[1]}_{\rm s.i.}
\end{align}
if the derivative exists.
Note that $\tilde{\mathrm{d}}^\mu(x,\eta)$ coincides with the logarithmic derivative $\mathrm{d}^\mu(x,\eta)$ of $\mu$ \cite{BDO}. 
For $p,q,n, k\in\N$, we set
\begin{align*}
b_j^\m(\x_\m, \eta)= \frac{1}{2}\tilde{\mathrm{d}}^\mu(x_\m^j, \eta + \mathfrak{u}(\x_\m^{\hiku{j}})),
\ 1 \le j \le \m, \quad 
b^{\m}(\x_\m, \eta) = (b_j^\m(\x_\m, \eta))_{j=1}^\m,
\end{align*}
and
\begin{align}\nonumber
&\mathsf{D}_{1}(p,q,n,k)
=\sup \bigg\{ \max_{1\le j\le \m}  
|b_j^\m(\x_\m,\eta)| 
; 
\eta\in\Pi (\mathfrak{K}[\bd{a}_n]_{p,q}\cap \mathfrak{I}_k^{[\m]}), \;  \x_\m \in R_{p,q} (\eta)
\bigg\}
\\ \nonumber
&\mathsf{D}_{2}(p,q,n,k)
=\sup \bigg\{
\frac{|b^{\m}(\x_\m, \eta)-b^{\m}(\y_\m, \eta)|}{|\x_\m-\y_\m|} ; 
 \eta\in\Pi (\mathfrak{K}[\bd{a}_n]_{p,q} \cap \mathfrak{I}_k^{[\m]}), \; 
 \\\nonumber
&\qquad\qquad\qquad\qquad\qquad\qquad\qquad\qquad\qquad\qquad 
\x_\m\not=\y_\m, \; (\x_\m, \eta) \underset{p,q}\sim (\y_\m,\eta)
\bigg\}.
\end{align}
We make the following assumption.

\vskip 3mm
\noindent \As[B2] For each $\m\in\N$, there exist a $\mu^{\m}$-version $\tilde{b}^\m$ of $b^\m$
and an increasing sequence $\{\mathfrak{I}_k^{[\m]}\}_{\k\in\N}$ of closed subsets of $S^{\m}\times \M$ such that
\begin{align} \notag 
&\lim_{k\to\infty} \mathrm{Cap}^{\mu^{[\m]}}(\mathfrak{K}_{p,q,n} \setminus \langle \mathfrak{K}_{p,q,n} \cap \mathfrak{I}_k^{[\m]} \rangle )=0,
\\ \label{:53s}
&\mathsf{D}_{1}(p,q,n,k) <\infty, \quad \mathsf{D}_{2}(p,q,n,k) <\infty.
\end{align}

\begin{lemma}\label{L:54c}
Suppose that $\mu$ has the logarithmic derivative $\mathrm{d}^\mu$ and $p$ is translation-invariant, i.e., $p(x,y)=p(y-x)$. 
Assume that \As[B2] holds. Then, \As[B1] holds.
\end{lemma}

\begin{proof} Let $j=1,2,\dots,\m.$
We set $\x_\m(t)=\x_\m+t\u^j$ and $\y_\m(t)=\y_m+t\u^j$ for $t\in [0,1]$. 
Assume that \As[B2] holds.
Suppose that $\eta\in\Pi (\mathfrak{K}[\bd{a}]_{p,q,n}\cap \mathfrak{I}_k^{[\m]})$, and that
$(\x_\m,\eta)\underset{p,q}{\sim} (\y_\m,\eta)$, $(\x_\m,\eta)\underset{p,q}{\sim} (\x_\m(t),\eta)$ and $(\y_\m,\eta)\underset{p,q}{\sim} (\y_\m(t),\eta)$, $t\in [0,1]$.
Because $(\x_\m(t),\eta)\underset{p,q}{\sim} (\y_\m(t),\eta)$,
from \eqref{:53q} and \eqref{:53s}
\begin{align}\notag
& |\mathsf{d}^\mu(x_\m^j,x_\m^j+u, \eta + \mathfrak{u}(\x_\m^{\hiku{j}}))-\mathsf{d}^\mu(y_\m^j,y_\m^j+u, \eta + \mathfrak{u}(\y_\m^{\hiku{j}}))| 
\\ \notag 
&\le 2|u| \int_0^1  \; |b_j^\m (\x_\m (t), \eta) - b_j^\m (\y_\m (t), \eta )|dt
\\ \label{:53v}
& \le 2\mathsf{D}_2 |u| \int_0^1 |\x_\m(t)-\y_m(t)|dt =  2\mathsf{D}_2|\x_\m -\y_\m| .
\end{align}

Next suppose that $(\x_\m,\eta), (\y_\m,\eta) \in\mathfrak{K}[\bd{a}]_{p-2,q,n} \subset \mathfrak{K}[\bd{a}]_{p,q,n}$
and 
\begin{align}\label{:53u}
 (\x_\m(1),\eta)\underset{p-2,q}{\sim} (\x_\m,\eta)\underset{p-2,q}{\sim} (\y_\m,\eta)\underset{p-2,q}{\sim} (\y_\m(1),\eta).
\end{align}
Noting that the distances of particles in $U_q$ are greater than $2^{-p+2}$,
we take an appropriate sequence $\{\u(k)\}_{k=1}^\ell$ with 
$\u(k)= (0,\dots,0,\underset{j\mathrm{th}}{u(k)}, 0,\dots 0) \in \R^{\m d}$ such that
$\displaystyle{\sum_{k=1}^\ell u(k) =u}$ and
$$
(\x_\m^k(0),\eta)\underset{p,q}{\sim} (\x_\m^k(t),\eta), \quad (\y_m^k(0),\eta)\underset{p,q}{\sim} (\y_\m^k(t),\eta), \quad t\in [0,1], \quad k=1,2,\dots,\ell,
$$
where 
$\x_\m^{1}(0)=\x_m$, $\y_\m^{1}(0)=\y_m$, and
$$
\x_\m^{k}(t)=\x_\m^{k-1}(1)+t\u(k),
\quad\y_\m^k(t)=\y_\m^{k-1}(1)+t\u(k), \quad t\in [0,1].
$$
Hence, using \eqref{:53v}  repeatedly, there exists some $\Ct\label{;54_a}>0$ such that 
\begin{align}\notag
&\max_{1\le j\le \m}  
|\mathsf{d}^\mu(x_\m^j,x_\m^j+u, \eta + \mathfrak{u}(\x_\m^{\hiku{j}}))-\mathsf{d}^\mu(y_\m^j,y_\m^j+u, \eta + \mathfrak{u}(\y_\m^{\hiku{j}}))| 
\\ \label{:53t}
&\qquad \le \cref{;54_a}|\x_\m-\y_\m||u|
\end{align}
for $(\x_\m,\eta), (\y_\m,\eta) \in\mathfrak{K}[\bd{a}_n]_{p-2,q}$ with (\ref{:53u}) and $|u|\le L$.
As $p$ is translation-invariant, we see that
\begin{align}\nonumber
C_{p,q,j}^L[\x_\m, \y_\m, \eta] \le \frac{\cref{;54_a}}{2}|\x_\m -\y_\m| \int_{|u|\le L} du |u|^2 p(u).
\end{align}
Then from \As[p.2] we obtain \eqref{:53m}.

We obtain (\ref{:53n}) from \eqref{:53t}. This completes the proof.
\end{proof}

\begin{remark}\label{R:54a}
From Lemmas \ref{L:54b} and \ref{L:54c}, \As[B2] is 
a sufficient condition to derive \As[IFC].
Sufficient conditions \As[C1] and \As[C2] for \As[B2] are given in \cite[Proposition 11.2]{o-t.tail} for systems of infinite Brownian particles with interactions. 
The argument in \cite[Sections 8.2 and 11.3]{o-t.tail} can be generalized to our case, because it is based on a general theory of Dirichlet forms.
We can check the conditions \As[C1] and \As[C2] for all examples in Section \ref{SS:34} in the same way as in \cite[Section 13]{o-t.tail}.
 \end{remark} 
 
\section{Appendix} \label{S:6}

We introduce some quaternion notations for $2 \times 2$ matrices (see also \cite[Chapter 2.4]{Meh04}). We set 
\begin{equation*}
\1 = \begin{bmatrix} 1 & 0 \\ 0 & 1 \end{bmatrix}, \quad 
{\bf e}_1 = \begin{bmatrix} i & 0 \\ 0 & -i \end{bmatrix}, \quad 
{\bf e}_2 = \begin{bmatrix} 0 & 1 \\ -1 & 0 \end{bmatrix}, \quad 
{\bf e}_3 = \begin{bmatrix} 0 & i \\ i & 0 \end{bmatrix}. 
\end{equation*}
A quaternion $q$ is represented as 
\begin{equation} \label{:6a}
q=q^{(0)}\1 + q^{(1)} {\bf e}_1 + q^{(2)} {\bf e}_2 + q^{(3)} {\bf e}_3, 
\end{equation}
where $q^{(i)}$ are complex numbers. 
We define a map $\Upsilon$ from the set of $2 \times 2$ complex matrices to that of quaternions as follows: 
\begin{equation}\label{:6b}
\Upsilon \left( \begin{bmatrix} a & b \\ c & d \end{bmatrix} \right) = \frac{1}{2}(a+d)\1 - \frac{i}{2}(a-d){\bf e}_1 + \frac{1}{2}(b-c){\bf e}_2 - \frac{i}{2}(b+c){\bf e}_3. 
\end{equation}
For a quaternion $q=q^{(0)}\1 + q^{(1)} {\bf e}_1 + q^{(2)} {\bf e}_2 + q^{(3)} {\bf e}_3$ $\bar{q}$ is defined by $\bar{q} = q^{(0)}\1 - q^{(1)} {\bf e}_1 - q^{(2)} {\bf e}_2 - q^{(3)} {\bf e}_3$. A matrix $A = [a_{ij}]_{i, j}$ whose entries are given by quaternions is said to be a quaternion matrix. A quaternion matrix $A=[a_{ij}]_{i, j}$ is called self-dual if $a_{ij} = \overline{a_{ji}}$ for any $i, j$. For a self-dual $n \times n$ quaternion matrix $A=[a_{ij}]_{i, j}$, we set
\begin{equation}\label{:6c}
\qdet A = \sum_{\sigma \in \mathfrak{S}_{\sf n}} {\rm sign}[\sigma] \prod_{i=1}^{L(\sigma)} [a_{\sigma_i(1)\sigma_i(2)} \cdots a_{\sigma_i(\ell-1)\sigma_i(\ell)} a_{\sigma_i(\ell)\sigma_i(1)}]^{(0)}, 
\end{equation}
where $\mathfrak{S}_{\sf n}$ is the symmetric group of degree $n$, $[ \cdot ]^{(0)}$ is defined in \eqref{:6a}, and we decompose $\sigma$ as products of the cyclic permutations $\{ \sigma_i \}_{i=1}^{L(\sigma)}$ with disjoint indices. In addition, we write each $\sigma_i$ as $(\sigma_i(1), \sigma_i(2), \ldots, \sigma_i(\ell_i))$, where $\ell_i$ is the length of the cyclic permutation $\sigma_i$. Here, note that the decomposition is unique up to the order of $\{ \sigma_i \}$. Hence, $\qdet A$ is well defined. $\qdet A$ is called the quaternion determinant of $A$. 






{it Acknowledgments.}
This work was supported by JSPS KAKENHI Grant Numbers 
JP17K14206, JP16H06338, JP19H01793.
We thank Stuart Jenkinson, PhD, from Edanz (htttps://jp.edanz.com/ac) for editing a draft of this manuscript.


\end{document}